\pdfoutput=1
\documentclass[11pt]{article}
\usepackage{amssymb}
\usepackage[all,cmtip]{xy}
\usepackage[width=17cm,top = 3.5cm,bottom=3.5cm]{geometry}
\usepackage{amsthm,amsfonts,amssymb,amsmath,amscd} 
\usepackage{aliascnt} 
\usepackage{mathrsfs} 
\usepackage{xargs,ifthen} 

\usepackage{comment}
\usepackage{color}

\usepackage{hyperref}
\definecolor{tocolor}{rgb}{.1,.1,.1}
\definecolor{urlcolor}{rgb}{.2,.2,.6}
\definecolor{linkcolor}{rgb}{.1,.1,.5}
\definecolor{citecolor}{rgb}{.4,.2,.1}
\hypersetup{backref=true, colorlinks=true, urlcolor=urlcolor, linkcolor=linkcolor, citecolor=citecolor}

\newcommandx{\thdef}[2]{
	\newaliascnt{#1}{theorem}  
	\newtheorem{#1}[#1]{#2}
	\aliascntresetthe{#1}  
	\newtheorem*{#1*}{#2}
	\expandafter\newcommand\expandafter{\csname #1autorefname\endcsname}{#2}
}

\makeatletter
\newtheorem*{rep@theorem}{\rep@title}
\newcommand{\newreptheorem}[2]{%
\newenvironment{rep#1}[1]{%
 \def\rep@title{#2 \ref{##1}}%
 \begin{rep@theorem}}%
 {\end{rep@theorem}}}
\makeatother

\newtheorem{theorem}{Theorem}[section]
\newreptheorem{theorem}{Theorem}

\thdef{lemma}{Lemma}
\thdef{corollary}{Corollary}
\thdef{conjecture}{Conjecture}
\newreptheorem{conjecture}{Conjecture}
\thdef{proposition}{Proposition}

\theoremstyle{definition}
\thdef{definition}{Definition}
\thdef{notation}{Notation}

\theoremstyle{remark}
\thdef{remark}{Remark}

\theoremstyle{remark}
\thdef{ex}{Example}

\newenvironment{example}
{\begin{ex}}%
{\hfill $\blacksquare$\end{ex}}

\newcommand{\spc}[1]{\mathsf{#1}} 
\newcommand{\shf}[1]{\mathcal{#1}} 

\newcommand{\RR}{\mathbb{R}}
\newcommand{\CC}{\mathbb{C}}
\newcommand{\ZZ}{\mathbb{Z}}
\newcommand{\NN}{\mathbb{N}}

\newcommand{\rbrac}[1]{\left(#1\right)} 

\newcommandx{\fn}[2][2=]{#1\ifthenelse{\equal{#2}{}}{}{\!\rbrac{{#2}}}} 
\newcommandx{\id}[2][2=]{\fn{{\rm id}_{#1}}[#2]} 

\newcommand{\ext}[2][\bullet]{\spc{\Lambda}^{#1}{#2}} 
\newcommandx{\End}[2][1=]{\fn{\spc{End}_{#1}}[#2]} 
\newcommandx{\Hom}[2][1=]{\fn{\spc{Hom}_{#1}}[#2]} 
\newcommandx{\Aut}[2][1=]{\fn{\spc{Aut}_{#1}}[#2]} 
\newcommandx{\image}[1]{\fn{\spc{img}}[#1]} 
\renewcommandx{\ker}[1]{\fn{\spc{ker}}[#1]} 
\newcommandx{\rank}[1]{\fn{\mathrm{rank}}[#1]} 

\newcommandx{\ann}[1]{\fn{\spc{ann}}[#1]} 
\newcommand{\ad}[1][]{\mathrm{ad}_{#1}} 

\newcommandx{\hlgy}[3][1=\bullet,3=]{\spc{H}_{#1}^{#3}\!\rbrac{{#2}}} 
\newcommandx{\cohlgy}[3][1=\bullet,3=]{\spc{H}^{#1}_{#3}\!\rbrac{{#2}}} 
\newcommandx{\chow}[3][1=\bullet,3=]{\spc{A}^{#1}_{#3}\!\rbrac{{#2}}} 

\newcommandx{\Ext}[3][1=\bullet,3=]{\fn{\spc{Ext}^{#1}_{#3}}[{#2}]} 
\newcommandx{\Tor}[3][1=\bullet,3=]{\fn{\spc{Tor}^{#1}_{#3}}[{#2}]} 

\newcommandx{\Pic}[1]{\fn{\spc{Pic}}[{#1}]} 
\newcommandx{\chernalg}[2][1=\bullet]{\fn{\spc{Chern}^{#1}}[{#2}]} 
\newcommandx{\chern}[2][1=]{\fn{c_{#1}}[#2]} 
\newcommandx{\ch}[2][1=]{\fn{\mathrm{ch}_{#1}}[{#2}]} 

\newcommandx{\sKer}[2][1=]{ \fn{ \shf{K}er_{#1}}[{#2}] } 
\newcommandx{\sHom}[2][1=]{ \fn{ \shf{H}om_{#1}}[{#2}] } 
\newcommandx{\sEnd}[2][1=]{ \fn{ \shf{E}nd_{#1}}[{#2}] } 
\newcommandx{\sExt}[3][1=\bullet,3=]{\fn{\shf{E}xt^{#1}_{#3}}[{#2}]} 
\newcommandx{\sTor}[3][1=\bullet,3=]{\fn{\shf{T}or^{#1}_{#3}}[{#2}]} 

\newcommandx{\forms}[2][1=\bullet]{\Omega^{#1}_{#2}} 
\newcommandx{\can}[1][1=]{\omega_{#1}} 
\newcommandx{\acan}[1][1=]{\omega_{#1}^{-1}} 
\newcommandx{\tshf}[1]{\shf{T}_{#1}} 
\newcommandx{\mvect}[2][1=\bullet]{ \ext[#1]{\tshf{#2}} }
\newcommandx{\der}[2][1=\bullet]{\mathscr{X}^{#1}_{#2}} 
\newcommandx{\sJet}[3][1=,2=]{\shf{J}^{#1}_{#2}#3} 

\newcommandx{\tb}[2][1=]{\spc{T}_{\!#1}{#2}} 
\newcommandx{\ctb}[2][1=]{\spc{T}_{\!#1}^*{#2}} 
\newcommandx{\lie}[2][2=]{\fn{\mathscr{L}_{#1}}[#2]} 
\newcommandx{\hook}[2][2=]{\fn{i_{#1}}[#2]} 

\newcommand{\del}{\partial}

\newcommand{\sI}{\shf{I}} 

\newcommand{\sA}{\shf{A}}

\makeatletter
\newcommand{\thickbar}{\mathpalette\@thickbar}
\newcommand{\@thickbar}[2]{{#1\mkern1.5mu\vbox{
  \sbox\z@{$#1\mkern-1mu#2\mkern-1mu$}%
  \sbox\tw@{$#1\overline{#2}$}%
  \dimen@=\dimexpr\ht\tw@-\ht\z@-.6\p@\relax
  \hrule\@height.4\p@ 
  \vskip1\p@
  \hrule\@height.4\p@ 
  \vskip\dimen@
  \box\z@}\mkern1.5mu}
}
\makeatother

\numberwithin{equation}{section}

\newtheoremstyle{parag}
  {\topsep}   
  {\topsep}   
  {}  
  {}       
  {\bfseries} 
  {.}         
  { } 
  {}          
\theoremstyle{parag}

\makeatletter
\def\@cite#1#2{{\normalfont[{#1\if@tempswa , #2\fi}]}}
\makeatother

\newcommand{\tot}{\mathrm{tot}}

\newcommand{\OO}{\mathcal{O}}
\newcommand{\JJ}{\mathbb J}
\newcommand{\J}{\mathbb J}

\newcommand{\Def}{\mathbf{Def}}
\newcommand{\sP}{\mathcal{P}}

\newcommand{\LL}{\mathcal{L}}

\newcommand{\delbar}{\overline\partial}
\newcommand{\eps}{\epsilon}

\newcommand{\gl}{\mathfrak{gl}}
\newcommand{\At}{\mathrm{At}}

\newcommand{\N}{N}
\newcommand{\Res}{\mathrm{Res}}
\newcommand{\TT}{\mathbb{T}}
\newcommand{\T}{\mathbb{T}}
\newcommand{\kk}{\mathfrak{k}}
\newcommand{\ul}{\underline}
\newcommand{\gcm}{generalized complex manifold}
\newcommand{\gcs}{generalized complex structure}
\newcommand{\into}{\to}
\newcommand{\wrt}{with respect to}
\newcommand{\ev}{\mathrm{ev}}
\newcommand{\Ii}{\mathcal{I}}
\newcommand{\anc}{\mathrm{a}}
\newcommand{\ol}{\overline}

\newcommand{\mc}[1]{\text{$\mathcal{#1}$}}
\renewcommand{\SS}[1]{S^1\hspace{-.1em}#1}
\newcommand{\C}{\mathbb{C}}
\newcommand{\wt}{\widetilde}
\renewcommand{\Re}{\mathrm{Re}}
\renewcommand{\Im}{\mathrm{Im}}
\newcommand{\minus}{\!\setminus\!}

\newcommand{\sU}{\shf{U}}

\newcommand{\gcss}{generalized complex structures}

\newcommand{\gc}{generalized complex}
\newcommand{\gcms}{generalized complex manifolds}

\newcommand{\gcy}{generalized Calabi--Yau}
\newcommand{\gcys}{generalized Calabi--Yau structure}

\newcommand{\gf}{\text{$\varphi$}}
\newcommand{\nhood}{neighbourhood}
\newcommand{\e}{\varepsilon}
\newcommand{\SSS}{\textsection}
\newcommand{\R}{\text{${\mathbb R}$}}

\begin{document}

\title{\vspace{-4em} \huge Stable generalized complex structures}
\date{}

\author{
Gil R. Cavalcanti\thanks{Utrecht University; {\tt g.r.cavalcanti@uu.nl}}
\and
Marco Gualtieri \thanks{University of Toronto; {\tt mgualt@math.toronto.edu}}
}
\maketitle

\renewcommand{\abstractname}{\vspace{-\baselineskip}}

\abstract{\vspace{-2em}
A stable generalized complex structure is one that is generically symplectic but degenerates along a real codimension two submanifold, where it defines a generalized Calabi-Yau structure. We introduce a Lie algebroid which allows us to view such structures as symplectic forms. 
This allows us to construct new examples of stable structures, and also to define period maps for their deformations in which the background three-form flux is either fixed or not, proving the unobstructedness of both deformation problems.  We then use the same tools to establish local normal forms for the degeneracy locus and for Lagrangian branes. Applying our normal forms to the four-dimensional case, we prove that any compact stable generalized complex 4-manifold has a symplectic completion, in the sense that it can be modified near its degeneracy locus to produce a compact symplectic 4-manifold. 
}
\renewcommand\contentsname{\vspace{-\baselineskip}}
\tableofcontents

\section*{Introduction}

Generalized complex geometry~\cite{MR2013140, MR2811595} is a common generalization of complex and symplectic geometry in which the pointwise structure may be described as a symplectic subspace with transverse complex structure.  This symplectic distribution is controlled by a real Poisson structure, and so its rank may vary in any given example.  Four-dimensional generalized complex manifolds have been thoroughly investigated, the main focus being on structures which are generically symplectic and degenerate along a 2-dimensional submanifold, which then inherits a complex structure rendering it a Riemann surface of genus one.  In~\cite{MR2312048,MR2574746,Goto:2013vn,MR2958956,MR3177992}, many examples of generalized complex four-manifolds were found, the most interesting of which were on manifolds, such as $\CC P^{2}\#\CC P^{2}\#\CC P^{2}$, which admit neither symplectic nor complex structures.

In this paper we develop the main properties of \emph{stable} generalized complex structures, in which the structure is generically symplectic but degenerates along a real codimension 2 submanifold $D$, a direct generalization of the four-dimensional case described above. We show that $D$ inherits a generalized Calabi-Yau structure (of type 1) as well as a holomorphic structure on its normal bundle, and we prove that a tubular neighbourhood of $D$ is completely classified by this data, a result which was not available even in dimension four.  As an application of this result, we prove that any compact stable generalized complex 4-manifold has a symplectic completion, in the sense that it can be modified near $D$ to produce a compact symplectic 4-manifold.  We prove a similar normal form theorem for Lagrangian branes, half-dimensional submanifolds analogous to Lagrangians in symplectic geometry. This involves a generalization of the cotangent bundle construction in symplectic geometry, where for example, we associate a natural stable generalized 6-manifold to any co-oriented link $K\subset S^{3}$, by modifying the cotangent bundle of $S^{3}$ along $K$ in a certain way.  We also provide a construction of stable structures on torus fibrations, obtaining, for instance, a stable structure on $S^{1}\times S^{5}$.  We then move to deformation theory and define two period maps controlling deformations of stable generalized complex structures on compact manifolds $M$. The first describes deformations with fixed background 3-form and is a map to $H^{2}(M\backslash D,\CC)$, independently discovered by Goto~\cite{Goto:2015fk}. The second describes simultaneous deformations of the pair $(\JJ,H)$ comprised of a stable structure $\JJ$ integrable with respect to the 3-form $H$, and is a map to 
$
H^{2}(M\backslash D,\RR)\oplus H^{1}(D,\RR).
$
In both cases, we obtain the unobstructedness of the deformation problem. Finally, we describe a number of topological constraints which the pair $(M,D)$ must satisfy in order to admit a stable generalized complex structure. These exclude, for example, the possibility that $D$ is a positive divisor in a compact complex manifold $M$.

The main insight behind the above results is that a stable generalized complex structure is equivalent to a  \emph{complex log symplectic form}, a complex 2-form with a type of logarithmic singularity along the divisor $D$ and whose imaginary part defines an \emph{elliptic symplectic form}, which is a symplectic form but for a Lie algebroid which we introduce called the \emph{elliptic tangent bundle}. This approach, analogous to that taken in holomorphic log symplectic geometry~\cite{MR1953353} as well as in the recent development of real log symplectic geometry~\cite{MR3250302,MR3214314,Marcut-Osorno,MR3245143,Cavalcanti13}, 
 justifies the intuition that a stable generalized complex structure is a type of singular symplectic structure, and it allows us to apply symplectic techniques such as Moser interpolation.  For this reason, we carefully develop the theory of logarithmic and elliptic forms associated to smooth codimension 2 submanifolds.  In particular, we compute the Lie algebroid cohomology of the elliptic tangent bundle and give an explicit description of its cup product. \\
 
{\noindent \it Organization of the paper:}

\noindent In Section~\ref{sectcxdiv}, we introduce the notion of a complex divisor in the smooth category and its associated pair of Lie algebroids, the logarithmic tangent bundle (\SSS\ref{logderhamcxsect}) and the elliptic tangent bundle (\SSS\ref{elltn}), which facilitate working with codimension 2 logarithmic forms.  We describe the various residues of an elliptic form (\SSS\ref{elllogcohmg}), allowing an explicit description of the elliptic de Rham cohomology and its cup product.  We then compare (\SSS\ref{compelllog}) the logarithmic and elliptic de Rham complexes in the case that the elliptic residue vanishes, a condition  which is relevant since stable generalized complex structures satisfy it.  We end with an alternative geometric definition of the key Lie algebroids we use, identifying them with certain generalized Atiyah algebroids (\SSS\ref{atiyahsect}) and using this we obtain a key lemma (\SSS\ref{isodiffdiv}) that any family of complex divisors may be rectified in the smooth category. 

In Section~\ref{sec2} we introduce the main object of study: stable generalized complex structures. Sections \SSS\ref{canbun}--\ref{hollinbu} establish general results about the geometry of canonical line bundles and of generalized Calabi-Yau manifolds, and in~\SSS\ref{stablestructures} we define stable structures and determine in Theorem~\ref{gholstrnd} the inherited geometry of the anticanonical divisor, ending with a method for constructing new examples (\SSS\ref{exst}).

In Section~\ref{equivloggc}, we establish 
the equivalence between stable structures and complex log symplectic structures~(Theorem~\ref{equivsgcils}), or with co-oriented elliptic symplectic structures (\SSS\ref{elllogsymp}) if we consider only gauge equivalence classes of stable structures, and we use this to define two period maps, one for deformations in which $H$ is fixed (\SSS\ref{perfixH}) and one where it is not (\SSS\ref{pernofixH}).  These then impose certain topological constraints on $D$ and on its complement (\SSS\ref{topconst}).  In the remainder of this section we establish three main local normal form theorems: Theorem~\ref{darbouximaglog} is a Darboux theorem for the neighbourhood of a point in $D$, Theorem~\ref{linaboutd} classifies a tubular neighbourhood of $D$, and Theorem~\ref{canonicalsymp} is a Lagrangian brane neighbourhood theorem.  We prove our symplectic completion result for stable 4-manifolds (Theorem~\ref{theo:symplectization}) using the second of these. 

The period maps defined in Section~\ref{equivloggc} establish the unobstructedness of the deformation problems (with and without fixed 3-form flux), which suggests that the $L_{\infty}$ algebras controlling them are formal. In Section~\ref{sec4} we prove that the dgLa controlling the first problem is in fact formal, and we prove that the $L_{\infty}$ algebra controlling the second is quasi-isomorphic to a formal dgLa; we conjecture (Conjecture~\ref{conjlinfty}) that our quasi-isomorphism extends to a $L_{\infty}$ morphism. \\
 
{\noindent \it Acknowledgements:}

\noindent We thank R. Goto for the opportunity to visit Kyoto in 2013 and share some of the results of the present work, including Corollary~\ref{independentgoto} concerning unobstructedness for deformations with fixed 3-form, which was independently discovered by him and has recently appeared~\cite{Goto:2015fk}. 

G. C. was supported by a VIDI grant from NWO, the Dutch science foundation.
M. G. was supported by an NSERC Discovery Grant and acknowledges support from U.S. National Science Foundation grants DMS 1107452, 1107263, 1107367 ``RNMS: GEometric structures And Representation varieties'' (the GEAR Network).

\section{Complex divisors on smooth manifolds}\label{sectcxdiv}

\begin{definition}
Let $U$ be a smooth complex line bundle over the smooth 
$n$--manifold $M$, and let $s\in C^\infty(M,U)$ be a section transverse to the zero section.  We refer to the pair $D=(U,s)$ as a \emph{complex divisor}. 
\end{definition}

Our nomenclature is by analogy with the well-known correspondence between holomorphic line bundles with section and divisors on complex manifolds.  In our case, we regard the pair $(U,s)$ as the divisor, though we may abuse notation and use $D$ to refer to the smooth real codimension 2 submanifold given by the zero set of $s$.
Note that as $s$ vanishes transversely along $D$, it has a nonvanishing normal derivative which establishes an isomorphism between the real normal bundle $N$ of $D$ and the restriction of $U$ to $D$.
\begin{equation}\label{isonu}
d_\nu s : \xymatrix{N\ar[r]^{\cong} & U|_D}.
\end{equation} 
As a result, we obtain a complex structure on $N$, and so 
$D$ is co-oriented, defining an integral class in the second cohomology group $H^2(M,\ZZ)$ which coincides with the Chern class of $U$, just as in the holomorphic theory.

We now observe that by considering infinitesimal symmetries of a complex divisor we obtain several useful Lie algebroids.

\subsection{The logarithmic tangent bundle}\label{logderhamcxsect}

\begin{definition}
The \emph{logarithmic tangent bundle} associated to the complex divisor $D$ is the Lie algebroid, denoted by $T(-\log D)$, given by the locally free sheaf of complex vector fields on $M$ which preserve the ideal $\Ii_s \subset C^\infty_\CC(M)$  given by the image of the map
\begin{equation}
\xymatrix{C^\infty(M,U^{*})\ar[r]^-{s} & C^\infty_\CC(M)}.
\end{equation} 
The anchor map $a:T(-\log D)\to T_{\CC}M$ is defined by the inclusion of sheaves, and the bracket is inherited from the Lie bracket of vector fields.
\end{definition}

Away from the zero locus of $s$, the anchor is an isomorphism, and we give an explicit description of $T(-\log D)$ near a point on the zero locus as follows.  
By the transversality of $s$, we may choose a local trivialization in which $s$ is given by the complex coordinate function $w$, and let $x_3,\ldots, x_n$ be real functions forming a completion to a coordinate system, so that $T^*_\CC M$ is locally generated by $(dw,d\bar w, dx_3,\ldots, dx_n)$. Then the algebroid $T(-\log D)$ is locally freely generated over $C^\infty_\CC$ as follows.
\begin{equation}
T(-\log D) = \left< w\del_w, \del_{\bar w}, \del_{x_3},\ldots, \del_{x_n}\right>.	
\end{equation}
In the case of a nonsingular divisor $\mathcal{D}$ on a complex manifold, with holomorphic log tangent bundle $\mathcal{T}(-\log \mathcal{D})$, the above complex Lie algebroid coincides with the natural Lie algebroid structure on $\mathcal{T}(-\log\mathcal{D})\oplus T^{0,1}M$, the smooth Lie algebroid underlying the holomorphic one (see~\cite{Laurent-Gengoux01012008}).


Also by analogy with the holomorphic case, we refer to the de Rham complex of the algebroid $T(-\log D)$ as the logarithmic de Rham complex of the complex divisor $D=(U,s)$. 
Similarly, we use the notation $\Omega^k(\log D)$ for the sheaf of sections of the bundle of logarithmic $k$-forms; in particular
\begin{equation}
\Omega^k(M,\log D) = C^\infty(M, \wedge^k(T(-\log D))^*).
\end{equation}
In the above coordinates, a general logarithmic form may be written as
\begin{equation}\label{expan}
\rho = d\log w\wedge \alpha + \beta,
\end{equation}
for uniquely determined $\alpha,\beta$ in the subalgebra generated by $(d\bar w,dx_3,\ldots, dx_n)$.

There are two important morphisms comparing the logarithmic de Rham complex with the usual de Rham complex. The first derives from the fact that the anchor map 
\begin{equation}
\xymatrix{T(-\log D)\ar[r]^-{a} & TM}
\end{equation}   
is an isomorphism over the divisor complement, i.e. the nonvanishing locus of $s$.  We obtain a pullback along the inclusion $i$ of the complement:
\begin{equation}
i^*: \Omega^k(M,\log D)\to \Omega^k(M\backslash D,\CC).
\end{equation}
Just as in the holomorphic theory~\cite{MR0199194}, $i$ induces an isomorphism on cohomology groups.
\begin{theorem}\label{logdcoh}
 The inclusion of the divisor complement induces an isomorphism between the logarithmic cohomology of $(M,D)$ and the complex de Rham cohomology of the complement:
\begin{equation}
H^k(i^{*}):\xymatrix{H^k(M,\log D)\ar[r]^-{\cong}& H^k(M\backslash D, \CC)}.
\end{equation}
\end{theorem}
\begin{proof}
As is done in the holomorphic case, we view the logarithmic cohomology as the hypercohomology of the sheaf-theoretic logarithmic de Rham complex, and similarly for the de Rham cohomology of the divisor complement.  The pullback $i^*$ is a morphism of complexes of sheaves, and we verify that it is a quasi-isomorphism, in the sense that it induces an isomorphism on local cohomology sheaves. By the usual argument using the hypercohomology spectral sequence, the quasi-isomorphism induces an isomorphism on hypercohomology, yielding the result.

Away from the divisor, $i^*$ is an isomorphism, so to prove it is a quasi-isomorphism we compute local cohomology in a small ball surrounding a point in $D$.  With local coordinates chosen as above, the logarithmic cohomology is 1-dimensional in degrees $0$ and $1$ and zero otherwise, generated by a constant and by $d\log w$, respectively.  Applying $i^*$ takes these to generators for the cohomology of the divisor complement in the ball, which is homotopic to the circle.
\end{proof}

The second comparison map between usual and logarithmic forms is given by the residue map, which takes a logarithmic $k$--form to a usual $(k-1)$--form along $D$. The residue of our general form~\eqref{expan} is given by
\begin{equation}\label{defres}
\Res(d\log w \wedge\alpha + \beta) = j^*\alpha,
\end{equation}
where $j:D\hookrightarrow M$ is the inclusion. Note that, in contrast to the holomorphic theory, the vanishing of the residue does not guarantee that a form is smooth, i.e. a member of the subcomplex of usual differential forms: it may have a nonvanishing component in the ideal generated by $\bar w d\log w$ and $d\log w\wedge d\bar w$.  

Just as in the holomorphic theory, however, the residue defined above~\eqref{defres} is a cochain morphism and so induces a map of de Rham cohomology groups 
\begin{equation}
\Res_{*}: H^{k}(M,\log D)\to H^{k-1}(D,\CC).
\end{equation}
The topological description of this map is well-known in the study of residues in the holomorphic category, see e.g.~\cite{MR1194180}.  We leave the translation of the usual proof to our situation to the reader.

\begin{proposition}\label{poincaregysin}
The residue map coincides with the Poincar\'e--Leray residue map in the Thom--Gysin sequence for inclusion $i$ of the complement $M\backslash D$ into $M$. That is, we have the long exact sequence of cohomology groups with complex coefficients:
\begin{equation}
\xymatrix@C=1em{
\cdots\ar[r] & H^{k}(M)\ar[r]^-{i^{*}} & H^{k}(M\backslash D)\ar[r]^-{R} & H^{k-1}(D)\ar[r]^-{j_{*}} & H^{k+1}(M)\ar[r] & \cdots
}
\end{equation}
where $R=2\pi i\cdot \Res_{*}$ and $j_{*}$ is the pushforward on cohomology associated to the co-oriented inclusion $j:D\hookrightarrow M$.
\end{proposition}

\subsection{The elliptic tangent bundle}\label{elltn}

Any complex divisor $D=(U,s)$ determines a complex conjugate divisor $\ol D=(\ol U,\ol s)$, with the same zero locus. As described above, this divisor gives rise to a Lie algebroid $T(-\log\ol D)$, which is, appropriately, complex conjugate to $T(-\log D)$.  In the same coordinates chosen above, we have
\begin{equation}
T(-\log \ol D) = \left< \del_w, \bar w\del_{\bar w}, \del_{x_3},\ldots, \del_{x_n}\right>.
\end{equation}
The anchor maps $\anc,\ol\anc$ of the algebroids $T(-\log D), T(-\log\ol D)$ are transverse as bundle maps to $T_\CC M$, and so there is a well-defined fibre product Lie algebroid, which is sent to itself by complex conjugation.
\begin{definition}
The elliptic tangent bundle associated to the complex divisor $D=(U,s)$ is the real Lie algebroid $T(-\log |D|)$ whose complexification is the fiber product of the logarithmic tangent bundle of $D$ with its complex conjugate.
\end{definition}

A convenient way to describe $T(-\log |D|)$ is as the real infinitesimal symmetries of the tensor product $(U\otimes \ol U, s\otimes \ol s)$ of $D$ with $\ol D$.  While $s\otimes\ol s$ is not transverse to the zero section, it defines an ideal of functions $\Ii_{s\ol s}\subset C^\infty_\RR(M)$ as before; the elliptic algebroid is identified with the subsheaf of vector fields preserving this ideal. In the coordinate system chosen above, it is given by the vector fields preserving the quadratic defining function $w\ol w$: in polar coordinates $w = re^{i\theta}$, we have 
\begin{equation}\label{basisell}
T(-\log |D|) = \left< r\del_r, \del_\theta,\del_{x_3},\ldots, \del_{x_n}\right>.
\end{equation}

In general, we may define the elliptic tangent bundle as follows.
\begin{definition}
Let $(R,q)$ be an \emph{elliptic divisor}, consisting of a real line bundle $R$ over $M$ with a smooth section $q$  whose zero set $D$ is a smooth codimension 2 critical submanifold along which the normal Hessian is positive-definite. The sheaf of vector fields preserving the ideal $\Ii_q = q(C^\infty(M,R^*))$ is called the \emph{elliptic tangent bundle} associated to $(R,q)$. 
\end{definition}
Note that the real line bundle $R$ is oriented by $q$ since it is a trivialization away from a codimension $2$ submanifold.  Also, the normal Hessian referred to above is the leading term of the Taylor expansion of $q$:
\begin{equation}\label{hessian}
\mathrm{Hess}(q)\in C^\infty(D, S^2N^*\otimes R).
\end{equation}
Using the Morse--Bott lemma and the factorization $x^2 + y^2 = (x+iy)(x-iy) = w\bar w$, one sees that an elliptic divisor $(R,q)$  may be expressed as $(U\otimes \ol U, s\otimes\ol s)$ for a complex divisor $(U,s)$ if and only if its zero set is co-orientable; then $(U,s)$ is uniquely determined up to diffeomorphism by the choice of a co-orientation.

\subsection{Elliptic logarithmic cohomology}\label{elllogcohmg}

We now describe the Lie algebroid de Rham complex of the elliptic tangent bundle $T(-\log|D|)$ associated to the elliptic divisor $|D|=(R,q)$. We use $D$ to denote the zero locus of $q$. In the main case of interest, $(R,q)$ shall be the elliptic divisor obtained from a complex divisor, i.e. $q = s\otimes \ol s$.  
We use $\Omega^{k}(\log |D|)$ to denote the smooth sections of $\wedge^{k}((T(-\log|D|))^{*})$ and call these \emph{elliptic forms}.
From~\eqref{basisell} we see that locally, a general elliptic form may be written
\begin{equation}\label{ellform}
\rho = d\log r \wedge d\theta \wedge \rho_0 + d\log r \wedge \rho_1 + d\theta \wedge \rho_2 + \rho_3,
\end{equation}
with $\rho_i$ smooth forms. Just as for logarithmic forms, there is a well-defined pullback $i^*$ from logarithmic forms to the de Rham complex of the complement $M\backslash D$ (with real coefficients).  In contrast to the previous case, however, since there are two independent singular generators $d\theta, d\log r$ for the elliptic forms, we obtain several residue forms, which we organize as follows.

The restriction of $T(-\log|D|)$ to $D$ defines an algebroid of infinitesimal symmetries of the normal bundle $N$ of $D$; it forms an exact sequence
\begin{equation}\label{restd}
0\to \ul\RR\oplus \kk\to T(-\log |D|) |_{D}\to TD\to 0,
\end{equation}
where $\ul \RR$ is the trivial bundle generated by the Euler vector field $E$, and $\kk\cong \wedge^{2}N^{*}\otimes R$ is the adjoint bundle of infinitesimal rotations preserving the Hessian; when $N$ is orientable, a choice of orientation on $N$ distiguishes a global trivialization $I$ of $\kk$, a complex structure which generates the $S^{1}$ action on $N$. So, the logarithmic forms, restricted along $D$, define algebroid forms for the above algebroid. Therefore we may define an \emph{elliptic} residue $\Res_q$ and, if this vanishes, a \emph{complex} residue $\Res_c$, by the canonical projection maps given by dualizing \eqref{restd}:
\begin{equation}\label{proja}
\begin{gathered}
\xymatrix@R=1em{
      \Omega^k_{0}(\log |D|)\ar[r] & \Omega^k(\log |D|) \ar[r]^-{\Res_q} &  \Omega^{k-2}(D, \kk^*)\\
     &  \Omega^k_{0}(\log |D|)\ar[r]^-{\Res_{c}} &   \Omega^{k-1}(D,\ul\RR \oplus \kk^*)   }
\end{gathered}
\end{equation}
We denote the kernel of $\Res_q$ by $\Omega^\bullet_{0}(\log |D|)$; it is a natural subcomplex of the elliptic de Rham complex.
The orientation bundle $\kk$ is flat, and the elliptic residue is a degree $-2$ cochain map to $\Omega^{\bullet}(D,\kk^{*})$.  Also, if $N$ is oriented, $\kk$ is trivialized and we may view $\Res_{c}$ as a form with complex coefficients.  Applied to the form~\eqref{ellform} and using $\del_{\theta}$ to orient $N$, we define $\Res_q(\rho) = j^*\rho_0$, and if this vanishes, we define 
\begin{equation}\label{ellcxres}
\Res_{c}(\rho) = j^*(\rho_{1}-i\rho_2),
\end{equation}
where $j$ denotes the inclusion $D\hookrightarrow M$.
This complex residue is compatible with the logarithmic residue defined in~\eqref{defres}: any $\rho\in\Omega^{\bullet}(\log D)$ may be pulled back to a complexified elliptic form, whose real and imaginary parts satisfy
\begin{equation}\label{identityres}
\Res_{c}(\Re(\rho)) = i\Res_{c}(\Im(\rho)) = \Res(\rho).
\end{equation}

The radial components of the residues $\Res_{q}, \Res_{c}$ play a special role. Quotienting the sequence~\eqref{restd} by the Euler vector field $E$, we obtain the Atiyah algebroid of the  circle bundle $\SS{N}$ associated to the rank 2 bundle $N$, an extension as below:
\begin{equation}\label{princ}
0\to \kk \to \At(\SS{N})\to TD\to 0.
\end{equation}
Because $T(-\log |D|)|_D$ is an extension of $\At(\SS{N})$ by a trivial bundle, the elliptic residue factors through a \emph{radial} residue map 
\begin{equation}
\Res_r:\Omega^k(\log|D|) \to C^\infty(D,\wedge^{k-1}\At(\SS{N})^*),
\end{equation}
\begin{definition}
The radial residue of the form~\eqref{ellform} is given by
\begin{equation}
\Res_r(\rho) = (d\theta\wedge\rho_0  + \rho_1)|_{D},
\end{equation}
well-defined as an algebroid form for the Atiyah algebroid of the principal circle bundle associated to the normal bundle of $D$.
\end{definition}

The radial residue may be viewed as an invariant form on the  $S^1$--bundle $\SS{N}$ associated to $N$ (i.e., the exceptional divisor of the real-oriented blow-up of $M$ along $D$). Also, the contraction $i_{\del_\theta}\Res_r(\rho)$ coincides with the elliptic residue. When this vanishes, $\Res_r(\rho)$ coincides with the real part of the complex residue.
We now compute the elliptic de Rham cohomology in terms of the de Rham cohomology of the complement and of the normal circle bundle. 

\begin{theorem}\label{thm:lie algebroid cohomology}
Let $|D|$ be an elliptic divisor.
Then the restriction of forms to the divisor complement, together with the radial residue map, defines an isomorphism 
\begin{equation}\label{decomposelllog}
H^k(\log|D|) = H^k(M\backslash D,\RR) \oplus H^{k-1}(\SS{N},\RR),
\end{equation}
where $\SS{N}$ is the $S^1$--bundle associated to the normal bundle of $D$.
\end{theorem}
\begin{proof}
Following the same strategy as in Theorem~\ref{logdcoh}, we show that restriction to the complement, together with the radial residue, define a quasi-isomorphism of complexes of sheaves
\begin{equation}
(i^*,\Res_r):\Omega^k(\log |D|) \into \Omega^k(M\backslash D,\RR) \oplus j_*C^\infty(\wedge^{k-1} A^*),
\end{equation}
where $j_*C^\infty(\wedge^k A^*)$ is the sheaf of algebroid forms for the Atiyah algebroid~\eqref{princ}, pushed forward to a sheaf on $M$ supported along $D$.  Recall that $A$-forms may be viewed as invariant forms on $\SS{N}$, so that this complex computes the required cohomology of $\SS{N}$. 

For open sets disjoint from $D$, the map is an isomorphism of complexes. For a sufficiently small ball centered on a point in $D$, the cohomology of the complement of $D$ is homotopic to the circle and the bundle $\SS{N}$ is locally trivial, so the right hand side has local cohomology given by
\begin{equation}
H^k(S^1,\RR) \oplus H^{k-1}(S^1,\RR)
\end{equation}

We now compute the local cohomology of the left hand side. In a coordinate chart $U$ as above, we choose the local splitting $T(-\log |D|) = E\oplus F$, with $E$ generated by $\del_{x_3},\ldots, \del_{x_n}$ and $F$ generated by $r\del_r, \del_\theta$. 
We then doubly grade the complex of logarithmic forms:
\begin{equation}
(\Omega^k(U,\log |D|), d)= (\oplus_{i+j=k} C^\infty(\wedge^i E^*\otimes \wedge^j F^*), d_E + d_F).
\end{equation}
We compute cohomology using the spectral sequence of a double complex. The $d_E$ cohomology is easily computed by the Poincar\'e lemma: it is generated by a constant in degree $(0,0)$, the forms $d\log r$ and $d\theta$ in degree $(0,1)$, and $d\log r\wedge d\theta$ in degree $(0,2)$.  The next differential is induced by $d_F$, and vanishes since the aforementioned generators are closed. Further differentials vanish, hence we obtain
\begin{equation}\label{localcompelld}
H^\bullet(U,\log |D|) = \wedge^\bullet (\RR [d\log r] \oplus \RR [d\theta]). 
\end{equation}
Finally, observe that $i^*1, i^*[d\theta]$ generate the local cohomology of the complement, whereas $\Res_r[d\log r] = 1$ and $\Res_r[d\log r \wedge d\theta] = [d\theta]$ generate the local cohomology of the normal $S^1$--bundle, establishing the quasi-isomorphism.
\end{proof}

Note that since the radial residue has cohomology class in 
$H^{k-1}(\SS{N})$, in the case that $N$ is oriented, we may compose with the pushforward along $\pi:\SS{N}\to D$ to obtain the elliptic residue in $H^{k-2}(D)$. We may also use the Gysin sequence
\begin{equation}
\xymatrix{  H^{k}(D)\ar[r]^-{\pi^*}& H^k(\SS{N})\ar[r]^{\pi_{*}}& H^{k-1}(D)\ar[r]^-{c_{1}} & H^{k+1}(D)
}
\end{equation}
to simplify the computation of elliptic de Rham cohomology: in the case that $N$ is trivial, for example, we conclude from Theorem~\ref{thm:lie algebroid cohomology} that  
\begin{equation}
H^k(\log|D|) = H^k(M\backslash D,\RR) \oplus H^{k-1}(D,\RR)\oplus H^{k-2}(D,\RR),
\end{equation}
where the first component corresponds to the restriction to the complement, and the second and third components are the radial residue, consisting of the real part of the complex residue and the elliptic residue, respectively.

Since the circle bundle $\SS{N}$ is homotopic to the intersection of a tubular neighbourhood of $D$ with $M\backslash D$, we have a canonical restriction homomorphism 
\begin{equation}\label{resttos1}
r:H^{k}(M\backslash D,\RR)\to H^{k}(S^{1}N,\RR),
\end{equation}
which is important for describing the product on $H^{\bullet}(\log |D|)$, as follows.
\begin{theorem}\label{cupprod}
The cup product on $H^{\bullet}(\log |D|)$ inherited from the differential graded algebra structure on the elliptic de Rham complex decomposes according to the splitting~\eqref{decomposelllog} as a sum of the usual cup product on $H^{\bullet}(M\backslash D,\RR)$ and the composition
\begin{equation}
\xymatrix@C=2em{
H^{k}(\SS{N})\times H^{l}(M\backslash D) \ar[r]^-{1\times r}&
H^{k}(\SS{N})\times H^{l}(\SS{N})\ar[r]^-{\cup}&
H^{k+l}(\SS{N})}.
\end{equation}
\end{theorem}
\begin{proof}
Choose a tubular neighbourhood and an identification with a neighbourhood $\sU$ of $D$ in the normal bundle $N$.  Choose a metric on $N$ so that we have a well-defined radial coordinate $r$ in $\sU$ and we may write any elliptic form as $\rho = d\log r \wedge \alpha + \beta$ with $i_{r\del_{r}}\alpha =i_{r\del_{r}}\beta =0$.
Denote the radial residue $\Res_{r}(\rho) = \alpha|_{r=0}$ by $\alpha_{0}$.  
We may view $\alpha_{0}$ as a $\CC^{*}$-invariant elliptic form on $\tot(N)$, and in this case the form $d\log r \wedge\alpha_{0}$ has the same residue as $\rho$, and the restriction 
\begin{equation}
(\rho - d\log r \wedge \alpha_{0})|_{D} = \beta|_{r=0}
\end{equation}
is a well-defined section $\beta_{0}$ of $\wedge^{\bullet}(\At(\SS{N})^{*})$ which we may again view as an invariant form on $\tot(N)$.  Therefore
\begin{equation}\label{removeprincpart}
\rho - d\log r \wedge \alpha_{0} - \beta_{0}
\end{equation}
vanishes as a logarithmic form along $D$.  Now observe that if $d\rho=0$, then $d\alpha_{0}=0$ as well, and further $d\beta_{0}=0$.  So, the form~\eqref{removeprincpart} is a closed elliptic form which vanishes along $D$.  By Lemma~\ref{retractlog}, such a form is exact in $\sU$, i.e.,
\begin{equation}
[\rho] = [d\log r \wedge \alpha_{0} + \beta_{0}].
\end{equation}
If we now restrict to the complement of $D$ in $\sU$, we see that $d\log r \wedge \alpha_{0}$ is exact, with primitive $(\log r) \alpha_{0}$, and so $[\rho] = [\beta_{0}]$ on $\sU\backslash D$.  Summarizing, we have 
\begin{equation}
r i^{*}[\rho] = [\beta_{0}].
\end{equation}

 If we apply this observation to the product of forms $\rho=d\log r\wedge \alpha + \beta$ and $\rho'=d\log r\wedge \alpha' + \beta'$ of degree $k$ and $l$ respectively, we obtain equalities
\begin{equation}
\begin{aligned}
\rho\wedge\rho' &= d\log r \wedge (\alpha\wedge\beta' + (-1)^{k} \beta \wedge\alpha') + \beta\wedge\beta'\\
\Res_{r}[\rho\wedge\rho'] &= \Res_{r}[\rho]\cup r(i^{*}[\rho'])
 + (-1)^{k} r(i^{*}[\rho])\cup \Res_{r}[\rho'],
\end{aligned}
\end{equation}
yielding the required expression for the elliptic cup product.
\end{proof}

\begin{lemma}\label{retractlog}
If a closed elliptic form vanishes along $D$, then it is trivial in the elliptic de Rham cohomology of a tubular neighbourhood of $D$.
\end{lemma}
\begin{proof}
Choose a tubular neighbourhood $U$ of $D$, which we identify with an $S^{1}$--invariant neighbourhood of the zero section in the total space of the normal bundle of $D$.
By Theorem~\ref{thm:lie algebroid cohomology}, the class defined by the closed elliptic form $\rho$ has two components: one in $H^{k}(U\backslash D,\RR)$ obtained by restricting to $U\backslash D$, and another in $H^{k-1}(S^{1}N,\RR)$, defined by the class of $\Res_{r}(\rho)$.  But this residue vanishes, since $\rho$ vanishes along $D$.  So it remains to show that $\rho$ is exact on $U\backslash D$.

We prove this by showing first that $\rho$ is cohomologous to a smooth $k$-form $\ol \rho$ on $M$, and that this $k$-form is zero when pulled back to $D$.  Since $M$ retracts onto $D$, this implies that $[\ol\rho]$ vanishes in $H^{k}(U,\RR)$, and hence in $H^{k}(U\backslash D,\RR)$, as required.

We construct $\ol\rho$ by averaging the family of forms $\rho_{t} = \phi_{t}^{*}\rho,\ t\in S^{1}$, where $\phi_{t}:U\to U$  is the circle action on the tubular neighbourhood. The rotational vector field generating the $S^{1}$ action is a section of $T(-\log|D|)$, and so the averaging process acts trivially on $H^{k}(\log |D|)$.  If we locally trivialize the bundle and choose polar coordinates $(r,\theta)$ on the fibre, we may write 
\begin{equation}
\ol\rho = d\log r \wedge d\theta \wedge \rho_0 + d\log r \wedge \rho_1 + d\theta \wedge \rho_2 + \rho_3,
\end{equation}
where $\rho_i$ are smooth forms independent of $\theta$.  Since the circle action fixes $D$, we also have that $\ol\rho$, and hence each $\rho_{i}$, vanishes along $D$. But this implies that each $\rho_{i}$ is divisible by $r^{2}$ in the smooth forms, and therefore that $\ol\rho$ is itself a smooth form. Finally, its pullback to $D$ is the pullback of $\rho_{3}$ to $D$, which vanishes since $\rho_{3}$ vanishes along $D$.
\end{proof}

\subsection{Comparison of elliptic and logarithmic forms}\label{compelllog}
Let $T_{\CC}(-\log |D|)$ be the complexification of $T(-\log |D|)$ and let $\iota$ be the algebroid morphism from $T_{\CC}(-\log |D|)$ to $T(-\log D)$ coming from the definition of $T(-\log |D|)$ as the fibre product of the logarithmic tangent bundle with its complex conjugate. This defines a pullback morphism from logarithmic forms to the complexified elliptic forms:
\begin{equation}
\iota^{*}:\Omega^{\bullet}(\log D)\to \Omega^{\bullet}_{\CC}(\log|D|).
\end{equation}

\begin{proposition}\label{isomzerores}
The composition $\Im^{*}$ of $\iota^{*}$ with the projection to the imaginary part is a surjection from the log forms to the elliptic forms with zero residue, with kernel given by the real smooth forms, defining an exact sequence of complexes
\begin{equation}\label{complogell}
\xymatrix{
0\ar[r] & \Omega^{\bullet}(M,\RR)\ar[r] & \Omega^{\bullet}(\log D)\ar[r]^-{\Im^{*}} & \Omega^{\bullet}_{0}(\log |D|)\ar[r] & 0}
\end{equation}
\end{proposition}
\begin{proof}
This is an exact sequence on the level of complexes of sheaves, so we may verify the statement locally in the coordinate system used above. A general log form $\rho = d\log w \wedge\alpha + \beta$ as in~\eqref{expan} has $\Res_{q}(i^{*}\rho) = 0$ since it has no $d\log \ol w$ component, so the same is true of its real and imaginary parts, showing that $\Im^{*}$ has the claimed codomain. We now show surjectivity of $\Im^{*}$: write a general real elliptic form as 
\begin{equation}
\rho =  d\log w \wedge d\log \ol w \wedge i\alpha + d\log w\wedge\beta  + d\log\ol w\wedge\ol\beta + \gamma,
\end{equation}
where $\alpha,\beta,\gamma$ are in the subalgebra generated by the remaining generators $dx_{3},\ldots,dx_{n}$, and $\alpha,\gamma$ are real. 
This form has vanishing elliptic residue if and only if  $\alpha$ vanishes along $D$, meaning $\alpha = w\alpha' + \ol w\ol \alpha'$ for smooth $\alpha'$.  Then we have 
\begin{equation}
\rho = d\log w \wedge (\beta + d\ol w\wedge i\ol\alpha') + d\log\ol w\wedge ( \ol\beta - dw\wedge i\alpha') + \gamma,
\end{equation}
but this is the imaginary part of $2i d\log w \wedge(\beta + d\ol w \wedge i\ol\alpha') + i\gamma$, a form in $\Omega^{\bullet}(\log D)$, as required.  

Injectivity at the first place follows from the fact that smooth forms form a subsheaf of the logarithmic forms. 	To show exactness at the middle place, suppose that $\Im^{*}\rho = 0$, i.e., 
\begin{equation}
0=d\log w \wedge \alpha - d\log\ol w\wedge\ol \alpha  +  \beta - \ol\beta,
\end{equation}
where $\alpha,\beta$ are in the subalgebra generated by $(d\bar w,dx_3,\ldots, dx_n)$.  If we write $\alpha = d\ol w\wedge \alpha_{1} + \alpha_{2}$ and $\beta = d\ol w \wedge \beta_{1} + \beta_{2}$, with $\alpha_{i}, \beta_{i}$ in the subalgebra generated by $(dx_{3},\ldots, dx_{n})$, then we obtain 
\begin{equation}
\begin{split}
0 &= d\log w\wedge d\log\ol w\wedge (\ol w \alpha_{1} + w\ol\alpha_{1}) \\
 &\quad +  d\log w \wedge (\alpha_{2} - w\ol\beta_{1}) + d\log\ol w\wedge (\ol\alpha_{2} -\ol w \beta_{1})\\
 &\quad\quad  +  (\beta_{2} - \ol\beta_{2}),
\end{split}
\end{equation}
and in this form each summand vanishes independently. This implies that each of $\alpha_{1}$ and $\alpha_{2}$, and therefore $\alpha$, is divisible by $w$, which proves that $\rho = d\log w\wedge\alpha + \beta$ is smooth.
\end{proof}

\begin{theorem}\label{ellzerores}
Let $|D|$ be an elliptic divisor.  The same morphism $(i^{*}, \Res_{r})$ from Theorem~\ref{thm:lie algebroid cohomology}, when applied to elliptic forms with vanishing elliptic residue,
 defines an isomorphism  
\begin{equation}\label{decompelliptzerores}
H^k_{0}(\log|D|) = H^k(M\backslash D,\RR) \oplus H^{k-1}(D,\RR),
\end{equation}
where $H^{k}_{0}(\log|D|)$ denotes the cohomology of the complex $\Omega^{\bullet}_{0}(\log |D|)$ of forms with vanishing elliptic residue.
\end{theorem}
\begin{proof}
The same strategy employed in the proof of Theorem~\ref{thm:lie algebroid cohomology} may be used here, with two differences: first, that when the elliptic residue vanishes, $\Res_{r}$ maps naturally to $\Omega^{k-1}(D,\RR)$ (without the need of a co-orientation), and second, that the local computation~\eqref{localcompelld} is modified by the absence of the generator $d\log r\wedge d\theta$. The remainder of the proof remains unchanged.
\end{proof}

We now combine our knowledge of the elliptic and logarithmic cohomology groups from Theorems~\ref{logdcoh} and~\ref{ellzerores} to give a purely topological description of the long exact sequence of cohomology groups resulting from the sequence~\eqref{complogell}. 

\begin{theorem}\label{sequenceselllog}
The long exact sequence deriving from~\eqref{complogell} may be written
\begin{equation}\label{toplogseq}
\xymatrix@C=1.1em{
\cdots\ar[r]& H^{k}(M,\RR)\ar[r]& H^k(M\backslash D,\CC)\ar[r]  &
H^{k}_{0}(\log |D|)
 \ar[r]&\cdots 
},
\end{equation}
and splits according to decomposition~\eqref{decompelliptzerores} as a sum of the trivial sequence 
\begin{equation}\label{impart}
\xymatrix@C=1.1em{
\cdots\ar[r]& 0\ar[r]& \Im(H^k(M\backslash D,\CC))\ar[r]^-{=}  &
 H^{k}(M\backslash D,\RR)
 \ar[r]&\cdots 
}
\end{equation}
and the Thom--Gysin sequence associated to the inclusion $i$ of $M\backslash D$ into $M$
\begin{equation}\label{relpart}
\xymatrix@C=1em{
\cdots\ar[r]& H^{k}(M,\RR)\ar[r]& \Re(H^k(M\minus D,\CC))\ar[r]  &
H^{k-1}(D,\RR)
 \ar[r]& \cdots
}.
\end{equation}

\end{theorem}
\begin{proof}
We first use Proposition~\ref{isomzerores} and Theorem~\ref{logdcoh} to write the sequence~\eqref{toplogseq}.  The  homomorphism from $H^{k}(M,\RR)$ to the cohomology of the complement is simply $i^{*}$, which maps only to the real part of $H^{k}(M\backslash D,\CC)$. We now compute the induced map of $\Im^{*}$ from sequence~\eqref{complogell}.  Since the pullback to the complement commutes with taking imaginary part, the component of $\Im^{*}$ mapping to $H^{k}(M\backslash D,\RR)$ is simply the projection to the imaginary part, as claimed in~\eqref{impart}.  The component of $\Im^{*}$ mapping to $H^{k-1}(D,\RR)$ is induced by the map $\Res_{r}\circ\Im^{*}$, and since $\Res_{r}$ coincides with the real part of the complex residue and we have identity~\eqref{identityres}, we have 
\begin{equation}
\Res_{r}\circ \Im^{*} = \Re\circ \Res_{c}\circ \Im^{*} = -i \Re\circ \Res = -(2\pi)^{-1}\Re\circ R,
\end{equation}
as operators on $H^{k}(M,\log D)$, where $R$ is the map from the Gysin sequence in Proposition~\ref{poincaregysin}, verifying~\eqref{relpart}.
\end{proof}

\subsection{Atiyah algebroids}\label{atiyahsect}

In this supplementary section, we describe a useful geometric construction of the algebroids $T(-\log D)$ and $T(-\log |D|)$ which we associated to a complex divisor $(U,s)$  in the previous sections.

The real derivations of $U$ preserving its complex structure are the sections of the Atiyah algebroid of $U$, a real Lie algebroid forming an extension of the form 
\begin{equation}\label{seqreati}
0\to \gl(U)\to \At(U)\to TM\to 0,
\end{equation}
where $\gl(U)$ is a trivial bundle generated over $\RR$ by the identity endomorphism and the complex structure on $U$. 
If $s$ is a section of $U$, then there is a vector bundle map
\begin{equation}
\ev_{s}:\At(U)\to U
\end{equation}
which evaluates a derivation on the given section $s$.  Transversality of $s$ is equivalent to the fact that $\ev_{s}$ 
is a surjective map of real vector bundles.  Since $U$ has real rank $2$, the kernel of $\ev_{s}$ is a corank 2 subalgebroid $\At(U,s)\subset \At(U)$.  In terms of vector fields, this subalgebroid consists of all $\CC^{*}$--invariant real vector fields on the total space of $U$ which are tangent to the image of the section $s$.

\begin{proposition}\label{subell}
The subalgebroid $\At(U,s)\subset \At(U)$ of symmetries of $U$ preserving $s$ is identified with $T(-\log |D|)$ by the anchor map. \end{proposition}
\begin{proof}
Choose local coordinates as above near a point on $D$, and let $\tau$ be a local trivialization for $U^{*}$, so that it defines a complex fibre coordinate on $U$.  The image of the section is cut out by the real and imaginary parts of the equation $\tau = w$. If we write $w = re^{i\theta}$ and $\tau=te^{i\phi}$, the invariant vector fields preserving this submanifold are then generated by  
\begin{equation}
(r\del_{r} + t\del_{t}, \del_{\theta} + \del_{\phi}, \del_{x_{3}},\ldots, \del_{x_{n}}).
\end{equation}
Comparing with~\eqref{basisell}, we see that the anchor projection to $TM$ factors through a canonical isomorphism to $T(-\log |D|)$.
\end{proof}

\begin{remark}
Proposition~\ref{subell} provides an alternative definition of the elliptic tangent bundle, one with the advantage that it is described simply as the tangent bundle to a $\CC^{*}$--invariant foliation on $\tot^{*}(U)$, the total space of $U$ with the zero section deleted. In the local calculation above, the foliation is given by the level sets of $w/t$, the function on $\tot^{*}(U)$ determined by $s$.  Note also that duality defines an isomorphism between $\tot^{*}(U)$ and $\tot^{*}(U^{*})$, so that we may work equivalently with a foliation on either space: the foliation described on $\tot^{*}(U)$ is sent to the foliation on $\tot^{*}(U^{*})$ defined by the level sets of $s$, viewed as a fibrewise linear function on the total space of $U^{*}$.  
\end{remark}

To obtain a similar geometric description of the complex log algebroid $T(-\log D)$, we begin by complexifying the Atiyah sequence~\eqref{seqreati}, obtaining
\begin{equation}
0\to \gl(U^{1,0})\oplus \gl(U^{0,1})\to  \At_{\CC}(U)\to T_\CC M\to 0,
\end{equation}
where we have decomposed $U\otimes_{\RR} \CC= U^{1,0} \oplus U^{0,1}$ into the $+i,-i$ eigenspaces of the given complex structure on $U$, so that $U^{1,0}\cong U$ and $U^{0,1}\cong \ol U$. Also, $\gl(U^{1,0})$ and $\gl(U^{0,1})$ are trivial complex line bundles generated by $\tau\del_{\tau}$ and $\ol\tau\del_{\ol\tau}$, respectively, in the coordinates chosen above.  As before, the sections of $\At_{\CC}(U)$ may be interpreted as $\CC^{*}$--invariant complex vector fields on the total space of $U$. 

The subalgebroid $\At_{\CC}(U,s)$ which preserves the section $s$ (when it viewed as a section of $U^{1,0}$) is then generated by 
\begin{equation}\label{defatcus}
(w\del_{w}+\tau\del_{\tau}, \ol\tau\del_{\ol\tau}, \del_{\ol w}, \del_{x_{3}},\ldots, \del_{x_{n}}),
\end{equation}
so that we have the following diagram of algebroids:
\begin{equation}
\xymatrix@R=1em@C=1em{0\ar[r] &\gl(U^{1,0})\oplus \gl(U^{0,1})\ar[r] &  
\At_{\CC}(U)\ar[r] & T_\CC M\ar[r] &  0\\
0\ar[r] &\gl(U^{0,1})\ar[r]\ar@{^{(}->}[u] & \At_{\CC}(U,s)\ar@{^{(}->}[u]\ar[r] &T(-\log D)\ar[r]\ar[u]_{\anc} &0}
\end{equation}
In this way, $T(-\log D)$ is expressed as the quotient of $\At_{\CC}(U,s)$ by $\gl(U^{0,1})$. In fact, we may take the same quotient of $\At_{\CC}(U)$, obtaining an algebroid of the form 
\begin{equation}
0\to \gl(U^{1,0})\to  \At^{1,0}(U)\to T_\CC M\to 0,
\end{equation}
which involves only the vertical vector fields generated by $\tau\del_{\tau}$.  Then the log tangent algebroid $T(-\log D)$ includes as a subalgebroid of corank 1 in $\At^{1,0}(U)$. We may argue similarly for the complex conjugate divisor $(\ol U, \ol s)$, and obtain the following result, summarizing the above discussion.
\begin{proposition}\label{geolog}
The subalgebroids $\At_{\CC}(U,s)$ and $\At_{\CC}(\ol U, \ol s)$ of the complexified Atiyah algebroid $\At_{\CC}(U)$ map surjectively to $T(-\log D)$ and $T(-\log \ol D)$, respectively, with kernels $\gl(U^{0,1})$ and $\gl(U^{1,0})$.  Their intersection $\At_{\CC}(U,s)\cap\At_{\CC}(\ol U, \ol s)$ is canonically identified with the complexification of the elliptic tangent bundle $T(-\log |D|)$.
\end{proposition}
The algebroid morphisms used in Proposition~\ref{geolog} are displayed below. 
\begin{equation*}
\xymatrix@R=1.3em@C=0em{
&\At_{\CC}(U) \ar[d]         & \\
\At_{\CC}(U,s)\ar[ru]\ar[d]  & T_{\CC}M & \At_{\CC}(\ol U, \ol s)\ar[lu]\ar[d]  \\
T(-\log D)\ar[ru]|\hole &   \At(U,s)\ar[lu]\ar[ru]\ar[d]^{\cong}       & T(-\log \ol D)\ar[lu]|\hole   \\
 & T(-\log|D|)\ar[lu]\ar[ru]         &  \\
}
\end{equation*}

\subsection{Rectification of complex divisors}\label{isodiffdiv}
 
We say that the complex divisors $(U_1, u_1)$, $(U_2, u_2)$ on $M$ are \emph{isomorphic} when there is a bundle map $\psi:U_1\to U_2$, covering the identity on $M$, taking $u_1$ to $u_2$. In terms of ideals, we are requiring $\Ii_{u_1} = \Ii_{u_2}$. This notion is fairly strict: even if $u_1, u_2$ are sections of the same bundle with the same zero set and inducing on it the same co-orientation, they need not be isomorphic as divisors.  For example, take the complex functions $w$ and $w + \bar w^2$ in a small neighbourhood of the origin in the complex plane;  since $\bar w^2$ is not in $\Ii_{w}$, these are non-isomorphic divisors. 

Less strict is the notion of \emph{diffeomorphism} of divisors, where we allow the bundle map $\psi:U_1\to U_2$ to cover a nontrivial diffeomorphism of $M$.  
\begin{lemma}\label{mosdiv}
Let $(U_{s}, u_{s})$, $s\in[0,1]$ be a smooth family of complex divisors on a compact manifold $M$.  Then the family may be 
\emph{rectified}, i.e., there is a smooth family of diffeomorphisms 
$\psi_{s}, s\in[0,1]$, $\psi_{0} =\id{}$, taking the given family to the constant family $(U_{0}, u_{0})$.  
\end{lemma}
\begin{proof}
Let $S=[0,1]$ and $X = M\times S$. The family $(U_{s},u_{s})$ is properly specified by giving a complex divisor $\ul D = (\ul U, \ul u)$ on $X$ which restricts to $(U_{s},u_{s})$ on each fiber $\pi^{-1}(s)$; in particular the zero set of $\ul u$ is transverse to the fibers of the projection $\pi:X\to S$.   Because of this, the derivative of $\pi$ induces the short exact sequence
\begin{equation}\label{vertelltang}
\xymatrix{0\ar[r] & T_{X/S}(-\log|\ul D/S|)\ar[r] &
T_{X}(-\log |\ul D|)\ar[r]^-{\pi_{*}} &
\pi^{*}TS\ar[r] &
0}, 
\end{equation}
where the first bundle is the vertical Lie algebroid assembled from the elliptic tangent bundles of each of the fibers of $\pi$.

Choose a splitting $\nabla:\pi^{*}TS\to T_{X}(-\log |\ul D|)$ for sequence~\eqref{vertelltang}, and use it to lift the standard generator $\del_{s}$ of $TS$.  Identifying the elliptic tangent bundle with the Atiyah algebroid using Proposition~\ref{subell}, we then have a derivation $Z=\nabla(-\del_{s})$ of $\ul U$ preserving $\ul u$ whose time-$s$ flow $\psi^{Z}_{s}$ takes $(U_{s},u_{s})$ diffeomorphically onto $(U_{0},u_{0})$.    
\begin{equation}
\xymatrix{\psi^{Z}_{s}:(U_{s},u_{s})\ar[r]^-{\cong} & (U_{0},s_{0}) }, 
\end{equation}
providing the required fammily of diffeomorphisms.
\end{proof}

%
\begin{corollary}\label{coincimpdiffeo}
Complex divisors with diffeomorphic compact co-oriented zero loci are diffeomorphic in the above sense.
\end{corollary}
\begin{proof}
Denote the divisors by $D_{0}=(U_0, s_0)$, $D_{1}=(U_1,s_1)$. Let $\phi$ be the diffeomorphism of the manifold taking the zero locus of $D_{0}$ to that of $D_{1}$, preserving co-orientations. This implies that, $c_1(U_0)=\phi^{*}c_1(U_1)$, and so there exists a bundle isomorphism $\psi:U_0\to U_1$ covering $\phi$.  We may choose $\psi$ so that $\psi(s_0) = s_1$ outside tubular neighbourhoods of the zero loci.  Then the family $(U_{1},s_t = (1-t)\psi(s_0) + ts_1)$ satisfies the hypotheses of Lemma~\ref{mosdiv} ($\tfrac{d}{dt} s_{t}$ has compact support), which provides a diffeomorphism $\varphi:U_1\to U_1$ covering some diffeomorphism of the base. The composition $\varphi\circ \psi$ is the required diffeomorphism of divisors.
\end{proof}

If $(U,s)$ is a complex divisor on $M$, then it may be \emph{linearized} along $D$, in the following sense. Let $\tot(N)$ be the total space of the normal bundle to $D$, and let $\pi:\tot(N)\to D$ be the bundle projection.  By the isomorphism~\eqref{isonu}, $N$ is a complex line bundle, and defines a complex line bundle $U_{0}=\pi^{*}N$ on $\tot(N)$, which furthermore has a tautological section $s_{0}\in C^{\infty}(\tot(N),\pi^{*}N)$.  This defines a complex divisor $(U_{0},s_{0})$ on $\tot(N)$ which we may call the \emph{linearization} of $(U,s)$ along $D$.  The tubular neighbourhood theorem, together with Corollary~\ref{coincimpdiffeo}, then directly yields the following normal form result.

\begin{corollary}\label{normaldiv}
Any complex divisor $(U,s)$ is diffeomorphic to its linearization $(U_{0},s_{0})$ in some tubular neighbourhood of its zero locus.
\end{corollary}

\section{Generalized complex structures}\label{sec2}

We refer to~\cite{MR2811595} for the theory of generalized complex structures, and summarize the key facts we will need below.  Let $H$ be a real closed 3-form on the smooth manifold $M$. A generalized complex structure $\JJ$ is a complex structure on $\T M = TM\oplus T^*M$, orthogonal for the split-signature metric on this bundle, whose $+i$--eigenbundle $L$
is involutive for the Courant bracket twisted by $H$.  

Generalized complex structures $(M,\JJ, H), (M',\JJ', H')$ are considered equivalent when there is a diffeomorphism $\varphi:M\to M'$ and a two-form $b\in \Omega^2(M,\RR)$ such that $\varphi^*H' = H + db$ and $\JJ'\circ ( \varphi_{*}e^b)=(\varphi_{*}e^b)\circ \JJ$,  where $e^b$ is the automorphism of $\T M$ given by 
\begin{equation}\label{bsym}
e^{b}: X+\xi\mapsto X + \xi + i_X b.
\end{equation}
If the diffeomorphism $\varphi$ is trivial in the above equivalence, we say that $\JJ,\JJ'$ are \emph{gauge equivalent}.
In particular, two generalized complex structures with the same underlying pair $(M,H)$ are gauge equivalent when they are conjugate by a B-field gauge symmetry, namely the automorphism $e^b$ for $b$ closed. 

At each point $p$, a generalized complex structure $\JJ$ induces on $T_p M$ the structure of a symplectic subspace with transverse complex structure; this structure is left invariant by B-field symmetries.  The transverse complex dimension is called the type of $\JJ$ at $p$. Courant involutivity of $L$ guarantees that the symplectic distribution on $M$ integrates to a singular foliation  and that the complex structure transverse to this foliation is integrable in the appropriate sense. In fact, the singular symplectic foliation is associated to a real Poisson structure 
$Q$ which underlies $\JJ$: if we view $\JJ$ as a section of $\wedge^2\T M$, then $Q = \wedge^2\pi(\JJ)$, for $\pi:\T M\to TM$ the projection. The data of a real Poisson structure with transverse complex structure is, however, not sufficient to describe a generalized complex structure; indeed, even for symplectic fibrations over a complex base there are delicate obstructions, as described in~\cite{MR3150703}.

\subsection{The canonical line bundle}\label{canbun}

The action of $\T M$ by interior and exterior product renders the differential forms into a spinor module for the Clifford algebra bundle of $\T M$.  In this action, the maximal isotropic subbundle $L\subset \T_\CC M$ annihilates a rank 1 subbundle $K\subset \wedge^\bullet T^*_\CC M$ of the complex differential forms called the \emph{canonical bundle} of $\JJ$.

The subbundle $K$ is generated pointwise by a form $\rho$ of the following algebraic type
\begin{equation}\label{ptwtype}
\rho = \Omega\wedge e^{B+ i \omega}.
\end{equation}
Here $\Omega$ is a decomposable form and $B$ and $\omega$ are real two-forms satisfying the nondegeneracy condition
\begin{equation}\label{nondegeneracycondition}
\Omega \wedge \ol{\Omega} \wedge \omega^{n-k}\neq 0,
\end{equation}
where $k$ is the degree of $\Omega$ and the dimension of $M$ is $2n$.  We see from this formula that the kernel of $\Omega\wedge\ol\Omega$ is a symplectic distribution and that $\Omega$ defines a transverse complex structure, as discussed earlier. 

While $L$ annihilates $K$, the Clifford action of the subalgebra $\wedge^\bullet \ol L$ on $K$ defines an isomorphism (using the canonical identification $\ol{L}=L^*$)
\begin{equation}\label{fockspace}
(\CC \oplus L^*\oplus \wedge^2 L^* \oplus \cdots \oplus \wedge^n L^*)\otimes K\cong \wedge^\bullet T^*M
\end{equation}
which induces a new $\ZZ$-grading on the differential forms.

The involutivity of $L$ may then be expressed in terms of $K$ by requiring that the differential operator $d^H = d + H\wedge$ takes $K$ into $L^*\otimes K$.  That is, for every nonvanishing local section $\rho$ of $K$, there exists a section $F = X+\xi$ of $\overline{L}$, called the \emph{modular field} of $\rho$, such that 
\begin{equation}\label{modfield}
d^H\rho = F\cdot \rho= i_X \rho + \xi \wedge \rho.
\end{equation}

This condition makes $K$ a generalized holomorphic bundle, in the sense that it has a flat $L$-connection: the restriction $\delbar$ of $d^H$ to the sections of $K$ defines an operator 
\begin{equation}
\delbar: C^\infty(K)\to C^\infty(L^*\otimes K),
\end{equation}
satisfying the Leibniz rule $\delbar(f\rho)= f\delbar\rho +  d_L f \otimes \rho$ and having zero curvature in $\Omega^2_L$, where here $(\Omega^\bullet_L, d_L)$ refers to the Lie algebroid de Rham complex of $L$. In this way, we see that the modular field is the analogue of a connection 1-form for usual connections. As for complex manifolds, there is a distinguished class of generalized complex structures with holomorphically trivial canonical bundle:

\begin{definition}
A {\it generalized Calabi--Yau structure} on $(M,H)$ is a \gcs\ determined by a nowhere vanishing $d^H$-closed form.
\end{definition}

The generalized holomorphic structure on $K$ actually renders the total space $\tot(K)$ into a generalized complex manifold, in the same way that a rank-1 Poisson module inherits a Poisson structure on its total space~\cite{Polishchuk1997,Hitchin2011a}. To see this, let $\rho$ be a local trivialization of $K$ as before, and let $\tau$ be the dual section, viewed as a fibrewise linear coordinate on $\tot(K)$.  We consider the product generalized complex structure $d\tau\wedge\rho$, which represents the product of the generalized complex structure with the standard complex structure on $\CC$, and we deform it by the Maurer-Cartan element $E\wedge F$, where $E=\tau\del_\tau$ is the complex Euler vector field and $F$ the modular field of $\rho$. The resulting form $\varpi$, given by 
\begin{equation}\label{localKCY}
\varpi = e^{E\wedge F}d\tau\wedge \rho = \tau F\cdot \rho + d\tau \wedge \rho,
\end{equation}
is then independent of the choice of local section $\rho$ and defines a canonical generalized Calabi-Yau structure on $\tot(K)$, integrable with respect to the pullback of the 3-form $H$ to the total space, which we also denote by $H$ below.
\begin{theorem}\label{gcytotk}
Let $\Theta$ be the tautological form on the total space of the canonical line bundle $K$.  Then the differential form 
\begin{equation}\label{cangcyform}
\varpi = d^{H} \Theta
\end{equation}
defines a generalized Calabi-Yau structure, which furthermore satisfies 
\begin{equation}\label{liouville}
i_{E}\varpi = \Theta,
\end{equation}
so that the Euler vector field acts as a Liouville vector field, in the sense that $L^{H}_{E}\varpi = (d^{H}i_{E} + i_{E} d^{H})\varpi = \varpi$.
\end{theorem}
\begin{proof}
If $\rho$ is a local trivialization of $K$ with corresponding fibre coordinate $\tau$, then $\Theta=\tau\rho$ is a local expression for the tautological form on $\tot(K)$. We then see that $d^{H}\Theta = d\tau\wedge\rho + \tau d^{H}\rho$, and using the definition~\eqref{modfield} of the modular field we obtain expression~\eqref{localKCY}, showing that it is independent of the local trivialization.  Nondegeneracy of $\varpi$ follows from the fact that $d\tau\wedge d\ol\tau \wedge\Omega\wedge\ol \Omega \wedge \omega^{n-k}$ is nonvanishing on $\tot(K)$.  Finally, $\varpi$ is exact and so certainly closed, defining the required generalized Calabi-Yau structure. Identity~\eqref{liouville} then follows from the local expression~\eqref{localKCY}, since $i_{E}(F\cdot\rho)=0$ and $i_{E}(d\tau\wedge\rho) = \tau\rho = \Theta$, as required.
\end{proof}

Since $\varpi$ satisfies $L_{E}\varpi = \varpi$, it follows that the line generated by $\varpi$ in the forms on $\tot(K)$ is invariant under rescaling, defining a $\CC^{*}$--invariant generalized complex structure $\JJ_{K}$ on $\tot(K)$. In fact, if we consider the principal $\CC^{*}$--bundle $\tot^{*}(K)$ defined by deleting the zero section, we may express the original generalized complex structure on $M$ as a Courant reduction of the structure on $\tot^{*}(K)$ along the generalized symmetry $E$, in the sense developed in~\cite{MR2323543}.  In particular, $\JJ$ is given by the Dirac pushforward~\cite{MR1973074} of $\JJ_{K}$, as follows.
\begin{proposition}\label{pushforwardrelation}
Let $L, L_{K}$ be the $+i$-eigenbundles of $\JJ$ and $\JJ_{K}$, respectively.  Then $L$ is given by the Dirac pushforward of $L_{K}$ along the bundle projection $\pi:K\to M$, i.e. 
\begin{equation}
L = \pi_{*}L_{K} = \{\pi_{*}X + \eta\in \T M\ |\ X + \pi^{*}\eta \in L_{K}\}.
\end{equation}
\end{proposition}
\begin{proof}
Let $\rho$ be a local trivialization of $K$ and $\tau$ the corresponding fibrewise linear coordinate on $\tot(K)$, so that $\varpi = d\tau\wedge\rho + \tau F\cdot \rho$, where $F$ is the modular field of $\rho$.  Then we have
\begin{equation}\label{pushdirgc}
(X + \pi^{*}\eta)\cdot  \varpi = -d\tau \wedge ((\pi_{*}X+\eta)\cdot \rho) + 
(i_{X}d\tau + 2\tau \left<\pi_{*}X + \eta, F\right>)\rho,
\end{equation}
where $\left<\cdot,\cdot\right>$ denotes the natural split-signature metric on $\T M$. 
If $X+\pi^{*}\eta\in L_{K}$, it annihilates $\varpi$, and both summands in~\eqref{pushdirgc} vanish independently since $d\tau$ is the only non-basic form. Thus, $\pi_{*}X+\eta\in L$ and we have $\pi_{*}L_{K}\subset L$. For the reverse inclusion, let $Y+\eta\in L$.  By choosing $X = Y - 2 \left<Y + \eta, F\right>\tau\del_{\tau}$, we see from~\eqref{pushdirgc} that $(X+\pi^{*}\eta)\cdot\varpi = 0$, as required.
\end{proof}

While there is a $\CC^{*}$-invariant generalized complex structure on $\tot^{*}(K)$, the Calabi-Yau form $\varpi$ is not invariant.  As a result, $i_{E}\varpi$ is not a basic form. Instead, $i_{E}\varpi$ varies linearly on each fibre and can be viewed as a section of $K^{*}\otimes \wedge^{\bullet}_{\CC}T^{*}M$ on $M$, defining an inclusion
\begin{equation}
i_{E}\varpi : K\hookrightarrow  \wedge^\bullet T^*_\CC M,
\end{equation}
recovering the original canonical bundle as a subbundle of the complex forms. 

\subsection{Generalized complex structures of type 1}

While our main interest in this paper is in generalized complex structures which are almost everywhere of type 0, it will be helpful to understand structures of type 1, which, as we shall see, govern the singular behaviour of generically type 0 structures. 

Let $D$ be a smooth manifold with real closed 3-form $H$, and let $\JJ$ be a generalized complex structure of type 1 on $(D,H)$, so that the underlying real Poisson structure $Q$ defines a foliation by symplectic leaves of real codimension 2.  The conormal bundle $\nu^{*} = \JJ(T^{*}D)\cap T^{*}D$ to the symplectic foliation is then a rank 1 complex subbundle of $\T D$, whose complexification decomposes in $+i, -i$ eigenbundles for $\JJ$ respectively:
\begin{equation}
\nu^{*}_\CC = \nu^{*}_{1,0} \oplus\nu^{*}_{0,1}.
\end{equation}
As a result, if we apply the tangent projection to the $+i$-eigenbundle $L\subset \T_{\CC}D$ of $\JJ$, we obtain the abelian Lie algebroid extension 
\begin{equation}\label{cxdir}
0\to \nu^*_{1,0} \to L \to A \to 0,
\end{equation} 
where $A\subset T_{\CC}D$ is the involutive corank 1 complex distribution with annihilator $\nu^{*}_{1,0}$.  We use the notation $(\Omega^{\bullet}_{A},d_{A})$ for the de Rham complex of $A$, an elliptic complex with cohomology groups denoted by $H^{\bullet}_{A}$.
As is always the case for regular Dirac structures~\cite{MR2811595}, the subbundle $L\subset\TT_{\CC}M$ determines and is determined by a 2-form 
$\sigma\in \Omega^{2}_{A}(D)$ via the graph construction
\begin{equation}\label{subdir}
L = \{ Z + \zeta \in A\oplus T^*D\ |\ \iota^*\zeta = i_Z\sigma\},
\end{equation}
where $\iota:A\hookrightarrow T_{\CC}M$ is the inclusion. Involutivity of $L$ holds if and only if 
\begin{equation}\label{dsigma}
d_A \sigma + \iota^{*}H=0,
\end{equation}
 From expression~\eqref{subdir}, we see that the condition $L\cap\ol L = \{0\}$ holds if and only if $\sigma$ has nondegenerate imaginary part when pulled back to the real distribution $\Delta$ defined by the transverse intersection $A\cap\ol A$, recovering the symplectic structure determined by the Poisson structure $Q$.  We summarize these observations as follows.
\begin{theorem}\label{asigma}
A type 1 generalized complex structure on $(D,H)$ is equivalently specified by a pair $(A,\sigma)$, where:
 $A\subset T_{\CC}D$ is an involutive distribution of complex corank 1 that is transverse to its complex conjugate, and 
 $\sigma$ is a section of $\wedge^{2}A^{*}$ such that the integrability condition~\eqref{dsigma} holds and such that its pullback to $\Delta\otimes\CC = A\cap\ol A$ has nondegenerate imaginary part.
\end{theorem}

The conormal bundle $\nu^{*}_{1,0}$ has a natural partial flat connection along $A$, given by the Lie derivative and often called the Bott connection. This equips $A\oplus \nu^{*}_{1,0}$ with a standard Lie bracket $[-,-]_{0}$, making it a Lie algebroid. Since $L$ is an extension of $A$ by $\nu^{*}_{1,0}$, we may split the sequence~\eqref{cxdir} and express the bracket on $L$ as a deformation of the standard one by a tensorial term $F\in\Omega^{2}_{A}(D,\nu^{*}_{1,0})$, that is, 
\begin{equation}\label{algcurv}
[X+\xi,Y+\eta]_{L} = [X+\xi,Y+\eta]_{0}+F(X,Y).
\end{equation}
The cohomology class $[F]\in H^{2}_{A}(D,\nu^{*}_{1,0})$ is independent of the splitting and called the \emph{twisting class} of the generalized complex structure.  We now describe this class in terms of the data provided by Theorem~\ref{asigma}.

The pullback $\iota^{*}$ of differential forms to forms on $A$ has kernel given by the differential ideal generated by $\nu^{*}_{1,0}$. This ideal may be identified with the de Rham complex of $A$ with coefficients in $\nu^{*}_{1,0}$, so that we have a short exact sequence of complexes, which in degree $k$ gives
\begin{equation}\label{relmorp}
\xymatrix{
0\ar[r] & \Omega^{k-1}_{A}(D,\nu^{*}_{1,0})\ar[r] & \Omega^{k}(D,\CC) \ar[r]^-{\iota^{*}} & \Omega^{k}_{A}(D)\ar[r] & 0
}
\end{equation}
Suppose we choose a 2-form $\wt\sigma\in\Omega^{2}(D,\CC)$ such that $\iota^{*}\wt\sigma = \sigma$. Then the map $Z\mapsto Z + i_{Z}\wt\sigma$ 
defines a splitting of~\eqref{cxdir}. Using the Courant bracket twisted by $H$ to compute the curvature of this splitting, we obtain the closed 3-form $H+d\wt\sigma$, which since it is in the kernel of $\iota^{*}$, is a closed form in $\Omega^{2}_{A}(D,\nu^{*}_{1,0})$, providing a representative of the twisting class.
\begin{proposition}\label{twprop}
The twisting class of the generalized complex structure associated to $(A,\sigma)$ by Theorem~\ref{asigma} coincides with the class in $H^{2}_{A}(D,\nu^{*}_{1,0})$ determined by the class $[(H,\sigma)]$ in the third relative cohomology of the morphism $\iota:A\to T_{\CC}D$.  

Furthermore, this class vanishes if and only if there exists $\wt\sigma\in \Omega^{2}(M,\CC)$ with $\iota^{*}\wt\sigma = \sigma$ and $d\wt\sigma + H=0$.
\end{proposition}
\begin{proof}
The relative cohomology of $\iota$ is the cohomology of the total complex of the double complex defined by the morphism $\iota^{*}$.  By the exact sequence~\eqref{relmorp}, the $k^{\text{th}}$ relative cohomology of $\iota$ computes precisely $H^{k-1}_{A}(D,\nu^{*}_{1,0})$.  Using the exact sequence, the cocycle $(H,\sigma)$ is sent to $H+d\wt\sigma$, where $\wt\sigma$ is chosen such that $\iota^{*}\wt\sigma = \sigma$.  This coincides with the twisting class as computed by applying the Courant bracket to the splitting of~\eqref{cxdir} defined by $\wt\sigma$, as argued above.  

For the second statement, note that the class in relative cohomology vanishes if and only if there exists $b\in\Omega^{2}(D,\CC)$ and $\tau\in \Omega^{1}_{A}(D)$ such that $db = H$, and $\iota^{*}b -d_{A}\tau = \sigma$.  If such $(b,\tau)$ exists, then choose $\wt \tau\in\Omega^{1}(D,\CC)$ such that $\iota^{*}\wt \tau = \tau$ and put $\wt\sigma = -b - d\wt \tau$, satisfying the requirements $\iota^{*}\wt\sigma = \sigma$ and $d\wt\sigma+H= 0$. Conversely, given such a $\wt\sigma$, simply put $b=-\wt\sigma$ and $\tau=0$, trivializing the twisting class. 
\end{proof}

In terms of differential forms, a local generator for the canonical line bundle $K$ of a generalized complex structure of type 1 may be written as  
\begin{equation}\label{type1spin}
\rho = \Omega \wedge e^{B+i\omega},
\end{equation}
where $\Omega$ is a complex 1-form locally trivializing $\nu^{*}_{1,0}$, and $B,\omega$ are real 2-forms such that $\omega$ is symplectic on the foliation determined by $\Omega\wedge\ol\Omega$.  The form $B+i\omega$ is uniquely determined only up to adding a 2-form in the ideal generated by $\Omega$, that is, $\sigma = \iota^{*}(B+i\omega)\in \Omega^{2}_{A}(D)$ is uniquely determined. The integrability condition is then that $\Omega\wedge(d(B+i\omega) + H) = 0$, which is a restatement of the condition $d_{A}\sigma=0$.  The form $d(B+i\omega) + H$, then, represents the twisting class in the cohomology of the differential ideal generated by $\Omega$. As shown in Proposition~\ref{twprop}, the twisting class is the obstruction to finding forms $B,\omega$ satisfying $d(B+i\omega) + H = 0$, which implies, but is stronger than, the integrability condition. 

The generalized Calabi-Yau condition holds for the above structure if and only if there is a global closed trivialization $\Omega$ for the line bundle $\nu^{*}_{1,0}$. In such a case, we have the exact sequence characterizing the differential graded ideal $\sI^{\bullet}_{\Omega}$ generated by $\Omega$:
\begin{equation}\label{idealkoszul}
\xymatrix{
0\ar[r] & \sI^{k}_{\Omega}\ar[r] & \Omega^{k}(D,\CC)\ar[r]^-{\Omega\wedge\cdot} & \sI^{k+1}_{\Omega}\ar[r] & 0
}.
\end{equation}
In this language, the twisting class is the class in $H^{3}(\sI^{\bullet}_{\Omega})$ determined by $d(B+i\omega) + H$.  Note that it maps to $[H]\in H^{3}(D,\RR)$ in the long exact sequence resulting from~\eqref{idealkoszul}, and so a necessary (but not sufficient) condition for the vanishing of the twisting class is that $[H]=0$ in $H^{3}(D,\RR)$. 


Assuming that the Poisson structure on $D$ makes it into a fibration, it is clear that if the fibers are not all symplectomorphic, the twisting class must be nontrivial. But even if the fibers are all symplectomorphic it may still happen that the twisting class is nonzero, as we illustrate next.
\begin{example}\label{kodthur}
 Let $D$ be the Kodaira--Thurston manifold, a $T^{2}$ principal bundle over $T^{2}$ given by the product of $S^{1}$ with the circle bundle with primitive Chern class. Let $\Omega$ be a complex 1-form defining a Calabi-Yau complex structure on $T^{2}$ and let $\theta_{1},\theta_{2}$ be connection 1-forms for the trivial and nontrivial circle bundles, respectively, so that $d\theta_{1}=0$ while $d\theta_{2}= i\Omega\wedge\ol\Omega$ after normalization. Then $\omega = \theta_{1}\wedge\theta_{2}$ defines a symplectic form on each torus fiber, and the following defines a generalized Calabi-Yau structure with $H=0$ on $D$: 
\begin{equation}
\rho = \Omega\wedge e^{i\omega}.
\end{equation}
If the twisting class vanished, there would be a closed form $B+ i\omega'$  for which $\rho = e^{B+i\omega'}\wedge \Omega$, but in this case $(\rho,\ol{\rho}) = 2i \Omega\wedge \ol{\Omega} \wedge \omega'$ would be exact, as $\Omega\wedge \ol{\Omega} $ is exact and $\omega'$  is closed, contradicting the fact that it is a nowhere vanishing volume form.

\end{example}

\subsection{Topological constraints for the anticanonical divisor}

We now describe several topological properties of manifolds $D$ which admit type 1 generalized Calabi-Yau structures, and indicate how these are affected by the vanishing of the twisting class introduced above.  	   

\begin{definition}
A \gcs\  is {\it proper} if it is compact and its symplectic leaves are compact.
\end{definition}

The following proposition summarizes basic topological properties of type one \gcy\ manifolds.

\begin{theorem}\label{prop:topology1} Let $D^{2n}$ be a compact, connected type one \gcy\ manifold. Then all of the following hold:
\begin{enumerate}
\item There is a surjective submersion $\pi: D \into T^2$, hence $b_1(D) \geq 2$ and $\chi(D) =0$.
\item If $D$ has a compact leaf, then $D$ is proper and $\pi$ can be chosen so that the components of the fibers of $\pi$  are the symplectic leaves of the underlying Poisson structure.
\item If the twisting class vanishes, the structure can be deformed into a proper one, $D$ admits a symplectic structure for which $\pi:D \into T^2$ is a symplectic fibration and there are classes $a,b \in H^1(D)$ and $c \in H^2(D)$ such that $abc^{n-1} \neq 0$. In particular $b_i(D)\geq 2 $ for $0<i<2n$.
\end{enumerate}
\end{theorem}
\begin{proof}
Throughout the proof we let $\rho = e^{B+i\omega}\wedge \Omega$ be a $d^H$-closed trivialization of the canonical bundle of $D$.

{\it 1.}  Let $\Omega_R$ and $\Omega_I$ be the real and imaginary parts of $\Omega$. First we show that $[\Omega_R]$ and $[\Omega_I]$ are linearly independent classes in $H^1(D)$. If there were a nontrivial linear combination, say, $\lambda_R[\Omega_R] + \lambda_I[\Omega_I] =0$,  we could define a map $\pi_1:D\into \R$ by
$$\pi_1(p)= \int_{p_0}^p \lambda_R\Omega_R + \lambda_I\Omega_I,$$
where the integral, taken over any path connecting the reference point $p_0$ to $p$, is well defined  because the integrand is an exact form.  Finally, if, say, $\lambda _R \neq 0$, then nondegeneracy implies that $\omega^{n-1}\wedge d\pi_1 \wedge \Omega_I \neq 0$, showing that $d\pi_1$ is nowhere zero and hence $\pi_1$ is a submersion of a compact manifold in $\R$, which is a contradiction.

To construct $\pi$ observe that since nondegeneracy is an open condition, there is a closed form $\Omega'\in \Omega^1(D;\C)$ near $\Omega$, whose real and imaginary parts represent linearly independent rational cohomology classes  and such that
\begin{equation}\label{eq:nondegeneracy condition 2}
\omega^{n-1}\wedge\Omega'\wedge \overline{\Omega'} \neq 0.
\end{equation}
Linear independence of the real and imaginary parts of $\Omega'$ implies that the following map is  a submersion
$$\pi:D\into \C/\Lambda; \qquad \pi(p) =\int_{p_0}^p \Omega'.$$
where $\Lambda = [\Omega'](H_1(D;\ZZ))$ is a co-compact lattice,
$p_0$ is a fixed reference point and the integral is independent of the path connecting $p_0$ to $p$.

{\it 2.} Let $X$ be a complex vector field for which  $\Omega(X)=1$ and $\ol{\Omega}(X)=0$. Then the real and imaginary parts of $X$, $X_R$ and $X_I$, are pointwise linearly independent, preserve $\Omega$ and $\ol\Omega$ and hence preserve the foliation determined by $\Omega\wedge \ol\Omega$.

Let $F$ be a compact leaf and define a map $\gf_F:F \times \R^2 \into D$ by
$$\gf_F(p,\lambda_1,\lambda_2) = e^{\lambda_1 X_R + \lambda_2 X_I}(p).$$
Since $(\gf_F)_*(TF\oplus \R^2) = TD$, we conclude that $\gf_F$ is a local diffeomorphism and  since the flow of the vector field $\lambda_1 X_R + \lambda_2 X_I$ preserves the foliation we conclude that all leaves in a neighbourhood of $F$ are diffeomorphic to $F$, that is, 
\begin{enumerate}
\item[a)] if a leaf $F$ is compact, the map $\gf_F$ above gives a local diffeomorphism between a neighbourhood of $F$ and $F \times \mathbb{D}^2$ for which the projection onto the open disc $\mathbb{D}^{2}\subset\RR^{2}$ is the quotient map of the foliation, and 
\item[b)] the set of points which lie in a leaf diffeomorphic to $F$ is an open set, $U \subset D$.
\end{enumerate}
Next, if $p \in \ol{U}$, let $\alpha:\mathbb{D}^{2n-2} \into D$ be a parametrization of the leaf through $p$, with $\mathbb{D}^{2n-2}\subset \RR^{2n-2}$ an open ball. Then
\begin{equation}
\gf: \mathbb{D}^{2n-2} \times \mathbb{D}^2 \into D, \qquad \gf(x,\lambda_1,\lambda_2) = e^{\lambda_1 X_R + \lambda_2 X_I}(\alpha(x)),
\end{equation}
is a local diffeomorphism and hence its image contains a point $q \in U$, say $q = e^{\lambda_1 X_R + \lambda_2 X_I}(\alpha(x))$. Let $F$ be the compact leaf through $q$. Then $\alpha(\mathbb{D}^{2n-2})\cap \Im(\gf_F) \neq \emptyset$, as, inverting the exponential, we get $\gf_F(q,-\lambda_1,-\lambda_2) \in \Im(\alpha)$ and hence $\gf_F$ gives a diffeomorphism between $F$ and the leaf through $p$. That is, the set $U$ above is also closed, and since $D$ is connected, $U = D$ and by property {a)} we conclude that $D$ is a fibration $D \into \Sigma$ over a compact surface.
To determine $\Sigma$, we observe that $\Omega$ is basic for this fibration  since it is closed and annihilates vertical vectors. Thererefore $\Sigma$ has a nowhere vanishing closed 1-form, giving $\Sigma = T^2$.

{\it 3.} If the twisting class vanishes, we can choose the forms $B$ and $\omega$ so that $d^H  e^{B + i\omega} =0$. Changing $\Omega$ to nearby form representing a rational class transforms the structure into a proper one. Then $\omega + \frac{1}{2i}\Omega \wedge \ol{\Omega}$ defines a symplectic form on $D$, and due to \eqref{eq:nondegeneracy condition 2}, it is  symplectic on the leaves of the distribution generated by $\Omega_R'$ and $\Omega_I'$, rendering $\pi$ a symplectic fibration.
\end{proof}

\begin{example}
There is no generalized Calabi-Yau structure of type one on $\C P^n, n>0$.
\end{example}

\begin{example}
For $n_i \in \NN$, the manifold $S^{2n_1+1} \times \cdots \times S^{2n_{2k}+1}$ only admits a type one \gcys\ if $n_i=0$ for all $i$.
\end{example}

\begin{example}
The only compact Lie groups that admit type one \gcys s are Abelian.
\end{example}

The manifold $S^1\times S^3$ does admit a type one \gcs\ with topologically trivial canonical bundle and, as we will see in Example \ref{ex:S^1 x S^5}, so does $S^1 \times S^5$, but by the results above, these are not generalized Calabi-Yau.

According to Proposition \ref{prop:topology1}, every type one \gcy\ is a fibration over a 2-torus. Next we show that in four dimensions the converse to this statement is also true, therefore giving a full characterization of type one \gcys s in that dimension.

\begin{theorem}
Let $\pi:D \into T^2$ be a fibration of an orientable, connected, compact four-manifold $D$ over the torus. Then $D$ admits a proper \gcys, integrable \wrt\ the zero 3-form,  for which the fibers of $\pi$ are the symplectic leaves. 
\end{theorem}
\begin{proof}
We use the following result, which translates Thurston's argument for symplectic structures on symplectic fibrations to \gc\ manifolds.

\begin{theorem}{\cite[\textsection 2.4]{cavalcanti-2004}}\label{theo:Thurston}
Let $\pi:M^{2n} \into N^{2n-2}$ be a fibration over a generalized complex base whose fibers are nontrivial in $H_2(M;\R)$. Then $M$ admits a \gcs\ whose canonical bundle is
$$K_M = e^{i \epsilon \omega} \wedge \pi^* K_N,$$
where $\omega$ is a closed 2-form on $M$ which is symplectic when restricted to the fibers, $K_N$ is the canonical bundle of $N$ and $\epsilon >0$ is a small real number.
\end{theorem}

It follows from the theorem that if $N$ is \gcy, then $M$ is as well, and the symplectic leaves of $M$ are the inverse images of the symplectic leaves of $N$.
Further, observe that the hypothesis of the theorem are fulfilled if the genus of the fiber is not one, as in this case the Euler class of the vertical bundle evaluates nonzero in any given fiber, showing that their real homology is nontrivial.

In the present case, if $D \to T^2$ is a fibration and the fibers have genus different than one, we can endow $T^2$ with its standard complex structure (making it a type one \gcy) and Theorem \ref{theo:Thurston} implies the desired result. Notice that in this case, the structure has trivial twisting class.

The only part of the proof that does not follow from Theorem \ref{theo:Thurston}  is the case of torus bundles over the torus. These were classified by Sakamoto and Fukuhara \cite{MR732086}: there are fourteen of these up to isomorphism. The existence of symplectic structures making these into sympletic fibrations was studied by Geiges \cite{MR1181312} and only two of the fourteen classes do not admit such structure. Those that do admit  this symplectic fibration structure satisfy the hypotheses of Theorem \ref{theo:Thurston} if we endow the base torus with a complex structure; hence these admit a type one \gcys.

The two exceptional cases that are not symplectic fibrations are nilmanifolds associated to the Lie algebras
\begin{equation}
\begin{aligned}
\mathfrak{nil}^3\oplus \R = \{e_1, e_2, e_3, e_4| [e_1,e_2]= -e_3 \mbox{ and } [e_i,e_j]=0 \mbox{ otherwise}\},\\
\mathfrak{nil}^4 = \{e_1, e_2, e_3, e_4| [e_1,e_2]= -e_3, [e_1,e_3]= -e_4,  \mbox{ and } [e_i,e_j]=0 \mbox{ otherwise}\}.
\end{aligned}
\end{equation}

The compact spaces associated to $\mathfrak{nil}^3\oplus \R$ have a type one \gcys\ given, at the Lie algebra level, by $e^{ie^3\wedge e^4}\wedge (e^1+ie^2)$, where  $\{ e^1, e^2, e^3, e^4\}\in \mathfrak{nil}^{3*}\oplus \R^*$ is the dual basis to the basis $\{e_1, e_2, e_3, e_4\}$ above. For this structure, the symplectic leaves are the fibers of the torus bundle in question and this is essentially the structure we introduced in the Kodaira--Thurston manifold in Example \ref{kodthur}.
The compact manifold obtained from $\mathfrak{nil}^4$ has a \gcys\ determined by the invariant form $e^{ie^3\wedge e^4}\wedge (e^1+ie^2)$ and again the fibers of the torus bundle agree with the symplectic leaves of the structure.
In both of these exceptional cases, the twisting class is nonzero as $\Omega\wedge \ol\Omega$ is exact.\end{proof}

\subsection{Generalized holomorphic line bundles}\label{hollinbu}

We now give a detailed description of generalized holomorphic bundles over generalized complex manifolds of type 1.  Recall that a generalized holomorphic bundle is simply a vector bundle equipped with flat algebroid connection for the $+i$-eigenbundle $L$.  Since we have described $L$ as the extension~\eqref{cxdir}, we choose a splitting of the sequence as before, with twisting form $F\in \Omega^{2}_{A}(D, \nu^{*}_{1,0})$ as defined in~\eqref{algcurv}.  Then an $L$-connection on the bundle $V$ decomposes as 
\begin{equation}
\delbar^{L} = \phi + \delbar^{A},
\end{equation}
where $\phi:V\to \nu_{1,0}\otimes V$ is called the \emph{transverse Higgs field} and $\delbar^{A}:C^{\infty}
(V)\to C^{\infty}(A^{*}\otimes V)$ is a partial connection along the distribution $A$. The tensor $\phi$ is independent of the splitting of $L$, but a change of splitting by $\alpha\in \Omega^{1}_{A}(\nu^{*}_{1,0})$ modifies $\delbar^{A}$ by the following transformation:
\begin{equation}
\delbar^{A}\mapsto \delbar^{A}+i_{\alpha}\phi.
\end{equation}
The flatness of $\delbar^{L}$ may then be expressed as follows, a special case of the general result in~\cite{MR2681704}. 
\begin{proposition}
The $L$-connection $\delbar^{L}=\phi + \delbar^{A}$ is flat if and only if $\phi$ is flat and the curvature of $\delbar^{A}$ coincides with the contraction of $\phi$ with the twisting form. That is, if and only if the following hold:
\begin{equation}
\begin{aligned}
{[}\delbar^{A},\phi{]}
 &=0\\
\mathrm{curv}(\delbar^{A}) &= i_{\phi}F.
\end{aligned}
\end{equation}
\end{proposition}
If $V$ has rank 1, then $\phi$ is simply a flat section of $\nu_{1,0}$, which if nonzero determines a generalized Calabi-Yau structure~\eqref{type1spin}, where $\Omega$ is dual to $\phi$.  By choosing an extension of $\delbar^{A}$ to a full connection, we immediately obtain the following analog of Bott's obstruction~\cite{MR0266248} governing generalized holomorphic line bundles. 
\begin{theorem}\label{theo:twisting class and chern class} 
Let $D$ be a type 1 generalized complex manifold with complex distribution $A$ and twisting class $F\in H^{2}_{A}(D,\nu^{*}_{1,0})$. Let $\phi\in H^{0}_{A}(D,\nu_{1,0})$ be a flat section of $\nu_{1,0} = T_{\CC}D/A$. 
A complex line bundle $N$ over $D$ admits a holomorphic structure with transverse Higgs field $\phi$ if and only if 
\begin{equation}\label{cohomconstr}
\iota^{*}c_{1}(N) = i_{\phi}F \in H^{2}_{A}(D),
\end{equation}
where $\iota:A\to T_{\CC}D$ is the inclusion map.  
\end{theorem}

If $\phi$ vanishes, then the constraint~\eqref{cohomconstr} implies that the real Chern class $c_{1}(N)$ vanishes when pulled back to $A$, and in particular to the foliation defined by $A\cap \ol A$. If $\phi$ is nonzero, then contraction by $\phi$ identifies $H^{2}_{A}(D,\nu_{1,0}^{*})$ with $H^{2}_{A}(D)$, and \eqref{cohomconstr} implies that if $c_{1}(N)$ vanishes, then the twisting class must also vanish.  In general, however, for nonzero $\phi$, the real class $c_{1}(N)$ need not vanish along the symplectic foliation.

\subsection{Stable generalized complex structures}\label{stablestructures}

The projection of a differential form to its zero-degree component is a linear map which, when restricted to the canonical bundle $K\subset\wedge^{\bullet}T^{*}_{\CC}M$ of a generalized complex manifold,  defines an anticanonical section $s\in C^\infty(M, K^*)$. In view of the pointwise structure~\eqref{ptwtype}, we see that a generalized complex structure is of type zero, that is, equivalent to a usual symplectic structure, precisely on the nonvanishing locus of this anticanonical section. 
\begin{definition}
A generalized complex structure is \emph{stable} when its anticanonical section vanishes transversely, so that $D=(K^{*},s)$ defines a complex divisor called the \emph{anticanonical divisor}.
\end{definition}

Of course, the simplest example of a stable generalized complex structure is one where $s$ is nowhere vanishing; then $K$ is generated by the form $e^{B+i\omega}$, where $\omega$ is a usual symplectic form and $B$ is a real 2-form satisfying $dB = H$. In the following we are interested in studying nondegenerate structures with nontrivial anticanonical divisor. 

\begin{example}\label{holpoisson}
Let $M$ be a complex $2n$-manifold equipped with a holomorphic Poisson structure $\pi$, defining a generalized complex structure $\JJ_{\pi}$ with canonical line bundle locally generated by $e^{\pi}\Omega$, where $\Omega$ is a trivialization of the holomorphic canonical bundle. The structure $\JJ_{\pi}$ is stable if and only if the anticanonical section $\pi^{n}$ is transverse to zero in $\wedge^{2n}T_{1,0}M$.  
\end{example}

Many examples of stable generalized complex structures which are not of the above holomorphic Poisson type are now known in dimension four: see~\cite{MR2312048,MR2574746,MR2958956,MR3177992,Goto:2013vn}. These references also provide examples of almost complex 4-manifolds which admit neither complex nor symplectic structures, though they do admit stable generalized complex structures.

We now show that stable generalized complex structures are 
sandwiched between type 1 generalized Calabi--Yau structures: one on the total space of the canonical line bundle and another on the anticanonical divisor.
 
\begin{lemma}\label{totspaceksgc}
Let $\tot(K)$ be the total space of the canonical line bundle of a stable generalized complex structure.   The generalized Calabi-Yau structure~\eqref{cangcyform} on $\tot(K)$ has constant type 1 away from the zero section. 
\end{lemma}
\begin{proof}
Let $\rho$ be a local trivialization of $K$ and let $\tau$ be the associated fibrewise linear coordinate on $\tot(K)$.  Then the tautological form may be written $\Theta = \tau\rho$, and the Calabi-Yau form is $\varpi = d^{H}\Theta = d^{H}(\tau\rho)$. 
Therefore, the component of $\varpi$ with lowest degree is $d(\tau s(\rho))$, the derivative of the fibrewise linear function on $\tot(K)$ defined by the anticanonical section $s$. The transversality of $s$ guarantees that this 1-form is nonzero when $\tau\neq 0$, showing that $\varpi$ has type 1, as required.
\end{proof}

We now describe the geometry inherited by the anticanonical divisor $D$.  In Section~\ref{lindeglocsec}, we will show that the generalized complex structure in a tubular neighbourhood of $D$ is completely determined by the structure of $D$ which we detail here.
The anticanonical divisor is an example of a \emph{generalized Poisson submanifold}, that is, its conormal bundle $N^*$ is a complex subbundle: $\JJ \N^* \subset \N^*$.  Such submanifolds inherit generalized complex structures by reduction~\cite{MR2323543}. Indeed, along $D$ the following exact sequence expresses $\T D = TD\oplus T^{*}D$ as a quotient of $\J$-invariant subbundles of $\T M|_D$: 
\begin{equation}\label{eq:submanifold}
\xymatrix{
0\ar[r] &  \N^*\ar[r]& \N^{*\perp}\ar[r]^-{\pi}& \N^{*\perp}/\N^*\cong\T D\ar[r]& 0
}.
\end{equation}
As a result, $D$ inherits a generalized complex structure $\J_D$, whose integrability with respect to the pullback of $H$ follows from integrability of $\J$. It also follows that $D$ is a Poisson submanifold (in fact, the degeneracy locus) for the underlying real Poisson structure $Q$.

For a Poisson structure $Q$, we say, following~\cite{MR3100779},  that the Poisson submanifold $D$ is \emph{strong} when any local Poisson vector field is tangent to $D$; degeneracy loci of $Q$ are the typical examples of strong Poisson submanifolds, whereas symplectic leaves need not be strong. There is a corresponding notion for generalized complex structures: we say that $D$ is strong when each local generalized complex symmetry $v\in C^\infty(\T M)$ restricts along $D$ to a section of the orthogonal complement of $N^*$ (that is, the vector component of $v$ must be tangent to $D$).  Strong submanifolds $D$ have the property that generalized holomorphic bundles pull back to $D$.  In particular, the anticanonical bundle $K^{*}$ pulls back to a generalized holomorphic bundle along $D$. Tranversality of the anticanonical section $s$ implies that $ds|_{D}:N\to K^{*}$ is an isomorphism, so that we obtain a generalized holomorphic structure on the normal bundle to $D$.

\begin{theorem}\label{gholstrnd}
The anticanonical divisor of a stable generalized complex structure inherits a generalized Calabi-Yau structure of type $1$ with distinguished Calabi-Yau form
\begin{equation}
\rho_{D}=\Omega\wedge e^{\sigma}.
\end{equation}
Furthermore, it inherits a generalized holomorphic structure on its normal bundle with transverse Higgs field dual to $\Omega$. 
\end{theorem}

\begin{proof}
Let $\rho$ be a local trivialization for $K$ near a point $p\in D$. Then $\rho^{0}(p)=0$ and $d\rho^{0}(p)\neq 0$. The integrability condition~\eqref{modfield} provides a section $F = X+\xi\in \ol L$ such that $d\rho^{0} = i_{X}\rho^{2} + \xi\rho^{0}$, implying that both $X$ and $\rho^{2}$ are nonvanishing along $D$; in particular this means $\JJ$ has type $2$ along $D$. 
Choosing a complement $C^{\bullet}$ to the kernel of $d\rho^{0}\wedge\cdot$ in a neighbourhood of $p$, we may write 
\begin{equation}
\rho^{2} = d\rho^{0}\wedge \wt\Omega + \beta,
\end{equation}
where $\wt\Omega,\beta$ are uniquely determined smooth forms in $C^{\bullet}$ and $\wt\Omega$ is nonzero along $D$.  
Away from $D$, $\rho$ is of type zero, so that $\rho = \rho^{0}e^{\rho^{2}/\rho^{0}}$, and the integrability condition implies that $d(\rho^{2}/\rho^{0}) = H$.  This implies that 
\begin{equation}
d\rho^{0}\wedge\rho^{2} = \rho^{0}d\rho^{2} - (\rho^{0})^{2}H,
\end{equation}
which must then hold on all of $M$ by continuity. As a result, we conclude that $\beta=\rho^{0}\wt\sigma$ for a smooth 2-form $\wt\sigma$, and consequently $\rho^{2}/\rho^{0}$ and $\rho/\rho^{0}$ are well-defined logarithmic forms for the divisor $D$.  

The reduction of complex structure is then performed by taking the residue of $\rho/\rho^{0}$, a smooth form on $D$ given by 
\begin{equation}
\rho_{D} = \Res(\rho/\rho^{0}) = \Omega\wedge e^{\sigma},
\end{equation}
where $\Omega =\iota_{D}^{*}\wt\Omega$ and $\sigma = \iota_{D}^{*}\wt\sigma$, for $\iota_{D}:D\to M$ the inclusion.  The logarithmic form $\rho/\rho^{0}$ is independent of the choice of trivialization $\rho$, and so its residue is as well. Finally, the residue is closed with respect to $d + \iota_{D}^{*}H\wedge\cdot$, since $\rho/\rho^{0}$ is closed for $d^{H}$. 

To obtain the holomorphic structure on the normal bundle, we use the fact that the transversality of $s$ implies that $ds|_{D}:N\to K^{*}|_{D}$ is an isomorphism, allowing us to transport the generalized holomorphic structure on $K^{*}$ to $N$.   If $\rho$ is a local trivialization for $K$, then $ds|_{D}^{*}$ takes this to the local trivialization of $N^{*}$ given by $\eta = d\rho^{0}|_{D}$. As explained in Section~\ref{canbun}, the generalized holomorphic structure on $K$ is given locally by $\delbar \rho = F\otimes\rho$; pulling back to $D$, we define the generalized holomorphic structure on the conormal bundle by $\delbar \eta = \pi(F)\otimes \eta$, where $\pi$ is the projection in~\eqref{eq:submanifold}, where we note that $F$ is orthogonal to ${N^{*}_{\CC}}$, so $\pi(F)$ lies in the $-i$-eigenbundle of the reduced generalized complex structure on $D$.  Explicitly, $\pi(F) = X + \iota_{D}^{*}\xi$, and we may verify that since $i_{X}\rho^{2} = d\rho^{0}$ along $D$, it follows that $i_{X}\Omega = -1$. This implies that the transverse Higgs field $[X]$ of $\delbar$ evaluates to $-1$ on $\Omega$, and so the generalized holomorphic structure on the dual bundle $N$ has opposite transverse Higgs field, evaluating to $+1$ on $\Omega$, as required.
\end{proof}

%
%

Because the generalized holomorphic structure on the normal bundle $N$ has transverse Higgs field $\phi\in C^{\infty}(D, \nu_{1,0})$ which satisfies $i_{\phi}\Omega=1$, it follows that the induced $\CC^{*}$-invariant generalized complex structure on $\tot(N)$ is symplectic away from the zero section. This can be seen by writing the structure on the total space as we did in~\eqref{localKCY}. Let $n$ be a local trivialization for $N$ and let $\delbar n = (X+\xi)\otimes n$, using the generalized holomorphic structure on $N$ defined in Theorem~\ref{gholstrnd}.  Here $F=X+\xi\in \ol L_{D}$, so that $X$ is the transverse Higgs field of $\delbar$. Let $\tau$ be the fibrewise linear coordinate on $\tot(N)$ corresponding to $n$.  Then the generalized complex structure on $\tot(N)$ may be defined locally by the form
\begin{equation}\label{locallindef}
e^{\tau\del_{\tau}\wedge F} d\tau \wedge \rho_{D} = (\tau + (d\tau + \tau (i_{X}\sigma +\xi ))\wedge\Omega)\wedge e^{\sigma},
\end{equation}
In the expression above, $\sigma$ is only defined modulo the ideal of $\Omega$, and in fact we may choose it such that $i_{X}\sigma + \xi = 0$, simplifying the above expression.  
In any case, the component of degree zero of~\eqref{locallindef} vanishes transversally, demonstrating that $\tot(N)$ has a stable generalized complex structure with anticanonical divisor given by the zero section.

\begin{definition}\label{linearizedef}
The natural $\CC^{*}$--invariant stable generalized complex structure~\eqref{locallindef} inherited by the total space of the normal bundle of $D$ is called the \emph{linearization} of the stable generalized complex manifold along $D$.
\end{definition}

We conclude this section by showing that there are no implicit constraints on the data determining the linearization; any type 1 generalized Calabi-Yau structure and any holomorphic line bundle over it may be realized as the linearization of a stable generalized complex manifold, as long as the transverse Higgs field pairs nontrivially with the Calabi-Yau form.
Since we show in Section~\ref{lindeglocsec} that the generalized complex structure in a neighbourhood of $D$ is completely determined by the linearization, the following result provides a local normal form for stable generalized complex manifolds about their anticanonical divisors.

\begin{proposition}\label{plus1boot}
Let $(D,\Omega\wedge e^{\sigma})$ be a type 1 generalized Calabi-Yau manifold and $N$
a generalized holomorphic line bundle over $D$ whose transverse Higgs field pairs nontrivially with $\Omega$. Then the total space $\tot(N)$ inherits a $\CC^{*}$-invariant stable generalized complex structure with anticanonical divisor given by the zero section.
\end{proposition}
\begin{proof}
Let $n$ be a local trivialization for $N$ and let $\delbar n = F\otimes n$, where $F\in \ol L_{D}$ is the algebroid connection 1-form and $X=\pi_{T_{\CC}}(F)$ is the transverse Higgs field. We may rescale $\Omega$ by a constant so that $i_{X}\Omega = 1$, since $X$ has constant pairing with $\Omega$ and this is nonzero by assumption. Let $\tau$ be the fibrewise linear coordinate on $\tot(N)$ corresponding to $n$.  Then the generalized complex structure on $\tot(N)$ may be written locally just as in Equation~\ref{locallindef},
demonstrating that it is a stable generalized complex structure with anticanonical divisor given by the zero section.
\end{proof}

\begin{example}
For $M$ 4-dimensional, $D$ has dimension 2, and so the generalized Calabi-Yau structure inherited by $D$ is a usual Calabi-Yau complex structure, implying that each component of $D$ must be a complex curve of genus 1 with a distinguished holomorphic 1-form $\Omega$. Furthermore, the generalized holomorphic structure induced on $N$ gives it the structure of a holomorphic line bundle over $D$ equipped with a transverse Higgs field, which in this case is simply the holomorphic vector field on $D$ dual to $\Omega$.  In this way, we recover the results of~\cite{MR2574746} characterizing the complex locus of a stable generalized complex 4-manifold.
\end{example}

\subsection{Constructions of stable structures}\label{exst}
 
If $(U,\delbar)$ is a generalized holomorphic line bundle over the stable generalized complex manifold $(M,\JJ)$, then we obtain via~\eqref{localKCY} a natural $\CC^{*}$--invariant generalized complex structure on the total space of $U$. We now describe how to construct a stable generalized complex structure on the same total space but with the zero section removed.  We discovered this construction by applying T-duality~\cite{MR2648900} to the canonical structure on $\tot(U)$.

Choose a Hermitian metric on the bundle $U$, and let $D:C^{\infty}(U)\to C^{\infty}(\TT M \otimes U)$ be the unique unitary generalized connection whose component along $L\subset\TT_{\CC}M$ coincides with $\delbar$ (see~\cite{MR2681704} for a detailed discussion of generalized connections).  If $u$ is a local unitary trivialization of $U$, and $\delbar u = \alpha\otimes u$ for $\alpha\in L^{*}\cong \ol L$, we have that 
\begin{equation}
Du = i\sA \otimes u = (\alpha-\ol\alpha)\otimes u,
\end{equation}
where $\sA$ is the generalized connection 1-form, a real local section of $\TT M$. If $\rho$ is a local trivialization for the canonical bundle of $\JJ$, and $(r,\theta)$ are fibrewise polar coordinates associated to $u$, then we construct the following form on the complement of the zero section in the total space of $U$:
\begin{equation}\label{bootstabgc}
e^{id\log r \wedge (d\theta - \sA)}\rho = \rho - id\log r\wedge(d\theta \wedge \rho -\sA \cdot \rho).
\end{equation}
This form is nondegenerate, independent of the choice of unitary trivialization and its integrability follows directly from the fact that $d_{L}\alpha = 0$. Its component of degree zero coincides with that of $\rho$, and hence defines a stable generalized complex structure on $\tot^{*}(U)$.  The form~\eqref{bootstabgc} is manifestly invariant by constant rescaling of the fibers; we may therefore take a $\ZZ$ quotient of $\tot^{*}(U)$ to form a torus bundle over $M$.  We summarize the construction as follows.
\begin{proposition}\label{bootstrapsgc}
Given a Hermitian metric on a generalized holomorphic line bundle $U$ over the stable generalized complex manifold $M$, the total space of $U$ inherits a $\CC^{*}$--invariant stable generalized complex structure away from the zero section given by~\eqref{bootstabgc}.  Quotienting by a subgroup $\ZZ\subset \CC^{*}$, we obtain a stable generalized complex $T^{2}$-bundle over $M$.
\end{proposition}
  
\begin{example}
On a symplectic manifold $(M,\omega)$, generalized holomorphic bundles are simply bundles equipped with complex flat connections. Let $U$ be a Hermitian line bundle with flat connection $\nabla$. If $u$ is a local unitary section then $\nabla u = \tfrac{1}{2}A\otimes u$ for $A$ a closed complex 1-form, and $\sA = \omega^{-1}(\Re(A)) + 2\Im(A)$.  The local expression for the stable structure~\eqref{bootstabgc} is then 
\begin{equation}
\exp(d\log r \wedge \Re(A) + i(d\log r \wedge (d\theta -\Im(A)) + \omega)) 
\end{equation}
\end{example}
%
%

\begin{example}\label{ex:S^1 x S^5} In this example we show that $S^1 \times S^5$ admits stable \gcss\ (generically of type 0) as well as structures of type 1, 2 and 3.

Equip $\CC P^{2}$ with the stable generalized complex structure obtained by deforming the complex structure by a holomorphic Poisson structure $\beta\in H^{0}(\CC P^{2},\wedge^{2}\shf{T})$ whose zero locus is  a smooth cubic curve $E$.  The $\OO(1)$ line bundle then has a canonical generalized holomorphic structure since its cube is the canonical line bundle.  Equipping it with the Fubini-Study metric, Proposition~\ref{bootstrapsgc} provides a stable generalized complex structure on the $T^{2}$-bundle $\tot^{*}(\OO(1))/\ZZ\cong S^{1}\times S^{5}$.  The anticanonical divisor $D\subset S^{1}\times S^{5}$ is then a symplectic fibre bundle of tori over the cubic curve $E$. The generalized Calabi-Yau structure on $D$ is analogous to that described in Example~\ref{kodthur} on the Kodaira--Thurston manifold.

It is interesting to note that, in addition to the stable generalized complex structure constructed above, $S^{1}\times S^{5}$ admits a generalized complex structure of type 1, by the construction of Lemma~\ref{totspaceksgc} adapted to $\OO(1)$, a fractional multiple of the canonical bundle of $\CC P^{2}$.  Furthermore, it admits a structure of constant type 2, by the following observation. 
Viewing $S^{5}$ as the unitary frame bundle of $\OO(1)$ over $\CC P^{2}$, it has a connection form $\theta_{1}$ with curvature of type $(1,1)$.  If  $\theta_{2}$ is a volume form on $S^{1}$, then 
the following defines the canonical line bundle of a generalized complex structure of type 2 on $S^{1}\times S^{5}$:
\begin{equation}
e^{i\theta_{1}\wedge\theta_{2}}\wedge \Omega^{2,0}.
\end{equation}
Finally, note that we may even endow  $S^{1}\times S^{5}$ with 
a type 3 structure, as it is in the Calabi-Eckmann family of compact complex manifolds.  
\end{example}

\section{Log symplectic forms}\label{equivloggc}

A stable generalized complex structure, since it is equivalent via a B-field $b$ to  a symplectic structure $\omega$ away from the anticanonical divisor $D$, may be viewed as a symplectic form which is singular along $D$.  In fact, in the proof of Theorem~\ref{gholstrnd}, we observed that the type of singularity is such that $\sigma = b+i\omega$ defines a logarithmic form in the sense of Section~\ref{logderhamcxsect}.  In this section we make precise the relationship between stable structures and logarithmic forms, and we use this relationship to give a local period map for stable generalized complex structures.

Let $D=(U,s)$ be a complex divisor on $M$ and let $\iota, a$ be the natural Lie algebroid morphisms (each an inclusion of sheaves) between the elliptic, logarithmic, and usual tangent bundles, as follows:
\begin{equation}\label{anciota}
\xymatrix{
T(-\log|D|)\otimes\CC \ar[r]^-{\iota} &  T(-\log D)\ar[r]^-{a} & T_{\CC}M
}.
\end{equation}
\begin{definition}
Let $H$ be a real closed 3-form and $D$ a complex divisor on the manifold $M$. A 
\emph{complex log symplectic form} is a logarithmic 2-form 
$\sigma \in \Omega^2(M,\log D)$ such that 
\begin{equation}
d\sigma = a^* H
\end{equation}
and such that the elliptic form $\iota^*\sigma = b + i\omega$ has nondegenerate imaginary part.\footnote{Note that a holomorphic log symplectic form for a reduced divisor $D$ is a special case of our notion of complex log symplectic form, which does not require the underlying manifold to be complex.}  
\end{definition}
The nondegeneracy condition on $\omega\in\Omega^2(M,\log |D|)$ is  that the induced skew map
\begin{equation}
\omega:T(-\log |D|)\to T(-\log |D|)^*
\end{equation}
is an isomorphism, and the integrability condition implies that $d\omega=0$, rendering $\omega$ into what we call an \emph{elliptic symplectic form} (see Section~\ref{elllogsymp}).  The notion of equivalence for complex log symplectic structures is the same as that for generalized complex structures: we say that $(M,H,D,\sigma)$ is equivalent to $(M',H',D',\sigma')$ when there is a diffeomorphism of divisors $\psi:(M,D)\to (M',D')$ in the sense of Section~\ref{isodiffdiv} as well as a real smooth 2-form $b\in\Omega^{2}(M,\RR)$ such that $\psi^{*}H' = H + db$ and 
\[
\psi^{*}\sigma' = \sigma + b.
\]
Given a complex log symplectic form $\sigma$, its graph defines a subbundle $\Gamma_{\sigma}\subset T(-\log D)\oplus T^{*}(\log D)$, and the anchor map $a$ may be used to push this forward to a subbundle $L\subset \TT_{\CC}M$: 
\begin{equation}\label{dirpushredsigma}
L = a_{*}\Gamma_{\sigma} = \{a(X) + \eta\ |\ X + a^{*}\eta\in \Gamma_{\sigma}\}.
\end{equation}
We now show that this defines the $+i$-eigenbundle of a stable generalized complex structure  
and that this establishes an isomorphism of categories between stable generalized complex structures and complex log symplectic structures on $(M,H)$.  

\begin{theorem}\label{equivsgcils}
There is a canonical bijection between stable generalized complex structures $\JJ$ on $(M,H)$ and complex log symplectic forms 
$\sigma$ for the anticanonical divisor $D= (K^{*},s)$, defined by the relation $L_{\JJ} = a_{*}\Gamma_{\sigma}$ between the $+i$-eigenbundle of $\JJ$ and the graph of $\sigma$.
In this correspondence, any local trivialization $\rho$ of the canonical line bundle satisfies the identity
\begin{equation}\label{actuallog}
a^{*}\rho = \rho^{0}e^{\sigma},
\end{equation}
so that $\rho/\rho^{0}$ extends over all of $M$ to a section of $\Omega^{\bullet}(\log D)$.
\end{theorem}
\begin{proof}
Let $\JJ$ be a stable generalized complex structure on $(M,H)$.  Then by Theorem~\ref{gcytotk} it determines a $\CC^{*}$-invariant generalized Calabi-Yau structure $\varpi$ on $\tot(K)$, the total space of its canonical line bundle, which is of type 1 on $\tot^{*}(K)$, the complement of the zero section. Let $L_{K}$ be the $+i$-eigenbundle of this Calabi-Yau structure.  By Theorem~\ref{asigma}, $L_{K}$ is completely determined by its tangent projection $A\subset T_{\CC}(\tot^{*}(K))$, together with the 2-form $\sigma\in \Omega^{2}_{A}$, which has nondegenerate imaginary part on $A\cap\ol A$.  

But $A$ coincides with $\At_{\CC}(K,s)$, the subbundle of the complexified Atiyah bundle of $K$ consisting of vector fields fixing the fibrewise linear function $s$.  By Proposition~\ref{geolog}, $A$ projects surjectively onto $T(-\log D)$ under the derivative of the bundle projection $\pi:K\to M$, with kernel generated by the conjugate Euler vector field $\ol E = \ol \tau\del_{\ol\tau}$:
\begin{equation}
\xymatrix{
0\ar[r] & 
\left< \ol E\right> \ar[r]  &
A\ar[r]^-{\pi_{*}} & 
T(-\log D)\ar[r] &
0
}
\end{equation}

But $\ol E$ annihilates $\varpi$ and so it lies in the kernel of $\sigma$, implying that $\sigma$ is basic for $\pi_{*}$, defining a logarithmic 2-form as required. The nondegeneracy condition is immediate from the fact that $\pi_{*}$  is an isomorphism from $A\cap \ol A$ to the complexification of $T(-\log |D|)$, and integrability is inherited from $d^{H}\varpi = 0$.   The relation $L_{\JJ}= a_{*}\Gamma_{\sigma}$ then follows from Proposition~\ref{pushforwardrelation}.

This argument is reversible: if $\sigma$ is a complex log symplectic form for the complex divisor $(K^{*},s)$, we may pull it back via $\pi$ to a 2-form on $A = \ker{ds}$.  The resulting Calabi-Yau structure $\varpi = ds\wedge e^{\sigma}$ on $\tot^{*}(K)$ may the be reduced along $\pi$ to define a generalized complex structure $\JJ$ on $M$.  The reduction is done as follows:  the action of $\CC^{*}$ is generated by the complex Euler vector field $E$ and its conjugate. Since $\ol E$ annihilates $\varpi$, the form $\rho_{K}=i_{E}\varpi$ is nowhere vanishing on $\tot^{*}(K)$, and satisfies 
\begin{equation}
i_{E}\rho_{K} = 0,\qquad L_{E}\rho_{K} = \rho_{K},
\end{equation}
and so defines an inclusion $\rho_{K}: K\hookrightarrow \wedge^{\bullet}T^{*}M$ of vector bundles over $M$; since $s$ is transverse, $\JJ$ is stable, as required. 

To verify identity~\eqref{actuallog}, if $\rho$ is a local trivialization for $K$ with dual trivialization $\tau$, can write $s = \rho^{0}\tau$ and $\varpi = d(\rho^{0}\tau) \wedge e^{\sigma}$.    Then $\rho_{K}=\tau \rho^{0}e^{\sigma}$ takes $\rho$ to the smooth differential form $\rho^{0}e^{\sigma}$, as required.
\end{proof}

\subsection{Elliptic log symplectic structures}\label{elllogsymp}

In the above equivalence between stable generalized complex structures $\JJ$ and logarithmic symplectic forms $\sigma = b+i\omega$, the imaginary part $\omega$ is a closed, nondegenerate elliptic 2-form for the elliptic divisor defined by $(K^{*}\otimes \ol K^{*}, s\otimes \ol s)$, where $s$ is the anticanonical section of $\JJ$. 
Recall that $\omega$ coincides with the inverse of the real Poisson structure $Q$ underlying $\JJ$. In fact, $Q$ itself determines the elliptic divisor, since 
the Chevalley pairing on differential forms restricts to an isomorphism
\begin{equation}
\xymatrix{
K\otimes \ol K\ar[r]^-{\cong} & \wedge^{2n} T^{*}_{\CC}M
}
\end{equation}
which takes $e^{\sigma}\otimes e^{\ol\sigma}$ to the top degree component of $e^{\sigma - \ol\sigma}$, namely to $(2i)^{n}\omega^{n}$.  This implies that $s\otimes\ol s$ is taken to $(2i)^{-n} Q^{n}$, giving the natural isomorphism
\begin{equation}
(K^{*}\otimes \ol K^{*}, s\otimes \ol s) \cong (\wedge^{2n}TM,\wedge^{n}Q).
\end{equation}
In this section, we show that the forgetful map taking $\JJ$ to its underlying real Poisson structure $Q$ defines a bijection between gauge equivalence classes of stable generalized complex structures (integrable with respect to any 3-form) and a certain class of Poisson structures, defined as follows.
\begin{definition}
A Poisson structure $Q$ is of elliptic log symplectic type when its Pfaffian defines an elliptic divisor $(\wedge^{2n} TM, \wedge^{n}Q)$. 
\end{definition}
To justify the terminology, we have the following equivalent characterization:
\begin{lemma}
A Poisson structure $Q$ is of elliptic log symplectic type if and only if its inverse $\omega = Q^{-1}$ is a closed, nondegenerate elliptic 2-form.
\end{lemma}
\begin{proof}
Let $Q$ be an elliptic log symplectic Poisson structure. Since $L_{X}(\wedge^{n}Q) = 0$ for any Hamiltonian vector field $X$, it follows immediately that $Q$ lifts to a section $\wt Q$ of $\wedge^{2}T(-\log |D|)$. Taking the top exterior power of the relation $Q = \anc \wt Q\anc^{*}$, where $\anc:T(-\log|D|)\to TM$ is the anchor, we see that $\wt Q$ is invertible, defining the elliptic log symplectic form
\begin{equation}
\omega = \wt Q^{-1}\in \Omega^{2}(\log |D|),\qquad d\omega = 0.
\end{equation}
For the reverse implication, let $D=(R,q)$ be an elliptic divisor and let $\omega$ be an elliptic log symplectic form.  We use the fact that the determinant of the algebroid anchor, a section of $\wedge^{2n}T^{*}(\log |D|)\otimes \wedge^{2n}TM$, lifts to an isomorphism:
\begin{equation}
\xymatrix{
R\ar[rr]^-{q^{-1}\det \anc}_-{\cong}& & \wedge^{2n}T^{*}(\log |D|)\otimes \wedge^{2n}TM.
}
\end{equation}
Since $\wedge^{n}\omega^{-1}$ trivializes $\wedge^{2n}T(-\log |D|)$ and is taken to $\wedge^{n}Q$ by $\det \anc$, we obtain an isomorphism between $(R,q)$ and $(\wedge^{2n}TM,  \wedge^{n}Q)$, proving that $Q$ is of elliptic log symplectic type.
\end{proof}

\begin{remark}
Note that $T(-\log |D|)$ is isomorphic to $TM$ away from $D$, which has real codimension 2; this implies that $M$ is oriented by the choice of an elliptic symplectic form.  
Further, the Hessian~\eqref{hessian} of $\wedge^{n}Q$ is a section over $D$ of $S^{2}N^{*}\otimes\wedge^{2n} TM$, and has determinant which trivializes the \emph{square} of the bundle $\kk = \wedge^{2}N^{*}\otimes\wedge^{2n} TM = \wedge^{2n-2} TD$.  The elliptic residue $\Res_{q}\omega\in \Omega^{0}(D,\kk^{*})$, if nonzero, is then a constant volume form on $D$ with respect to this trivialization, and defines an orientation on $D$.  If $\Res_{q}(\omega)=0$, then $D$ need not be orientable, as the following example shows.
\end{remark}

\begin{example}
Let $E$ be the elliptic curve $\CC/\ZZ(1,i)$ with standard coordinate $z$ and consider the holomorphic Poisson structure $\beta = w\del_{w}\wedge \del_{z}$ on $\CC P^{1}\times E$.  The $\ZZ_{2}$ action $\tau: (w,z)\mapsto (\ol w, \ol z + \tfrac{1}{2})$ acts via $\tau_{*}(\beta) = -\ol\beta$ and so preserves the imaginary part of $\beta$, an elliptic symplectic structure with vanishing elliptic residue.  The quotient $(\CC P^{1}\times E)/\ZZ_{2}$ then inherits an elliptic symplectic structure with degeneracy locus given by the pair of Klein bottles $\{w^{\pm 1}=0\}$. 
\end{example}

We now state the main result, a relative of Theorem~\ref{equivsgcils} in which $H$ is allowed to vary and \emph{gauge equivalence classes} of generalized complex structures are identified with elliptic symplectic structures; recall that $(\JJ,H)$ is gauge equivalent to $(\JJ',H')$ when there is a 2-form $b\in\Omega^{2}(M,\RR)$ such that $H' = H + db$ and $\JJ' = e^{b}\JJ e^{-b}$.

\begin{theorem}\label{correspellj}
Fix the smooth manifold $M$.  The forgetful map which takes the pair $(\JJ,H)$ of a stable generalized complex structure integrable with respect to the closed 3-form $H$ to the pair $(Q,\mathfrak{o})$, where $Q$ is the real Poisson structure of $\JJ$ and $\mathfrak{o}$ is the co-orientation of the anticanonical divisor $D$ of $\JJ$, defines a bijection between gauge equivalence classes of stable generalized complex structures and elliptic log symplectic structures  with vanishing elliptic residue and co-oriented degeneracy locus.  

The map $(\JJ,H)\mapsto (Q,\mathfrak{o})$ is equivariant for the action of the diffeomorphism group and commutes with the natural maps to $H^{3}(M,\RR)$; that is, the radial residue $\Res_{r}[Q^{-1}]\in H^{1}(D,\RR)$ is mapped to the class $[H]\in H^{3}(M,\RR)$ by the Thom--Gysin pushforward map associated to the co-orientation $\mathfrak{o}$ of the inclusion $j:D\hookrightarrow M$.
\end{theorem}
\begin{proof}
By Theorem~\ref{equivsgcils}, $\omega = Q^{-1}$ is the imaginary part of a complex log symplectic form $\sigma$, and by Proposition~\ref{isomzerores}, the elliptic residue of this form must vanish.  Furthermore, $\omega$ is invariant under gauge transformations of $\JJ$: if $\pi:\TT M\to TM$ is the projection, then for any 2-form $b\in \Omega^{2}(M,\RR)$, we have  $\pi\circ e^{b}=\pi$, so $Q = \pi\JJ   \pi^{*} = \pi e^{b} \JJ e^{-b}\pi^{*}$. Of course the co-orientation of the anticanonical divisor $D$ is unaffected by the action of the connected group of gauge transformations, so we have finally that the forward map $\JJ\mapsto (Q,\mathfrak{o})$ is well-defined.  

We see the map is surjective as follows: given an elliptic log symplectic form $\omega$ with zero elliptic residue, we use the coorientation $\mathfrak{o}$ to apply Proposition~\ref{isomzerores}, which implies that there exists $\sigma\in \Omega^{\bullet}(\log D)$ such that $\Im^{*}(\sigma)= \omega$.  Since $d\omega=0$ by assumption, $d\sigma$ has vanishing imaginary part, and by the same Proposition, it must be a smooth real 3-form $H$, proving that $\sigma$ defines a complex log symplectic form, which is the required stable generalized complex structure by Theorem~\ref{equivsgcils}.  

Finally, the map is injective on gauge equivalence classes: if $\JJ$ and $\JJ'$ give rise to the same elliptic log symplectic form $\omega$, this means that their corresponding complex log symplectic forms $\sigma, \sigma'$ satisfy $\Im^{*}(\sigma) = \Im^{*}(\sigma')$. By Proposition~\ref{isomzerores}, this means $\sigma'=\sigma + b$ for $b$ a real smooth 2-form, implying that $\JJ' = e^{b}\JJ e^{-b}$, as needed.

Diffeomorphism equivariance is manifest from the description of the map. For compatibility with the maps to $H^{3}(M,\RR)$, note that each gauge equivalence class $[(\JJ,H)]$ has a well-defined class $[H]$, and by the long exact sequence~\eqref{toplogseq} this class coincides with the image of $[\omega] = [Q^{-1}]$ under the connecting homomorphism $H^{2}_{0}(\log |D|)\to H^{3}(M,\RR)$, which by sequence~\eqref{relpart} is the image of the radial residue $\Res_{r}[\omega]\in H^{1}(D,\RR)$ under the Thom--Gysin pushforward $j_{*}:H^{1}(D)\to H^{3}(M)$, which is well-defined by the specified co-orientation $\mathfrak{o}$.
\end{proof}

\begin{remark}
Theorem~\ref{correspellj} may be viewed as the statement that the set of pairs $(\JJ,H)$, where $\JJ$ is a stable generalized complex structure integrable with respect to the closed 3-form $H$, forms a principal bundle over the set of co-oriented elliptic symplectic structures, where the principal structure group is the abelian group of real smooth 2-forms, where $b\in\Omega^{2}(M,\RR)$ acts via 
\begin{equation}
b \cdot (\JJ,H) = (e^{b}\JJ e^{-b}, H+db).
\end{equation}
This principal bundle is twisted equivariant for the action of diffeomorphisms, in the sense that for any diffeomorphism $\varphi$, we have 
\begin{equation}
\varphi^{*} ( b \cdot (\JJ, H))  = (\varphi^{*}b) \cdot \varphi^{*}(\JJ,H).
\end{equation}
\end{remark}

\subsection{The period map for fixed 3-form flux}\label{perfixH}

Having established the equivalence between stable generalized complex structures and complex log symplectic forms in the previous section, we now observe that in analogy with usual symplectic structures, we may define a period map which gives a complete description of the local moduli space of deformations of these structures.  
We shall consider two period maps associated to stable generalized complex structures.  The difference between them is whether or not the 3-form $H$ is fixed in the definition of a family of structures and in the definition of equivalence for such families. In this section we treat the case with $H$ fixed.

\begin{definition}\label{defsdef}
Let $\JJ$ be a generalized complex structure on $(M,H)$.  A 
\emph{deformation} of $\JJ$ is defined to be a smoothly varying family of structures $\JJ_{s}$, $s\in[0,1]$, each integrable with respect to the fixed 3-form $H$ and such that $\JJ_{0}=\JJ$. 
Two such deformations $\JJ_{s},\JJ'_{s}$ of $\JJ$ are said to be \emph{equivalent} when there is a family of sections $e_{s}$ of $\TT M$, allowed to be time-dependent for each $s$, whose associated exact time-1 flow $\Phi_{1}^{e_{s}}$ (Definition 
\ref{exactflow}) takes $\JJ_{s}$ to $\JJ'_{s}$.
\end{definition}

The flow $\Phi_{t}(X,b)$ of a pair $(X, b)$ consisting of a time-dependent vector field $X$ and 2-form $b\in\Omega^{2}(M,\RR)$ is a smooth family of automorphisms of $\TT M$ defined by the initial value problem\footnote{We use conventions for flows in which $\tfrac{d}{dt}{(\varphi_{t})_{*}} = -L_{X_{t}}\circ(\varphi_{t})_{*}$ and $\tfrac{d}{dt}{(\varphi_{t})^{*}} = L_{X}\circ (\varphi_{t})^{*}$.} 
\begin{equation}
\tfrac{d}{dt} \Phi_{t} = -\LL_{(X,b)}\circ \Phi_{t},\qquad \Phi_{0}=\id,
\end{equation}
where $\LL_{(X,b)}$ acts on sections of $\TT M$ taking $Y+\eta$ to $L_{X}(Y+\eta) - i_{Y}b$. The solution may be written $\Phi_{t} = \varphi_{t}e^{B_{t}}$, where $\varphi_{t}$ is the time-$t$ flow of $X$, acting on $\TT M$ via $(\varphi_{t})_{*}\oplus (\varphi_{t}^{*})^{-1}$, and $B_{t}$ is given by 
\begin{equation}\label{thebequ}
B_{t} = \int_{0}^{t}\varphi_{s}^{*}b_{s}\ ds.
\end{equation}
If we equip $\TT M$ with the $H$-twisted Courant bracket, then $\Phi_{t}$ takes it to the $H_{t}$-twisted Courant bracket, where 
$H_{t}$ satisfies the initial value problem 
\begin{equation}\label{derivh}
\tfrac{d}{dt} H_{t} = - L_{\varphi^{t}_{*}X}H_{t} + db,\qquad H_{0}=H,
\end{equation}
which has solution $H_{t}$ defined by the equation 
\begin{equation}\label{hchange}
\varphi_{t}^{*}H_{t} = H + dB_{t}.
\end{equation}
\begin{definition}\label{exactflow}
The \emph{exact flow} $\Phi_{t}^{e}$ associated to the section $e = X+\xi$ of $\TT M$ is the flow $\Phi_{t}(X,b)$ defined above for $b$ given by 
\begin{equation}\label{exactb}
b = (\varphi^{t}_{*}X)\lrcorner\,H + d\xi.
\end{equation}
It is an automorphism of $\TT M$ preserving the $H$-Courant bracket.
\end{definition}
As automorphisms of $\TT M$, flows $\Phi_{t}(X,b)$ operate on generalized complex structures by conjugation. Such flows also act upon complex log symplectic structures, taking $\sigma$ to $\sigma_{t}$, where 
\begin{equation}\label{actonsigma}
\sigma_{t} = (\varphi_{t}^{*})^{-1}(\sigma + B_{t}), 
\end{equation}
the result of operating by $\varphi_{t} e^{B_{t}}$ on the graph of $\sigma$. Indeed, if $d\sigma = H$, then from~\eqref{hchange} we obtain that $d\sigma_{t} = H_{t}$.

Applying the equivalence from Theorem~\ref{equivsgcils}, we obtain the corresponding notion of deformation for complex log symplectic forms, taking care to include the possible variation of the complex divisor: 

\begin{definition}
A deformation of the pair $(D,\sigma)$ is a smoothly varying family $(D_{s}, \sigma_{s})$, $s\in[0,1]$, of complex divisors $D_{s}$ and complex log symplectic forms $\sigma_{s}\in\Omega^{2}(\log D_{s})$ such that $(D_{0},\sigma_{0}) = (D,\sigma)$.  Denote the set of deformations of $(D,\sigma)$ by $\Def(D,\sigma)$.  When a deformation of $(D,\sigma)$ is such that $D_{s} = D$ for all $s$, we say that it is a deformation of $\sigma$ only; denote the set of deformations of $\sigma$ by $\Def(\sigma)$.

Two deformations $(D_{s},\sigma_{s})$, $(D'_{s},\sigma'_{s})$ are equivalent when there is a family of (possibly time-dependent) sections $e_{s}$ of $\TT M$ whose associated time-1 flows 
$\Phi_{1}^{e_{s}}$ take $(D_{s},\sigma_{s})$ to $(D'_{s},\sigma'_{s})$ via~\eqref{actonsigma}.
For deformations of $\sigma$ only, we say that two deformations are equivalent when $e_{s}$ has vector component in $T(-\log |D_{s}|)$ for all $s\in[0,1]$.
\end{definition}

The main observation which makes it possible to define a period map for stable generalized complex structures is that any deformation of complex divisors may be rectified: 
By Lemma~\ref{mosdiv}, each deformation $(D_{s},\sigma_{s})$ is equivalent to a deformation with fixed divisor, and this defines a canonical bijection between equivalence classes of deformations of $(D,\sigma)$ and equivalence classes of deformations of $\sigma$ only:
\begin{equation}
\xymatrix{
 \Def(D,\sigma)/\!\sim\ar[rr]^-{\cong}& & \Def(\sigma)/\!\sim
}
\end{equation}
 
\begin{lemma}\label{rectlemma1}
Let $M$ be compact. The rectification of divisors given by Lemma~\ref{mosdiv} defines a canonical bijection between equivalence classes of deformations of $(D,\sigma)$ and equivalence classes of deformations of $\sigma$ with $D$ fixed.
\end{lemma}

As a result, the deformation theory of the stable generalized complex structure $\JJ$ is equivalent to the theory of deformations of the corresponding complex log symplectic structure $\sigma$, keeping the anticanonical divisor fixed.  We now characterize these completely by defining a period map.

\begin{definition}
Let $\sigma_{s}$, $s\in[0,1]$, be a deformation of the complex log symplectic structure $\sigma$ with $(M,H,D)$ fixed. Its period is defined to be the path given by
\begin{equation}
s\mapsto \sP(\sigma_{s}) = [\sigma_{s}-\sigma] \in H^{2}(M\backslash D,\CC),
\end{equation}
where we use the identification of $H^{2}(M,\log D)$ with $H^{2}(M\backslash D,\CC)$ in Theorem~\ref{logdcoh}.
\end{definition} 

\begin{theorem}\label{periodmapcx}
Let $\sigma$ be a complex log symplectic structure on $(M,H,D)$ with $M$ compact.  The period map defines a canonical bijection between germs of deformations of $\sigma$ up to equivalence  and germs of smooth paths beginning at the origin in $H^{2}(M\backslash D, \CC)$.
\end{theorem} 
\begin{proof}
First we show that the period map descends to equivalence classes. Let $\sigma_{s}, \sigma'_{s}$ be equivalent deformations of $\sigma$.  To show that $\sP(\sigma_{s}) = \sP(\sigma'_{s})$, it suffices to show that $[\sigma'_{s} - \sigma_{s}]=0$ for each $s$.   Fix $s$ and let $\sigma_{0}=\sigma_{s}$ and $\sigma_{1}=\sigma'_{s}$. Equivalence of deformations implies that there is a time-dependent elliptic vector field $X$ and 1-form $\xi$ which determine an interpolating family $\sigma_{t}$ given by 
$\sigma_{t} = (\varphi_{t}^{*})^{-1}(\sigma_{0} + B_{t})$,
where $\varphi_{t}$ is the $t$-flow of $X$ and $B_{t}$ is given by~\eqref{exactb}.  To show $[\sigma_{1}-\sigma_{0}]=0$, we prove the infinitesimal version, that $[\del_{t}\sigma_{t}]=0$:
\begin{equation}
\begin{aligned}
\del_{t}\sigma_{t} &= -L_{\varphi^{t}_{*}X}\sigma_{t} + b_{t}\\
 &= - d (\varphi^{t}_{*}X\lrcorner\, \sigma_{t}
 + \xi),
\end{aligned}
\end{equation}
where we have used that $d\sigma_{t}=H$ for all $t$ to obtain the result. 

The main difficulty is to show the period map is injective on germs: let $\sigma_{s}, \sigma'_{s}$ be two deformations of $\sigma$ with the same period, i.e. $[\sigma_{s}-\sigma] = [\sigma'_{s}-\sigma]$ for all $s$.  We aim to show that $\sigma_{s},\sigma'_{s}$ have equivalent germs about $s=0$, using a 2-step Moser method in families. 
\\

\noindent{\bf Step 1.} We begin by finding an equivalence between the family $\sigma'_{s}$ and a family $\sigma''_{s}$ whose imaginary part coincides with that of $\sigma_{s}$.  
Decompose $\sigma_{s}= b + i\omega_{s}$ and $\sigma'_{s} = b'_{s} + i\omega'_{s}$ into real and imaginary parts.  Since $[\sigma'_{s}-\sigma_{s}]=0$, the imaginary parts $\omega'_{s}, \omega_{s}$ are cohomologous in the elliptic de Rham cohomology, i.e. there exists a smooth family of primitives $\alpha_{s}\in\Omega^{1}(\log |D|)$ such that 
\begin{equation}
\omega'_{s} - \omega_{s} = d\alpha_{s}
\end{equation}
and $\alpha_{0}=0$.  We then interpolate between $\omega_{s}$ and $\omega'_{s}$, defining
\begin{equation}
\omega_{s,t} = t\omega'_{s} + (1-t)\omega_{s},
\end{equation}
which is an elliptic symplectic form for all $t\in[0,1]$ if $s$ is sufficiently small (since $\omega'_{0}=\omega_{0}$). 
Now let $X_{s} = \omega^{-1}_{s,t}\alpha_{s}$ be a family of vector fields, each time-dependent, and let $\Phi_{s,t}$ be the exact time-$t$ flow generated by $X_{s}$.  Applying this flow at time $1$ to $\sigma'_{s}$, 
we obtain a new deformation $\sigma''_{s}$ of $\sigma_{s}$, defined by 
\begin{equation}
\sigma''_{s} = (\varphi_{s,1}^{*})^{-1}(\sigma'_{s} + B'_{s,1}),\qquad 
B'_{s,t} = \int_{0}^{t}  i_{X_{s}}(\varphi_{s,u}^{*}H) du.
\end{equation}
Then $\sigma''_{s} - \sigma_{s}$ is an exact log form, but with zero imaginary part since $\varphi_{s,1}^{*}\omega_{s} = \omega'_{s}$.\\

\noindent{\bf Step 2.} We now produce an exact flow taking $\sigma''_{s}$ to $\sigma_{s}$.  First interpolate between the two families:
\begin{equation}
\sigma_{s,t} = t\sigma''_{s} + (1-t)\sigma_{s}. 
\end{equation}
By Step 1, the time derivative $\del_{t}\sigma_{s,t} = \sigma''_{s}-\sigma_{s}$ is exact as a log form, so we have $\del_{t}\sigma_{s,t} = d\zeta_{s}$, for $\zeta_{s}\in\Omega^{1}(\log D)$ with $\zeta_{0}=0$ and such that $d\zeta_{s}$ is a smooth real form.  If $\nu_{s} = \Im(\zeta_{s})$ then $\nu_{s}$ is a closed elliptic 1-form. 
To trivialize this interpolating family using an exact flow, we need a family of sections $e_{s} = Y_{s} +\xi_{s}$ of $\TT M$ with the property that 
\begin{equation}\label{infinitesimalflow}
\del_{t}\sigma_{s,t} = -d (({\psi^{s,t}_{*}Y_{s}})\lrcorner\, \sigma_{s,t} -  \xi_{s}),
\end{equation}
where $\psi_{s,t}$ is the time-$t$ flow of $Y_{s}$.  We may solve this as follows: let $Y_{s}=-\omega_{s}^{-1}(\nu_{s})$, a family of time-independent Poisson vector fields associated to the closed elliptic form $\nu_{s}$; then $\psi^{s,t}_{*}Y_{s} = Y_{s}$ and, crucially, $\zeta_{s} + i_{Y_{s}}\sigma_{s,t}$ has zero imaginary part.  By Proposition~\ref{isomzerores}, this implies that it is a family of smooth real closed 1-forms
\begin{equation}
\xi_{s} = \zeta_{s} + i_{Y_{s}}\sigma_{s,t},
\end{equation}
solving~\eqref{infinitesimalflow} and so providing the required exact flow identifying $\sigma''_{s}$ with $\sigma_{s}$, completing the proof of injectivity.

Finally, surjectivity of the period map follows from the fact that for any path $\gamma:[0,1]\to H^{2}(M\backslash D,\CC)$ with $\gamma(0)=0$, we may lift this to a smooth family of cocycles $\wt\gamma:[0,1]\to \Omega^{2}(\log D)$ with $\wt\gamma(0)=0$. Then, since the nondegeneracy condition is open, $\sigma_{s} = \sigma + \wt\gamma(s)$ is, for sufficiently small $s$, a deformation of $\sigma$ whose period realizes the given path germ. 
\end{proof}

\begin{remark}
When the flux $H$ vanishes, the complex log symplectic form $\sigma$ itself defines a class in $H^{2}(\log D)$, and so the image of the period map may be taken to be naturally centered at $[\sigma]$. 
\end{remark}

In summary, we have shown that the local deformation problem for a stable generalized complex structure $\JJ$ on the fixed manifold with 3-form $(M,H)$ is equivalent to the local deformation problem for the complex log symplectic structure $(D,\sigma)$ determined by $\JJ$, and that this, in turn, is equivalent to the local deformation problem for $\sigma$, keeping $D$ fixed.   Finally, this last deformation problem is governed by a period map to a neighbourhood of zero in $H^{2}(M\backslash D,\CC)$.  We therefore obtain the following explicit description, in the stable case, of the Kuranishi family of local deformations of generalized complex structures described in~\cite{MR2811595}:

\begin{corollary}\label{independentgoto}
The Kuranishi moduli space of deformations of the stable generalized complex structure $\JJ$ is unobstructed and is identified by the period map with an open set surrounding the origin in $H^{2}(M\backslash D,\CC)$, where $D$ is the anticanonical divisor of $\JJ$.
\end{corollary}

\begin{example}
Consider the case of Example~\ref{holpoisson} in which $\pi$ is a generic holomorphic Poisson structure on the complex projective plane, vanishing on a smooth elliptic curve $D\subset \CC P^{2}$.  Then $H^{2}(M\backslash D)$ is 2-dimensional, implying that all germs of deformations deformations of the generalized complex structure defined by $\pi$ are, up to equivalence, obtained by deforming the holomorphic section $\pi$. 
\end{example}

\begin{example}\label{hopf}
The generalized K\"ahler structure on the Hopf surface 
described in~\cite[Example 1.21]{MR3232003} involves a pair $\JJ_{-},\JJ_{+}$ of stable generalized complex structures on $M = S^{3}\times S^{1}$, each integrable with respect to a cohomologically nontrivial 3-form $H$, and each of which has anticanonical divisor $D_{\pm}$ given by a single $T^{2}$ fiber of the Hopf projection $S^{3}\times S^{1}\to S^{2}$.  Therefore, the period map for $\JJ_{\pm}$ takes values in $H^{2}(M\backslash D_{\pm},\CC)$, which is 1-dimensional since $M\backslash D_{\pm}$ is homotopic to $T^{2}$.
\end{example}

\subsection{The period map for variable flux}\label{pernofixH}

A similar strategy to that used in the previous section may be used to describe the local moduli space of simultaneous deformations of the stable generalized complex structure $\JJ$ and the closed 3-form $H$ for which $\JJ$ is integrable.  

\begin{definition}
A deformation of the pair $(\JJ,H)$ is defined to be a smooth family $(\JJ_{s},H_{s})$, $s\in [0,1]$, of generalized complex structures where each $\JJ_{s}$ is integrable with respect to $H_{s}$ and such that $(\JJ_{0},H_{0})=(\JJ,H)$. Two such deformations $(\JJ_{s},H_{s})$ and $(\JJ'_{s},H'_{s})$ of the same pair are said to be equivalent when there is a family of vector fields $X_{s}$ and real 2-forms $b_{s}$, each allowed to be time-dependent, such that the time-1 flow $\Phi_{1}(X_{s},b_{s})$ takes $\JJ_{s}$ to $\JJ'_{s}$. 
\end{definition}

Unlike the equivalence relation in the previous section, where only exact flows were used, here we use any path of equivalences of generalized complex structures, without regard to the effect on the 3-form.  As a result, Theorem~\ref{correspellj} implies that deformations of pairs $(\JJ,H)$ are, up to equivalence, in bijection with equivalence classes of deformations of co-oriented elliptic symplectic structures $(Q,\mathfrak{o})$ with vanishing elliptic residue, where a deformation is defined as usual and two deformations are equivalent if there is a family of diffeomorphisms taking one to the other. 

Any family $Q_{s}$ of elliptic symplectic structures may be viewed as a family of elliptic symplectic forms $\omega_{s} = Q_{s}^{-1}$ for the family of elliptic divisors defined by $(\wedge^{2n}TM, \wedge^{n}Q_{s})$, and just as in the previous section, we may always rectify the family of divisors by a path of diffeomorphisms. As a result, we may pass directly to deformations of elliptic symplectic forms with fixed underlying elliptic divisor.
\begin{definition}
Fix the manifold $M$, the elliptic divisor $|D|$, and the elliptic symplectic form $\omega$.  A (zero-residue) deformation of $\omega$ is a smoothly varying family $\omega_{s}\in \Omega^{2}(\log |D|)$, $s\in[0,1]$, of elliptic symplectic forms such that $\omega_{0}=\omega$  (such that each $\omega_{s}$ has zero elliptic residue).  Two such deformations $\omega_{s},\omega'_{s}$ are equivalent when there is a family of time-dependent sections $X_{s}$ of the elliptic tangent bundle $T(-\log |D|)$ whose associated time-1 flow takes $\omega_{s}$ to $\omega'_{s}$.  
\end{definition}
We then have the analog of Lemma~\ref{rectlemma1}, allowing us to pass between deformations of pairs $(\JJ,H)$ and of elliptic symplectic forms.
\begin{lemma}
Let $\JJ$ be a stable generalized complex structure on $(M,H)$ with $M$ compact. Rectification of divisors defines a canonical bijection between equivalence classes of deformations of $(\JJ, H)$ and equivalence classes of zero-residue deformations of the elliptic symplectic structure $\omega = Q^{-1}$ with fixed divisor $|D|=(\wedge^{2n} TM, \wedge^{n} Q)$, where $Q$ is the real Poisson structure underlying $\JJ$. 
\end{lemma}

We now define a period map for deformations of pairs $(\JJ,H)$: since, after rectification, these give a family of elliptic symplectic forms with zero elliptic residue, our period map must take values in $H^{2}_{0}(\log |D|)$.

\begin{definition}
Let $\omega_{s}$, $s\in[0,1]$, be a zero-residue deformation of the elliptic symplectic structure $\omega$, with fixed elliptic divisor $|D|$. Its period is defined to be the path given by
\begin{equation}
s\mapsto \sP(\omega_{s}) = [\omega_{s}] \in H^{2}(M\backslash D,\RR)\oplus  H^{1}(D,\RR),
\end{equation}
where we use the identification of $H^{2}_{0}(M,\log |D|)$ with the above group from Theorem~\ref{ellzerores}.
\end{definition} 

\begin{theorem}\label{theo:stable<=>elliptic symplectic}
Let $\omega$ be an elliptic symplectic form with zero elliptic residue on the compact manifold with elliptic divisor $(M,|D|)$.  The period map defines a canonical bijection between germs of deformations of $\omega$ up to equivalence and germs of smooth paths beginning at $[\omega]\in H^{2}(M\backslash D, \RR)\oplus H^{1}(D,\RR)$.
\end{theorem} 
\begin{proof}
The proof is a simpler version of that for Theorem~\ref{periodmapcx}.  The period map descends to equivalence classes because equivalence uses flows of elliptic vector fields, which act trivially on the elliptic de Rham cohomology groups.  

To show injectivity of the period map, suppose $\omega_{s},\omega'_{s}$ are two deformations of $\omega$ with the same period, so that $[\omega_{s}] = [\omega'_{s}]$ in $H^{2}_{0}(\log |D|)$. Then we can apply a simple Moser argument to identify $\omega_{s}$ with $\omega'_{s}$ in a sufficiently small neighbourhood of $s=0$.  

As in the previous case, surjectivity holds for germs by the openness of the nondegeneracy condition for elliptic symplectic forms.
\end{proof}

In summary, we have shown that the local deformation problem for the pair $(\JJ, H)$ of a stable generalized complex structure on the fixed compact manifold $M$ is equivalent to the local deformation problem for the zero-residue elliptic symplectic form $\omega=Q^{-1}$ determined by $\JJ$, and that this is governed by a period map to a neighbourhood of $[\omega]\in H^{2}_{0}(\log |D|)$. 

\begin{corollary}
The Kuranishi moduli space of simultaneous deformations of the pair $(\JJ, H)$ of a stable generalized complex structure integrable with respect to $H$ is unobstructed and is identified by the period map with an open set surrounding the class determined by the underlying real Poisson structure $Q$ in $H^{2}(M\backslash D,\RR)\oplus H^{1}(D,\RR)$, where $D$ is the anticanonical divisor of $\JJ$.
\end{corollary}

\begin{remark}
Let $(\JJ, M, H)$ be a compact stable generalized complex manifold.  The forgetful map from deformations of $(\JJ,H)$ 
to deformations of $H$ induces a map on cohomology groups 
$H^{2}_{0}(\log |D|)\to H^{3}(M,\RR)$ which by Theorem~\ref{sequenceselllog} is the projection from $H^{2}_{0}(\log |D|)$ to $H^{1}(D,\RR)$ followed by the Thom-Gysin map to $H^{3}(M,\RR)$.  
Consequently, not all directions from $[H]$ in $H^{3}(M,\RR)$ may be obtained by deforming the pair $(\JJ,H)$: only those in the kernel of the pullback $i^{*}:H^{3}(M,\RR)\to H^{3}(M\backslash D, \RR)$ to the anticanonical complement.   
\end{remark}

\begin{example}
Revisiting Example~\ref{hopf}, we see that deformations of the pair $(\JJ_{\pm},H)$ are controlled by a period map to $H^{2}(M\backslash D_{\pm})\oplus H^{1}(D_{\pm})$, a real vector space of dimension $3$, in contrast to the period map for $\JJ_{\pm}$, which maps to a complex line.
\end{example}

\subsection{Topological constraints for stable  structures}\label{topconst}

The equivalence between stable generalized complex structures and logarithmic symplectic forms developed in this section, together with the computations of logarithmic and elliptic de Rham cohomology in Section~\ref{sectcxdiv}, lead to topological consequences for the anticanonical divisor $D$, which must admit generalized Calabi-Yau structure, as well as for $M\backslash D$.    
%
%

\begin{theorem}\label{b1big}
Let  $(M,H,\J)$ be a stable \gcm\ of dimension $2n$ whose anticanonical divisor $D$, has a compact component, $D'$. Then the following hold:
\begin{enumerate}
\item $D'$ fibers over $T^2$, $b_1(D') \geq 2$ and $\chi(D') =0$;
\item If $M$ is compact, the pair $(\J,H)$ can be deformed so that the \gcys\ on each component of  $D$ is proper.
\item If $c_1(K)|_{D'} =0$, then the twisting class of $D'$ vanishes, $D'$ admits a symplectic form and there are classes $a,b\in H^1(D')$ and $c\in H^2(D')$ such that $abc^{n-2} \neq 0$. 
\end{enumerate}
\end{theorem}
\begin{proof} The first and last  claims follow directly from  Theorems \ref{gholstrnd},  \ref{theo:twisting class and chern class} and \ref{prop:topology1}. As for the second claim, if $\omega$ is the elliptic symplectic form associated to $\J$, then $\omega$ induces two classes in $H^1(D,\R)$. The first, $\Res_{r}[\omega]$, comes from the second component of the isomorphism from Theorem \ref{ellzerores}:
\begin{equation}
\xymatrix{
H^2_0(\log |D|)\ar[rr]^-{\cong}_-{(i^{*},\Res_{r})}& & H^2(M\backslash D,\R) \oplus H^1(D,\R)
},
\end{equation}
and the second is obtained by applying the topological residue map $R$ (c.f. Proposition~\ref{poincaregysin}) to the first component $i^{*}[\omega]$.  If $\sigma$ is the complex log symplectic form representing $\JJ$, then the complex 1-form $\Omega$ making up the generalized Calabi-Yau structure of $D$ satisfies
\begin{equation}
[\Omega] = [\Res(\sigma)] = \tfrac{1}{2\pi}R i^{*}[\omega] + i\Res_{r}[\omega].
\end{equation}
Since we have the period map for elliptic symplectic forms, we may perturb $\omega$, keeping it symplectic but also ensuring that its classes in $H^{1}(D,\RR)$ and $H^2(M\backslash D,\R)$ have periods which are rational and rational multiples of $2\pi$, respectively, therefore making the periods of $[\Omega]$ rational. Then the argument from Theorem \ref{prop:topology1} shows that the Calabi--Yau structure induced on $D$ is proper. By Theorem \ref{correspellj}, this deformation of $\omega$ corresponds to a deformation of the pair $(\J,H)$, as needed.
\end{proof}

Next, we state the topological constraints placed on $M$ by the stable \gcs, in analogy with those obtained in~\cite{Cavalcanti13,Marcut-Osorno} for real log symplectic structures.

\begin{theorem}\label{theo:topology2}
Let  $(M,H,\J)$ be a stable \gcm\ of dimension $2n$ whose anticanonical divisor $D$ has at least one compact component, $D'$. Then the following hold:
\begin{enumerate}
\item There is a class $\alpha \in H^2(M \backslash D)$ such that $\alpha^{n-1}\neq 0$;
\item There is a class $\beta \in H^2_{c}(M\backslash D)$ such that $\beta^2 =0$ and  $\alpha^{n-1} \beta \neq 0$, both as classes with compact support.
\end{enumerate}
\end{theorem}
\begin{proof} {\it 1.} 
Let $\omega$ be the elliptic symplectic structure corresponding to $\JJ$, and let $i:M\backslash D\to M$ be the inclusion of the divisor complement.  Then by Theorem~\ref{thm:lie algebroid cohomology}, the class $[\omega]\in H^{2}(\log |D|)$ decomposes as $[\omega] = \alpha + \gamma$,  with $\alpha = i^{*}[\omega]$ a class in $H^{2}(M\setminus D,\RR)$ and $\gamma = \Res_{r}[\omega]$ a class in $H^{1}(\SS{N},\RR)$. 

The key geometric observation is that since $\omega^{n}$ is nowhere vanishing, its radial residue defines a volume form on $\SS{N}$ and so $\Res_{r}([\omega]^{n})\neq 0$ in $H^{2n-1}(\SS{N}D',\RR)$, where $\SS{N}D'$ denotes the circle bundle of the normal bundle of $D'$.
We then use the description of cup product on elliptic de Rham cohomology from Theorem~\ref{cupprod} to compute 
\begin{equation}
\Res_{r}([\omega^{n}]) = \Res_{r}([\omega]^{n}) = n r(i^{*}[\omega]^{n-1})\cup \Res_{r}[\omega] = n r(\alpha^{n-1})\cup \gamma.
\end{equation}
where $r:H^{2}(M\backslash D,\RR)\to H^{1}(\SS{N},\RR)$ is the canonical restriction map~\eqref{resttos1}. Since the above class is nonzero on $D'$, we obtain that $\alpha^{n-1}\neq 0$, as required.

%
%
%

{\it 2.} Identify $\SS{N}D'$ with the boundary of a small tubular neighbourhood of $D'$, including via $k:\SS{N}D'\to M\backslash D$ into the divisor complement.  By the proof of the first part, we see that $k^{*}\alpha^{n-1}\cup \gamma\neq 0$.  Pushing forward along $k$, we get a class of compact support $k_{*}\gamma$ in $H^{2}_{c}(M\backslash D)$ with the property $\alpha^{n-1}\cup k_{*}\gamma\neq 0$ in cohomology with compact support.  
At the cochain level, $k_{*}\gamma$ is obtained by wedging with a Thom form for the normal bundle of $S^1ND'$, so it squares to zero and hence the same holds in cohomology.
\end{proof}

\begin{corollary}\label{cor:simple corollary}
Let $M^{2n}$ be a compact manifold and $i:D^{2n-2}\into M^{2n}$  a compact submanifold. If $i^*:H^1(M)\into H^1(D)$ is surjective and $i^*:H^2(M)\into H^2(D)$ is injective, then $M$ does not have a stable \gcs\ whose type change is $D$.   
\end{corollary}
\begin{proof}
Let $U$ be a tubular \nhood\ of $D$. The sequence for the pair $(M,U)$ gives the vanishing of the relative cohomology group $H^2(M,U)$. Since $U$ can be made arbitrarily small, this implies the vanishing of the compact support cohomology group $H^2_c(M\backslash D)$, precluding the existence of the nonzero class $\beta$.
\end{proof}

\begin{example}[Lefschetz hyperplane theorem] 
We saw in Example~\ref{holpoisson} that a holomorphic Poisson structure $\pi$ defines a stable generalized complex structure when its Pfaffian $\pi^{n}$ vanishes transversely.  We now observe that although many interesting holomorphic Poisson structures exist on Fano manifolds, that is, compact complex manifolds for which the anticanonical bundle is positive, these cannot define stable generalized complex structures because the Lefschetz hyperplane theorem implies
that the conditions of Corollary \ref{cor:simple corollary} are met, and so Fano varieties can not be deformed to stable \gcms\ by means of holomorphic Poisson bivectors. This was known before and is a consequence a simpler argument, as pointed out in~\cite{MR3100779}: Fanos are simply connected, and so any smooth anticanonical divisor must also be, contradicting the constraint $b_{1}(D)>2$ from Theorem~\ref{b1big}.

Notice, however that the Lefschetz hyperplane theorem implies that conditions of Corollary \ref{cor:simple corollary} are met  whenever $D$ is a positive divisor in a complex manifold $M$,  hence no such pair $(M,D)$ may admit a stable \gcs, a fact which does not follow simply from Theorem \ref{b1big}.

The same argument carries through for symplectic manifolds: given a symplectic manifold $(M^{2n},\omega)$ such that $[\omega]$ is a rational class, for $k$ large enough $k[\omega]$ determines a line bundle and by work of Donaldson \cite{MR1438190}, this line bundle has a section for which the Lefschetz hyperplane theorem holds. Again, Corollary~\ref{cor:simple corollary} implies that the pair $(M,D)$ cannot admit a stable \gcs. 
\end{example}

%

\subsection{Darboux coordinates about a degenerate point}\label{darbo}

Start with complex coordinates $(w,z)$ and extend by real coordinates $x_{i}, p_{i}, i=1,\dots, m-2$. Suppose that $w$ is the local defining function for the complex divisor $D$.  Consider the closed logarithmic 2-form 
\begin{equation}\label{stdformimaglog}
\sigma_{0} = d\log w \wedge dz + i \omega,
\end{equation}
 where $\omega=\sum_{i} dx_{i}\wedge dp_{i}$ is the standard symplectic form.  Then this is a complex log symplectic structure since its imaginary part is nondegenerate in the elliptic sense:
 \begin{equation}\label{ellvol}
(\Im^{*}\sigma_{0})^{m} = d\log w\wedge d\log \ol w \wedge dz\wedge d\ol z \wedge \omega^{m-1} 
\end{equation}
is a nowhere vanishing elliptic form of top degree. 
Our aim is to prove that locally all complex log symplectic forms are equivalent to the one above. 

More precisely, we show that if $(\sigma,H)$ is any complex log symplectic form, integrable with respect to the 3-form $H$, then we can find a smooth real 2-form $b$ and a diffeomorphism $\varphi$ such that 
\begin{equation}
\begin{aligned}
\varphi^{*}H +db &= 0\\
\varphi^{*}\sigma + b &= \sigma_{0}.
\end{aligned}
\end{equation}

\begin{theorem}\label{darbouximaglog}
Any complex log symplectic form is equivalent, near a point on its degeneracy divisor, to the normal form~\eqref{stdformimaglog}.
\end{theorem}
\begin{proof}
We are only concerned about the local structure near a point $p$ on the divisor, so we may identify the divisors and assume that the algebroid $T(-\log D)$ is fixed for the remainder of the argument, and that both $\sigma$ and $\sigma_{0}$ are elements of $\Omega^{2}(\log D)$.  We also assume that $\sigma_{0},\sigma$ induce the same orientation on the neighbourhood, in the sense that the 
elliptic volume forms~\eqref{ellvol} of $\sigma_{0},\sigma$ have positive ratio.

Let $\omega_{0},\omega$ be the imaginary parts of $\sigma_{0},\sigma$, which are elliptic forms with vanishing elliptic residue, by Proposition~\ref{isomzerores}. By Theorem~\ref{ellzerores}, in a small ball $W$ surrounding a point on the divisor, we have $H^{2}_{0}(W,\log|D|)=H^{2}(W\backslash D,\RR) \oplus H^{1}(W\cap D)$, both of whose summands vanish since $W\backslash D$  is homotopic to the circle and $W\cap D$ is contractible.  Hence $\omega,\omega_{0}$ must be cohomologous:
\begin{equation}
\omega = \omega_{0} + d\alpha,\qquad \alpha\in \Omega^{1}(W,\log |D|),\qquad \alpha(p) = 0.
\end{equation}
Furthermore, we may choose $\alpha$ such that it vanishes at $p$ as a section of $T^{*}(\log |D|)$, which is possible since locally we have a basis of closed sections. 
We then employ the Moser argument: the interpolating family $\omega_{t} = t\omega + (1-t)\omega_{0}$ is nondegenerate for all $t\in[0,1]$ and has derivative $\tfrac{d}{dt}\omega_{t} = d\alpha$.  The vector field $X_{t}=\omega_{t}^{-1}\alpha$ is then a section of $T(-\log |D|)$ which vanishes at $p$, so that we may integrate it (in a possibly smaller neighbourhood of $p$) to a flow $\varphi_{t}$ such that $\varphi_{1}^{*}\omega = \omega_{0}$.

Having found $\varphi$ such that $\varphi^{*}\sigma, \sigma_{0}$ share an imaginary part, we may appeal to Proposition~\ref{isomzerores} to conclude that $\varphi^{*}\sigma - \sigma_{0} = -b$ for a smooth real 2-form $b$, which by the integrability condition satisfies $\varphi^{*}H+db =0$, as required.
\end{proof}

In view of the equivalence between stable generalized complex structures and complex log symplectic forms, the existence of the normal form~\eqref{stdformimaglog} means that we have a normal form for stable generalized complex structures. Indeed, a generator for the canonical line bundle may be written as 
\begin{equation}\label{localformgc}
w e^{\sigma_{0}} = (w + dw\wedge dz)\wedge e^{i\omega},
\end{equation}
which may also be viewed as a deformation of the type 2 structure 
\begin{equation}
dw\wedge dz\wedge e^{i\omega}
\end{equation}
by the holomorphic Poisson structure $w\del_{w}\wedge \del_{z}$.  In fact, one can alternatively deduce the local normal form from Bailey's theorem~\cite{bailey-2012}, which states that near a point of type $k$, a generalized complex structure is equivalent to the product of a symplectic structure with a deformation of the complex structure on a neighbourhood of the origin in $\CC^{k}$ by a holomorphic Poisson structure vanishing at the origin; this, together with the nondegeneracy assumption, determines the form~\eqref{localformgc} uniquely.

\subsection{Linearization about the degeneracy locus}\label{lindeglocsec}

In Section~\ref{stablestructures}, we introduced the linearization of a stable generalized complex $\JJ$ structure along its anticanonical divisor $D$. This is a $\CC^{*}$-invariant stable structure on the total space of the normal bundle $N$ of $D$, and was defined purely in terms of the induced geometric structure on $D$, namely, the generalized Calabi-Yau structure and the generalized holomorphic structure on $N$. 
We now show that $\JJ$ is equivalent to its linearization $\JJ'$ in a sufficiently small neighbourhood of $D$, using a Moser argument applied to the complex log symplectic structures associated to $\JJ,\JJ'$.

\begin{theorem}\label{linaboutd}
Let $\JJ,\JJ'$ be stable generalized complex structures on $(M,H)$ with the same anticanonical divisor $D$.  If $\JJ,\JJ'$ induce the same generalized Calabi-Yau structure on $D$, as well as the same generalized  holomorphic structure on the normal bundle to $D$, then $\JJ$ and $\JJ'$ are equivalent on a neighbourhood of $D$.  In particular, any stable generalized complex structure is equivalent to its linearization in a sufficiently small tubular neighbourhood of $D$.
\end{theorem}
\begin{proof}
Let $\sigma,\sigma'$ be the complex log symplectic forms corresponding to $\JJ,\JJ'$, let $w$ be a local defining function for the anticanonical divisor, and let $\iota:D\to M$ be the inclusion.  That is, $\JJ,\JJ'$ have local trivializations $\rho,\rho'$ for their canonical bundles given by 
\begin{equation}
\rho = we^{\sigma},\qquad \rho' = we^{\sigma'}.
\end{equation}
The assumption that $\JJ,\JJ'$ induce the same Calabi-Yau structure on $D$ is the condition
\begin{equation}\label{equalres}
\Res(e^{\sigma}) = \Res(e^{\sigma'}).
\end{equation}

Now write $\sigma = d\log w \wedge \Omega + \beta$ and $\sigma' = d\log w \wedge \Omega' + \beta'$, for $\Omega,\Omega'$ and $\beta,\beta'$ smooth complex forms, so that $\Res(\sigma) = \iota^{*}\Omega$ and similarly for $\sigma'$.  Then condition~\eqref{equalres} is equivalent to the condition that 
\begin{equation}
\iota^{*}(\Omega e^{\beta}) = \iota^{*}(\Omega' e^{\beta'}).
\end{equation}
This implies that $\iota^{*}(\Omega'-\Omega) = 0$ and also that $\iota^{*}(\beta'-\beta) \wedge \iota^{*}\Omega = 0$. 

The condition that $\JJ,\JJ'$ induce the same generalized holomorphic structure on the normal bundle $N$ to D is equivalent to the condition that the modular vector fields $X+\xi$, $X'+\xi'$ associated to $\rho,\rho'$, determined uniquely by the conditions
\begin{equation}\label{fielddeterm}
\begin{aligned}
i_{X}\ol\sigma + \xi &=0\\
i_{X}\sigma + \xi &=d\log w 
\end{aligned}
\qquad
\begin{aligned}
i_{X'}\ol\sigma' + \xi' &=0\\
i_{X'}\sigma' + \xi' &=d\log w,
\end{aligned}
\end{equation}
must induce the same connection forms for the normal bundle, namely
\begin{equation}\label{equalconnform}
X + \iota^{*} \xi = X' + \iota^{*}\xi',
\end{equation}
as sections of $\TT_{\CC}D$.
From~\eqref{fielddeterm}, we have that $i_{X'}\sigma + \xi' = i_{X}\sigma' + \xi$, and using~\eqref{equalconnform} this implies that $\iota^{*}(i_{X}(\sigma'-\sigma)) = 0$, and therefore $\iota^{*}(i_{X}(\beta'-\beta))=0$. But note that~\eqref{fielddeterm} implies that $\iota^{*}(i_{X}\Omega )= -1$, 
and applying $X$ to the equation $\iota^{*}(\beta'-\beta) \wedge \iota^{*}\Omega = 0$, we obtain that $\iota^{*}(\beta'-\beta) = 0$.  So, we have that 
\begin{equation}
\sigma'-\sigma = d\log w \wedge (\Omega' -\Omega)  + (\beta'-\beta),
\end{equation}
and each of the components $\Omega'-\Omega$ and $\beta'-\beta$ are smooth forms which vanish upon pullback to $D$. We apply Lemma~\ref{pullthenzero} to each of these components to conclude that $\sigma'-\sigma$, as a complexified elliptic form, vanishes along $D$; in particular its imaginary part is an elliptic form which vanishes along $D$.  We now apply Lemma~\ref{semilocalform} to conclude that $\sigma',\sigma$ are equivalent in a tubular neighbourhood of $D$, as required. 

For the final statement, let $\JJ'$ be the linearization of $\JJ$   along $D$, as defined in Definition~\ref{linearizedef}. By the construction of the linearization, we may identify a tubular neighbourhood of $D$ with a neighbourhood in its normal bundle in such a way that the anticanonical divisors of $\JJ$ and $\JJ'$ are identified.  Then $\JJ,\JJ'$ are integrable with respect to three-forms $H,H'$ respectively, which agree on $D$, i.e., $\iota^{*}H = \iota^{*}H'$.  Therefore $H,H'$ are cohomologous in a tubular neighbourhood of $D$, and we may choose $B$ such that $dB=H'-H$ and with the additional property $\iota^{*}B = 0$.  We may then gauge transform $\JJ$ by $B$ so that $\JJ,\JJ'$ share the same 3-form $H$, without changing the fact that, by construction, $\JJ,\JJ'$ induce the same generalized Calabi-Yau structure on $D$ and the same holomorphic structure on its normal bundle.   We then proceed as before.
\end{proof}

\begin{lemma}\label{pullthenzero}
Let $D$ be a complex divisor, $\iota:D\to M$ the inclusion map, and $a:T(-\log |D|)\to TM$ the anchor map for the elliptic tangent bundle.  If $\varpi$ is a smooth differential form such that $\iota^{*}\varpi = 0$, then $a^{*}\varpi$ is an elliptic logarithmic form which vanishes along $D$.
\end{lemma}
\begin{proof}
Let $(r,\theta,x_{3},\ldots, x_{n})$ be local coordinates near $D$ as in Section~\ref{elltn}, and let $u = r\cos\theta$, $v=r\sin\theta$, so that $D$ is the common zero set of $u,v$ and we may write 
\begin{equation}\label{smothform}
\varpi = du \wedge dv \wedge \varpi_0 + du \wedge \varpi_1 + dv \wedge \varpi_2 + \varpi_3,
\end{equation}
where $\varpi_{i}$ are smooth forms lying in the subalgebra generated by $dx_{3},\ldots dx_{n}$.  Then since $du = ud\log r - v d\theta$ and $dv=vd\log r + ud\theta$, we see that the first three summands in~\eqref{smothform} vanish along $D$ as elliptic logarithmic forms.  Finally, $\iota^{*}\varpi_{3}=0$ if and only if $\varpi_{3}$ vanishes along $D$, implying it vanishes as an elliptic form as well.
\end{proof}

\begin{lemma}\label{semilocalform}
Complex log symplectic forms whose imaginary parts coincide along $D$ must be equivalent in a tubular neighbourhood of $D$. 
\end{lemma}
\begin{proof}

Let $\sigma_{0},\sigma_{1}$ be the given forms, and let their respective imaginary parts be $\omega_{0},\omega_{1}$, elliptic forms which coincide along $D$. The linear interpolation $\omega_{t} = t\omega_{1} + (1-t)\omega_{0}$ is nondegenerate for all $t\in[0,1]$, and has derivative $\dot\omega_{t}=\omega_{1}-\omega_{0}$, a closed elliptic form vanishing along $D$. By Lemma~\ref{retractlog}, it is exact, so that $\omega_{1} = \omega_{0} + d\alpha$ for $\alpha\in\Omega^{1}(W,\log |D|)$.  Then the vector field $X_{t} = \omega_{t}^{-1}(\alpha)$ is tangent to the compact submanifold $D$, meaning that the time-1 flow $\varphi$ exists in a sufficiently small neighbourhood of $D$ and satisfies $\varphi^{*}\omega_{1} = \omega_{0}$.  We complete the proof as we did for Theorem~\ref{darbouximaglog}, using Proposition~\ref{isomzerores} to argue that $\varphi^{*}\sigma_{1}-\sigma_{0}=-b$ for a smooth real 2-form $b$ satisfying $\varphi^{*}H+db=0$.
\end{proof}

\subsection{Symplectic filling in dimension four}


The idea that one can perform surgeries on a symplectic manifold to produce nontrivial examples of stable structures has been used to produce large families of examples in four dimensions \cite{MR2312048,MR2574746,Goto:2013vn,MR2958956,MR3177992}. As an application of the normal form about the anticanonical divisor~\ref{linaboutd}, we show that a converse to this also holds, in the sense that any stable generalized complex 4-manifold may be modified in a neighbourhood of its anticanonical divisor to produce a compact symplectic manifold.

\begin{theorem}\label{theo:symplectization}
Let $M^4$ be a compact stable \gcm\ whose anticanonical divisor has connected components  $D_{1},\ldots, D_{n}$.  Then for each $i$ there is a tubular neighbourhood $U_{i}$ of $D_{i}$, a symplectic manifold with boundary $(X_i,\omega_i)$ and an orientation reversing diffeomorphism of coisotropic submanifolds $\gf_i:\del X_i \stackrel{\cong}{\longrightarrow} \del {\ol U_i}$, so that 
\begin{equation}
\widetilde{M}= M\backslash \ol{U}\cup_{\gf} X
\end{equation}
is a symplectic manifold, where $U,\gf, X$ denote the unions of $U_{i},\gf_{i},X_{i}$ for all $i$.

Further, $X_i$ can be chosen so that $b^+(X_i) > 0$ and the restriction map $H^2(X_i)\into H^2(\del X_i)$ is surjective.
\end{theorem}
\begin{proof}
Let $\omega$ be the elliptic log symplectic form associated to the stable \gcs\ and $U_i$ be a tubular \nhood\ of a component 
$D_{i}$ of the anticanonical divisor in which $\omega$ is equivalent to its linearization on $N$,  the normal bundle of $D$. After a choice of Hermitian structure and connection on $N$ and identification of $N$ with a tubular \nhood\ we have that on an $\e$-disc bundle, $D_\e N$, of $N$
$$\omega = d \log r \wedge \omega_1 + \theta\wedge \omega_2 + \omega_3,$$
where $\omega_j$ are forms pulled back from $D_i$. If $E= r\del_{r}$, then we have that that $\mc{L}_E \omega =0$ and $E$ is transverse to $\SS{N}$, hence $\omega$ induces a co-symplectic structure on $\SS{N}$, i.e. a pair of closed forms $\sigma \in \Omega^2(\SS{N})$ and $\tau  \in \Omega^1(\SS{N})$ such that  $\sigma^{n-1}\wedge \tau \neq 0$, namely,  $\sigma =\omega|_{\SS{N}}$ and $\tau = i_X\omega|_{\SS{N}}$.

To prove the result we consider $M\backslash D_\e N$ and we need to prove that $\del(M\backslash D_\e N) = S^1N$ can be filled symplectically. By the coisotropic neighbourhood theorem, to achieve this it suffices to find a copy of each component $Y$ of $\SS{N}$ as a separating submanifold in a compact symplectic four manifold $(X,\omega_X)$ such that $\omega_X$ restricted to $Y$ is $\sigma$. This is achieved by observing that co-symplectic 3-manifolds are particular examples of taut foliations: indeed, $\tau$ determines an integrable distribution and $\sigma$ is a positive form on the leaves. Therefore our result follows from the following general result of Kronheimer and Mrowka on 3-manifolds with taut foliations:

\begin{theorem}[{\cite[Theorem 41.3.1]{MR2388043}}]
Let $Y$ be a closed oriented 3-manifold. Suppose that $Y$ has a  taut foliation, $\mc{F}$, and let $\sigma$ be a closed 2-form on $Y$ which is positive on the leaves of $\mc{F}$. Then there is a closed symplectic manifold $(X,\omega_X)$ containing $Y$ as a separating submanifold and such that the restriction of $\omega_X$ to $Y$ is $\sigma$.

Furthermore, if $Y$ is not $S^1 \times S^2$, then we can arrange that the map $H^2(X)\into H^2(Y)$ is surjective and that the two components $X_1$ and $X_2$ into which $Y$ divides $X$  both have $b^+>0$. 
\end{theorem}
  
\end{proof}

\subsection{Neighbourhood theorem for Lagrangian branes}

In this section we introduce the elliptic analog of the cotangent bundle construction in symplectic geometry, and in this way produce a large family of examples of stable generalized complex manifolds.  We also prove a generalization of Weinstein's Lagrangian neighbourhood theorem, resulting in a normal form result for neighbourhoods of Lagrangian branes in stable generalized complex manifolds.

Let $D=(R,q)$ be an elliptic divisor on the $n$-manifold $L$, and $TL(-\log |D|)$ the associated elliptic tangent bundle. Then let $M = \tot(T^{*}L(\log |D|))$ be the $2n$-manifold defined by the total space of the elliptic cotangent bundle, with projection map $\pi:M\to L$.  Then $\pi^{*}D = (\pi^{*}R, \pi^{*}q)$ defines an elliptic divisor on $M$, and we have a tautological one-form $\Theta\in\Omega^{1}(M,\log |\pi^{*}D|)$ defined in the familiar way:
\begin{equation}\label{tautoneform}
\Theta_{\xi}(X) = \xi(\pi_{*} X), 
\end{equation}
for $\xi\in M$ and $X\in TM(-\log|\pi^{*}D|)$.
\begin{theorem}\label{canonicalsymp}
The derivative $\omega = d\Theta$ of the tautological 1-form~\eqref{tautoneform} on the total space of the elliptic cotangent bundle is an elliptic symplectic form with vanishing elliptic residue.  Furthermore, it satisfies $i_{E} \omega = \Theta$ for $E$ the Euler vector field, which is therefore Liouville in the sense
\begin{equation}
L_{E}\omega = \omega.
\end{equation}
\end{theorem}
\begin{proof}
Using coordinates on $L$ as in~\eqref{basisell}, we write any elliptic 1-form as 
\begin{equation}
\Theta = s\,  d\log r+ td\theta + \textstyle\sum_{i=3}^{n} p_{i}dx_{i},
\end{equation}
defining an extension of the coordinate system to $M$ and providing an explicit expression for the tautological form. Its derivative is then   
\begin{equation}\label{canelsympform}
d\Theta = ds\wedge d\log r + dt\wedge d\theta + \textstyle\sum_{i=3}^{n} dp_{i}\wedge dx_{i},
\end{equation}
showing that $\omega$ is nondegenerate and has zero elliptic residue.  Since the Euler vector field is $E = s\del_{s}+t\del_{t}+ \sum_{i} p_{i}\del_{p_{i}}$, we obtain $i_{E}\omega = \Theta$ directly from the local expression.
\end{proof}

We now show that the elliptic cotangent bundle construction is the universal example of a Lagrangian neighbourhood.
Let $\omega\in\Omega^{2}(M,\log |D|)$ be an elliptic symplectic form, and let $\iota: L\hookrightarrow M$ be a submanifold transverse to $D$, so that $D$ pulls back to define an elliptic divisor $D\cap L$ in $L$. We then have an induced inclusion map
\begin{equation}\label{loginc}
\iota_{*}:TL(-\log |D\cap L|)\to TM(-\log |D|),
\end{equation}
and we say that $L$ is \emph{Lagrangian} when $\iota^{*}\omega = 0$.

\begin{theorem}
Let $(M,D,\omega)$ be an elliptic symplectic manifold and $L$ a compact Lagrangian submanifold transverse to $D$.  Then a neighbourhood of $L$ in $M$ is isomorphic to a neighbourhood of the zero section in $T^{*}L(\log|L\cap D|)$.
\end{theorem}
\begin{proof}
The inclusion~\eqref{loginc} has cokernel given by the normal bundle of $L$, which is identified with $M_{0}=T^{*}L(\log|D\cap L|)$ by the elliptic symplectic form. We then choose an identification of a tubular neighbourhood $U$ of $L$ with a neighbourhood of the zero section in $M_{0}$, with the property that the elliptic divisors on $U$ and $M_{0}$ are identified.  Then the canonical form~\eqref{canelsympform} and the given form define a pair of elliptic symplectic forms $\omega_{0},\omega_{1}$ on $U$ such that $\iota^{*}\omega_{0}=\iota^{*}\omega_{1}=0$.  But this implies that $\omega_{0},\omega_{1}$ are cohomologous in $U$ since $(U,D)$ is smoothly homotopic to $(L,D\cap L)$. That is, $\omega_{1}-\omega_{0} = d\xi$ for some $\xi\in \Omega^{1}(U,\log |D|)$.  To produce a diffeomorphism taking $\omega_{0}$ to $\omega_{1}$, we apply the Moser argument: the interpolating family of symplectic forms $\omega_{t} = t\omega_{1} +(1-t)\omega_{0}$ satisfies $\tfrac{d}{dt}\omega_{t} = d\xi$, and so we obtain the required diffeomorphism by integrating the (elliptic, hence smooth) vector field $X_{t} = -\omega_{t}^{-1}(\xi)$.
Let $\pi:U\to L$ be the retraction. Since $\iota^{*}d\xi=0$, we may subtract $\pi^{*}\iota^{*}\xi$ from $\xi$ in order to ensure that $\xi$ is chosen such that $\iota^{*}\xi =0$. But then $\xi$ is conormal to $L$ away from $L\cap D$ and so the vector field $X_{t}$ is tangent to $L$. By compactness of $L$, we may integrate $X_{t}$ for unit time in a sufficiently small neighbourhood of $L$, yielding the result.    
\end{proof}

\begin{example}
Let $L$ be a 3-dimensional Lagrangian submanifold in a stable generalized complex 6-manifold $M$ which is transverse to the anticanonical divisor $D$.  Then $L\cap D$ defines a complex divisor $K\subset L$ with zero locus consisting of a link of embedded circles. 
A tubular neighbourhood of $L$ in $M$ is then isomorphic to the  canonical structure on the total space of $T^{*}L(\log K)$ provided by Theorem~\ref{canonicalsymp}.
\end{example}

The results of this section may be used to obtain a classification of a certain class of generalized complex branes, defined as follows (we simplify the definition given in~\cite{MR2811595} by ignoring the vector bundle over the submanifold).

\begin{definition}
Let $(M,H,\JJ)$ be a generalized complex manifold. A brane is a pair $(L,F)$ consisting of a submanifold $\iota:L\hookrightarrow M$ and a 2-form $F\in \Omega^{2}(L,\RR)$ such that $\iota^{*}H = dF$ and $\JJ\tau_{F} =\tau_{F}$, where $\tau_{F} = \{X+\xi\in TL\oplus T^{*}M\ |\ \iota^{*}\xi = i_{X}F\}$.  
\end{definition}
The bundle $\tau_{F}$ is an extension over $L$ of the form 
\begin{equation}
\xymatrix{
N^{*}L\ar[r] & \tau_{F}\ar[r] & TL
},
\end{equation}
and requiring that $\tau_{F}$ is $\JJ$-invariant implies that $Q(N^{*}L)\subset TL$, i.e., that $L$ is coisotropic for the underlying real Poisson structure $Q$.  In the stable case, therefore, there is a distinguished class of Lagrangian branes, essentially defined to be Lagrangian for the elliptic symplectic form away from the anticanonical divisor:
\begin{definition}
Let $(M,H,\JJ)$ be a stable generalized complex $2n$-manifold.  We call the brane $(L,F)$ Lagrangian when $L$ has dimension $n$ and is transverse to the anticanonical divisor of $\JJ$.
\end{definition} 
Certainly, any Lagrangian brane defines a Lagrangian submanifold for the elliptic symplectic structure, but the converse also holds.
\begin{proposition}
Let $\JJ$ be a stable generalized complex structure on $(M,H)$. Then any submanifold $L\subset M$ which is transverse to the anticanonical divisor $D$ and Lagrangian for the elliptic symplectic structure underlying $\JJ$ inherits a 2-form $F$ making it a generalized complex brane.
\end{proposition}
\begin{proof}
Since $L$ is transverse to $D$, we obtain an induced complex divisor $D_{L}$ on $L$ and an inclusion of logarithmic tangent bundles $\iota_{*}:TL(-\log D_{L})\to TM(-\log D)$.  Then let $\sigma$ be the complex log symplectic structure given by $\JJ$; since $L$ is Lagrangian, the imaginary part of $\iota^{*}\sigma$ vanishes, and by Proposition~\ref{isomzerores}, it is a smooth 2-form, i.e. $\iota^{*}\sigma = F\in\Omega^{2}(L,\RR)$.  The integrability condition $d\sigma = H$ then yields the required condition $\iota^{*}H = dF$. The condition $\JJ \tau_{F} = \tau_{F}$ then automatically holds since it is a closed condition which is gauge equivalent, away from $D$, to the condition that $L$ is Lagrangian in the usual sense.
\end{proof}

In this way, we have obtained a classification of neighbourhoods of compact Lagrangian branes in stable generalized complex manifolds, generalizing the following 4-dimensional result to all dimensions. 
\begin{example}[{\cite[Theorem 2.6]{MR2574746}}]
Let $(L,F)\subset (M,H,\JJ)$ be a compact Lagrangian brane in a stable generalized complex 4-manifold.  Then the generalized complex structure in a tubular neighbourhood of $L$ is completely determined by the (0-dimensional) complex divisor $D_{L} = D\cap L$, which itself is completely determined by orienting the tangent spaces $T_{p}L$ at each of the finitely many points $p\in D\cap L$. 
\end{example}

\section{Implications for deformation theory}\label{sec4}

In Section~\ref{equivloggc}, we showed that the deformation problem for stable generalized complex structures, with or without fixing the background 3-form, is unobstructed.  These deformation problems are governed by a differential graded Lie algebra and a $L_{\infty}$-algebra, respectively, and the unobstructedness suggests that these algebras are formal, i.e. quasi-isomorphic to a differential graded Lie algebra with vanishing bracket.  In this section we will establish the basic results which verify this formality.  

\subsection{Real Poisson deformations}\label{realpoisdefs}

We begin with a motivating example from real Poisson geometry. Let $\pi$ be a real Poisson structure on the smooth manifold $M$.   The deformations of $\pi$ are governed by the Lichnerowicz differential graded Lie algebra of multivector fields $\der{M}$, with Lie bracket given by the Schouten bracket and differential given by $d_{\pi}=[\pi,\cdot]$.  Specifically, a deformation of the Poisson structure is given by a tensor $\eps\in \der[2]{M}$ such that 
\begin{equation}\label{poismc}
d_{\pi}\eps + \tfrac{1}{2} [\eps,\eps]=0.
\end{equation}
This deformation theory has the special feature that it receives a morphism of differential graded Lie algebras from the usual de Rham complex $\Omega^{\bullet}(M)$, equipped with the Koszul bracket $[\cdot,\cdot]_{\pi}$, extending the Poisson bracket on functions.  The morphism is given by the exterior powers of $\pi$:
\begin{equation}\label{wedgepi}
\xymatrix{
(\Omega^{\bullet}(M), d, [\cdot,\cdot]_{\pi})\ar[r]^-{\wedge^{\bullet}\pi} & (\der{M}, d_{\pi}, [\cdot,\cdot])
}
\end{equation}
Importantly, this Koszul algebra is formal, in that it receives an $L_{\infty}$ quasi-isomorphism $f=(f_{1},f_{2},\ldots)$ from the de Rham complex with trivial bracket. We refer to~\cite{MR3007085} for a careful description of this $L_{\infty}$ morphism, and describe only its action on Maurer-Cartan elements: if $B\in \Omega^{2}(M)$ is a closed 2-form, then according to~\cite{MR3007085}, the $k$-linear component of the $L_{\infty}$ morphism is
\begin{equation}
f_{k}(B,\cdots ,B) = k! \cdot B(\pi B)^{k-1},
\end{equation}
where we view $B$ and $\pi$ as maps between $TM$ and $T^{*}M$ and the expression above is their successive composition.
Therefore, the Maurer-Cartan element $B$ in the abelian algebra is sent to the 2-form
\begin{equation}\label{linftymcmap}
f_{*}(B) = \sum_{k=1}^{\infty}\tfrac{1}{k!}f_{k}(B,\ldots, B) = B + B\pi B + B\pi B \pi B + \cdots
\end{equation}
which converges to $B(1-\pi B)^{-1}$ when the endomorphism $(1-\pi B)$ has an inverse.
This then satisfies the Maurer-Cartan equation for the Koszul bracket.  Finally, we apply the morphism~\eqref{wedgepi} to obtain
\begin{equation}
\eps = \wedge^{2}\pi(f_{*}(B)) = \pi B \pi + \pi B \pi B \pi + \cdots,
\end{equation}
which satisfies~\eqref{poismc}, giving a deformation of $\pi$.  While this may seem surprising from an algebraic point of view, it is well-understood in Poisson geometry (see~\cite{MR2023853}) and consists of the $B$-field symmetry~\eqref{bsym} applied to the graph of the Poisson structure $\pi$ in $\T M$. Indeed, if we let $\Gamma_{\pi}=\{\pi(\xi) + \xi\ |\  \xi\in T^{*}M\}$, then when $(1-\pi B)$ is invertible, we have the identity
\begin{equation}\label{bfieldgamma}
e^{-B}\Gamma_{\pi} = \Gamma_{(1-\pi B)^{-1}\pi}.
\end{equation}
From a more pedestrian standpoint, one simply deforms $\pi$ by subtracting the pullback of $B$ from the symplectic form on each symplectic leaf of $\pi$.


If $\pi$ is invertible, then the morphism~\eqref{wedgepi} is an isomorphism, and so all deformations of $\pi$ may be obtained by the $B$-field transform~\eqref{bfieldgamma}, reflecting the fact that the symplectic form $\pi^{-1}$ is deformed by adding closed 2-forms.  When $\pi$ does degenerate, however, we cannot expect to capture all deformations of $\pi$ in this way. On the other hand, in the interesting case of real log symplectic structures as developed by~\cite{MR3250302}, this failure can be repaired.

A real log symplectic structure is a Poisson structure $\pi$ on a real $2n$-manifold $M$ such that $\wedge^{n}\pi$ vanishes transverseley.  The zero locus of $\wedge^{n}\pi$ is a real codimension 1 submanifold $Z$, the real analogue of the complex divisors studied in Section~\ref{sectcxdiv}.  Similarly to the complex case, the sheaf of vector fields tangent to $Z$ forms a Lie algebroid $T(-\log Z)$.  The Poisson structure $\pi$ is then a section of $\wedge^{2}(T(-\log Z))$, and so induces a Koszul bracket on the logarithmic de Rham complex $\Omega^\bullet(M,\log Z)$ as before. Furthermore, as a map from $T(-\log Z)$ to $\Omega^{1}(M,\log Z)$, $\pi$ is invertible to a nondegenerate section of $\wedge^{2}(T(-\log Z))^{*}$, i.e., a logarithmic symplectic form. 

As observed in~\cite{MR3245143}, we may factor the morphism~\eqref{wedgepi} through the inclusion of forms into logarithmic forms, and the resulting morphism
\begin{equation}\label{wedgepilog}
\xymatrix{
(\Omega^{\bullet}(M,\log Z), d, [\cdot,\cdot]_{\pi})\ar[r]^-{\wedge^{\bullet}\pi} & (\der{M}, d_{\pi}, [\cdot,\cdot])
}
\end{equation}
is a quasi-isomorphism. The $L_{\infty}$ quasi-isomophism used in Section~\ref{realpoisdefs} then also factors through the inclusion of forms into logarithmic forms, inducing the same map on Maurer-Cartan elements~\eqref{linftymcmap}, but for $B\in\Omega^{2}(M,\log Z)$.  Composing these two quasi-isomorphisms, we render the deformation complex formal, leading to the identification made in~\cite{MR3245143} of the local moduli space of log symplectic structures with an open neighbourhood of $[\pi^{-1}]$ in the second logarithmic cohomology group, which by the theorem of Mazzeo-Melrose~\cite{MR1348401} is canonically isomorphic to $H^{2}(M,\RR)\oplus H^{1}(Z,\RR)$.

\subsection{Stable generalized complex deformations}

In~\cite{MR2811595}, the deformation theory for generalized complex structures is described as follows.  Since $L,\ol L$ are involutive subbundles in duality, we obtain a graded Lie bracket on the algebroid de Rham complex $(\Omega^\bullet_L, d_L)$, where
\begin{equation}
\Omega^\bullet_L(M) = C^\infty(M,\wedge^k L^*),
\end{equation}
rendering it a differential graded Lie algebra.
A deformation of the generalized complex structure $\JJ$ is an element $\eps$ in $\Omega^2_L(M)$ satisfying the Maurer-Cartan equation
\begin{equation}
d_L\eps + \tfrac{1}{2}[\eps,\eps]=0.
\end{equation}
If we set up a local deformation problem in which generalized complex structures are equivalent if they are related by diffeomorphisms connected the identity and exact B-field transforms, we obtain that the tangent space to the local Kuranishi moduli space is given by the Lie algebroid cohomology group $H^2_L(M)$.

For a stable generalized complex structure with anticanonical divisor $D = (K^{*},s)$, we saw in Section~\ref{equivloggc} that the anchor map of $L$ factors through the morphism $T(-\log D)\to T_{\CC}$, and so we have the morphism $\wt\anc: L\to T(-\log D)$,
which is an inclusion on sheaves of sections.  Motivated by the equivalence of stable generalized complex structures and complex log symplectic forms, we expect the dual sheaf inclusion
\begin{equation}\label{qilcomplex}
\wt\anc^{*}:\Omega^{\bullet}(\log D)\to \Omega^{\bullet}_{L}
\end{equation}
to be a quasi-isomorphism, just as in the previous case.
\begin{proposition}
The pullback $\wt\anc^{*}$ of logarithmic forms for the anticanonical divisor to the algebroid de Rham complex $\Omega^{\bullet}_{L}$ of a stable generalized complex manifold is a quasi-isomorphism, yielding the isomorphism
\begin{equation}\label{lcohomology}
H^{k}_{L}(M) \cong H^{k}(M\backslash D,\CC).
\end{equation}
\end{proposition} 
\begin{proof}
Working in the local Darboux coordinates of Section~\ref{darbo}, we have that $L^{*}$ is locally generated by $w^{-1}dw, w^{-1}dz, d\ol z, d\ol w,$ as well as the remaining real symplectic generators $dx_{i}, dp_{i}, i=1,\ldots, m-2$, whereas 
$\Omega^{1}(\log D)$ has identical generators except $w^{-1}dz$ is replaced with $dz$. 

Viewing each complex as a dg module over the de Rham complex generated by the $dx_{i}, dp_{i}$, we may apply the Poincar\'e lemma to reduce the problem to the case $m=2$, that is, for $\CC^{2}$ with coordinates $w,z$.  Then, viewing each complex as a dg module over the Dolbeault complex generated by $d\ol w, d\ol z$, we apply the Dolbeault lemma to reduce the problem to proving a quasi-isomorphism only for sheaves of holomorphic sections. 

We now compute the cohomology groups $H^{k}_{L}$ explicitly on $\CC^{2}$ in the holomorphic category. For $f\in \Omega^{0}_{L}$, 
we have
\begin{equation}\label{eq:dpif}
df = (w\del_{w}f) w^{-1}dw +(w\del_{z}f) w^{-1}dz,
\end{equation}
which vanishes if and only if $f$ is constant, proving $H^{0}_{L}=\CC$, with the same generator as $H^{0}(\log D)$.  

Now we compute  $H_{L}^1$. For $\alpha = f_{1} w^{-1}dw + f_{2}w^{-1}dz$, we may write $\alpha = w^{-1} a$ for $a$ a smooth 1-form, and  
\begin{equation}\label{eq:dpiX}
d\alpha = -w^{-2}dw \wedge a + w^{-1}da,
\end{equation}
so that $\alpha$ is closed if and only if $da = w^{-1}dw \wedge a$.  For $a$ smooth this is possible only if $a = fdw + wa'$ for $a'$ a smooth 1-form, which can be taken to satisfy $i_{\del_{w}}a' =0$.  Differentiating, we obtain
$df\wedge dw + wda' = 0$,
implying that $f$ is constant and $a' = dg$, so that $\alpha = fw^{-1}dw + dg$, proving that $H^{1}_{L}$ has rank 1 and is generated by the class of $w^{-1}dw$, also the generator for $H^{1}(\log D )$.

Finally, we compute $H^2_{L}$.  Given any $\sigma = f w^{-1}dw \wedge w^{-1}dz$, write $f = f_{0} + wf_{1}$ for $f_{0}, f_{1}$ holomorphic functions with $\del_{z}f_{0}=0$.  Then $\sigma$ is exact:
\begin{equation}
\sigma  =  d( - f_{0} w^{-1}dz + (\textstyle\int^{z}\!f_{1}) w^{-1}dw      ).
\end{equation}

This verifies that $\wt\anc^{*}$ is an isomorphism on local cohomology groups, proving it is a quasi-isomorphism. The final isomorphism~\eqref{lcohomology} then follows from Theorem~\ref{logdcoh}.


\end{proof}

As described in Section~\ref{equivloggc}, the real Poisson structure underlying the generalized complex structure defines an elliptic log symplectic structure $\omega$, or equivalently, the imaginary elliptic Poisson structure $\pi = i\omega^{-1}/2$, which we may view as a section of $\wedge^{2}T_{\CC}(-\log|D|)$.  This defines a Koszul bracket on the elliptic forms $[\cdot,\cdot]_{\pi}$, which is induced on the logarithmic forms by the inclusion of $\Omega^{\bullet}(M,\log D)$ in $\Omega^{\bullet}_{\CC}(M,\log|D|)$.  
As in the real log symplectic case~\eqref{wedgepilog}, the quasi-isomorphism of complexes~\eqref{qilcomplex} defines a homomorphism of differential graded Lie algebras from the logarithmic de Rham complex with the above Koszul bracket to the deformation complex of the generalized complex structure. 
Also, by the results of~\cite{MR3007085}, there is a formality $L_{\infty}$ quasi-isomorphism $f = (f_{1}, f_{2}, \ldots)$ from the logarithmic de Rham complex to itself, taking the zero bracket to the Koszul bracket.  Summarizing, we have the following morphisms 
\begin{equation}\label{homllogdgla}
\xymatrix@C=2ex{
(\Omega^{\bullet}(M,\log D), d, 0 )\ar[r]^-{f} & (\Omega^{\bullet}(M,\log D), d, [\cdot,\cdot]_{\pi})\ar[r]^-{\wt\anc^{*}} & (\Omega^{\bullet}_{L}(M), d_{L}, [\cdot,\cdot])
},
\end{equation}
whose composition provides the formality map for the deformation complex of stable generalized complex structures.
\begin{proposition}
The formality map given by the composition of the quasi-isomorphisms~\eqref{homllogdgla} takes a closed 2-form $\beta\in\Omega^{2}(\log D)$ to the Maurer-Cartan element 
\begin{equation}
\wt\anc^{*}(f(\beta)) = \wt\anc^{*}(\beta + \beta\pi\beta + \beta\pi\beta\pi\beta + \cdots),
\end{equation}
which converges to $\wt\anc^{*}(\beta(1-\pi\beta)^{-1})$ when $1-\pi\beta$ is an invertible endomorphism of $T_{\CC}M$. 
\end{proposition}
\begin{proof}
We present a direct geometric calculation of the map on Maurer-Cartan elements which circumvents but coincides with the general algebraic method of~\cite{MR3007085}.  Let $\sigma\in\Omega^{2}(M,\log D)$ be the complex log symplectic form corresponding to $\JJ$.
Then $L = \{X + \sigma X\ |\ X\in T_{\CC}M\}$ and its deformation by the closed logarithmic 2-form $\beta$ is $L_{\beta} = \{X + \sigma X + \beta X\ |\  X\in T_{\CC}M\}$. Writing $L_{\beta}$ as the graph of a skew map from $L$ to $\ol L\cong L^{*}$, and expressing this map as the pullback $\wt\anc^{*}\delta$ of $\delta\in\Omega^{2}(\log D)$, we obtain that for every $X\in T_{\CC}M$ there exists $Y\in T_{\CC}M$ such that 
\begin{equation}
X + \sigma X + \beta X = Y + \sigma Y  + (\wt\anc^{*}\delta)(Y + \sigma Y).
\end{equation}
The map $\wt\anc^{*}:T^{*}_{\CC}M\to \ol L$ may be written $\wt\anc^{*}\eta = \pi\eta + \ol \sigma \pi\eta$, and so we have 
\begin{equation}
X + \sigma X + \beta X = Y + \sigma Y  + \pi\delta Y + \ol\sigma\pi\delta Y.
\end{equation}
Equating tangent and cotangent components, we obtain 
$(\sigma + \beta)(1+\pi\delta) = \sigma + \ol\sigma\pi\delta$,
and solving for $\delta$, we obtain $\delta = \beta(1-\pi\beta)^{-1}$, as required.
\end{proof}

\subsection{Deformations with varying background 3-form}

In the previous section we studied the deformation theory of a generalized complex structure $\JJ$ on the pair $(M,H)$, where $H$ was a fixed real closed 3-form.  We now consider the simultaneous deformations of $(\JJ,H)$: a deformation is given by $\eps\in \Omega^{2}_{L}(M)$ as before, together with a 3-form variation $\eta\in \Omega^{3}(M,\RR)$.  The integrability condition for the deformed structure may be obtained as follows: first the deformed twisted differential $d^{H+\eta} = d + (H+\eta)\wedge\cdot$ must square to zero, and second, we require that the deformed canonical bundle $e^{\eps}K$ satisfies the integrability condition~\eqref{modfield} with respect to the deformed differential $d^{H+\eta}$. 

The first condition holds if and only if $d\eta = 0$, and the second may be phrased in terms of the alternative grading~\eqref{fockspace}:  the operator 
\begin{equation}\label{integsimdef}
e^{-\eps}d^{H+\eta}e^{\eps} = e^{-\ad[\eps]}(d^{H} +\eta\wedge\cdot)
\end{equation}
must have only degree +1 and -1 components, which holds if and only if the degree $3$ component vanishes. Since we are deforming away from an integrable structure, $d^{H}$ decomposes into degree $-1$ and $+1$ components, namely $d^{H}= \del + \delbar$, and we may decompose $\eta$ into its components according to the decomposition of $\wedge^{3}\T M$ into summands $\wedge^{i}L\otimes \wedge^{j}\ol L$ for $i+j=3$ which act by the Clifford action with degree $j-i$:
\begin{equation}
\eta = \eta^{3,0} + \eta^{2,1} + \eta^{1,2} + \eta^{0,3}.
\end{equation}
Finally, as a section of $\wedge^{2}\ol L$, $\eps$ acts with degree $+2$.
Expanding~\eqref{integsimdef} and extracting the degree $+3$ component, we obtain the sum of commutators
\begin{equation}\label{triplebracket}
\eta^{0,3} - [\eps, \delbar + \eta^{1,2}] + \tfrac{1}{2}[\eps,[\eps,\del + \eta^{2,1}]] - \tfrac{1}{3!} [\eps,[\eps,[\eps,\eta^{3,0}]]].
\end{equation}
The linear part of the above sum is simply $\eta^{0,3} - [\eps,\delbar]$, which may be written in terms of the anchor and de Rham operator for $L$ as $\anc^{*}\eta + d_{L}\eps$. Together with the earlier linear condition $d\eta =0$, this suggests that deformations are governed by the double complex
\begin{equation}\label{dblcxgch}
\begin{aligned}
\xymatrix@R=2.2ex{
\ar[r]&\Omega^{k-1}_L\ar[r]^-{d_{L}} & \Omega^{k}_L\ar[r]^-{d_{L}} & \Omega^{k+1}_L\ar[r]&\\
\ar[r]&\Omega^{k-1}_M\ar[r]^-{d}\ar[u]_{\anc^*} & \Omega^{k}_M\ar[r]^-{d}\ar[u]_{\anc^*} & \Omega^{k+1}_M\ar[u]_{\anc^*}\ar[r]&
}
\end{aligned}
\end{equation}
where the bottom row is the de Rham complex with real coefficients.
In fact, the vanishing of ~\eqref{triplebracket} and the requirement $d\eta=0$ are the components of the Maurer-Cartan equation for an $L_{\infty}$ structure on the total complex of the above double complex.  The $L_{\infty}$ structure may be obtained via the derived bracket construction of homotopy algebras~\cite{MR2163405}, as implemented by~\cite{Fregier:2012uq} for the study of simultaneous deformations. 

In the case of stable generalized complex structures, 
this double complex receives a quasi-isomorphism from the double complex governing deformations of the pair $(\sigma, H)$ consisting of a complex log symplectic structure integrable with respect to the real closed 3-form $H$.  The latter complex is as follows:
\begin{equation}\label{formaldbl}
\begin{aligned}
\xymatrix@R=2.2ex{
\ar[r]&\Omega^{k-1}_{\log D}\ar[r]^-{d} & \Omega^{k}_{\log D}\ar[r]^-{d} & \Omega^{k+1}_{\log D}\ar[r]&\\
\ar[r]&\Omega^{k-1}_M\ar[r]^-{d}\ar[u] & \Omega^{k}_M\ar[r]^-{d}\ar[u] & \Omega^{k+1}_M\ar[u]\ar[r]&
}
\end{aligned}
\end{equation}
and the quasi-isomorphism is given by combining $\wt\anc^{*}$ on the top row with the identity map on the bottom row.  But the latter deformation problem is formal, with Maurer-Cartan equation for an element $(\beta,\eta)\in \Omega^{2}(\log D)\oplus \Omega^{3}(M,\RR)$ being
\begin{equation}
d\beta + \eta = 0.
\end{equation}
This leads us to the following conjecture, which should be tractable by combining the results of~\cite{MR3007085} and~\cite{Fregier:2013fk}:
\begin{conjecture}\label{conjlinfty}
There is a canonical $L_{\infty}$ morphism from the total complex $\Omega^{\bullet}_{\JJ,H}$ of~\eqref{formaldbl} equipped with the zero bracket to the $L_{\infty}$ algebra structure on the total complex of~\eqref{dblcxgch}, extending the given quasi-isomorphism and for which the map on Maurer-Cartan elements is given by the map
\begin{equation}
(\beta,\eta)\mapsto (\wt\anc^{*}(\beta(1-\pi\beta)^{-1}), \eta).
\end{equation}
\end{conjecture}

Using only the above quasi-isomorphism, we may go further:
because of the short exact sequence of complexes~\eqref{complogell}, the total complex of~\eqref{formaldbl} is quasi-isomorphic to the complex of elliptic forms with zero elliptic residue.  As a consequence of Proposition~\ref{qilcomplex}, therefore, we obtain an explicit topological expression for the cohomology of the $L_{\infty}$ algebra controlling deformations of the pair $(\JJ,H)$ of a stable generalized complex structure with background closed 3-form $H$:
\begin{corollary}
The total complex $\Omega^{\bullet}_{\JJ,H}$ of~\eqref{dblcxgch} which controls the simultaneous deformations of stable generalized complex structures and background 3-forms is quasi-isomorphic to the complex of elliptic forms with zero elliptic residue, and its cohomology groups are given by 
\begin{equation}
H^{k}_{\JJ,H} = H^{k}(M\backslash D,\RR)\oplus H^{k-1}(D,\RR).
\end{equation}
 \end{corollary}

\bibliographystyle{hyperamsplain} \bibliography{references}

\end{document}